\documentclass[11pt,letterpaper]{amsart}
\usepackage{amsmath}
\usepackage{amsthm}
\usepackage{amssymb}
\usepackage{amsfonts,mathrsfs}
\usepackage{stmaryrd}
\usepackage{amsxtra}  
\usepackage{epsfig}
\usepackage{verbatim}
\usepackage[all]{xy}

\usepackage{dsfont}
\usepackage{color}
\usepackage{mathtools}
\usepackage{tikz}
\usepackage{tikz-cd}
\usetikzlibrary{fpu}
\usetikzlibrary{matrix}

\theoremstyle{plain}
\newtheorem{theorem}[equation]{Theorem}
\newtheorem{proposition}[equation]{Proposition}
\newtheorem{lemma}[equation]{Lemma}
\newtheorem{corollary}[equation]{Corollary}
\newtheorem{conjecture}[equation]{Conjecture}

\usepackage{bm}

\theoremstyle{remark}
\newtheorem{remark}[equation]{Remark}
\newtheorem{definition}[equation]{Definition}
\numberwithin{equation}{section}

\newcommand{\dee}{\partial}

\newcommand{\w}{\wedge}

\usepackage{enumerate}
\newenvironment{enum}{  
\begin{enumerate}[\upshape(\arabic{section}.\arabic{equation}a)] }  
{  \end{enumerate}   }
\newcommand{\itemref}[2] {{\upshape(\ref{#1}\ref{#2})}}

\DeclareMathOperator{\Span}{Span}
\renewcommand{\hat}{\widehat}

\renewcommand{\bar}{\overline}
\newcommand{\chih}{\hat\chi}
\newcommand{\wt}{\widetilde}

\usepackage{scalerel,stackengine}
\stackMath
\newcommand\reallywidehat[1]{%
\savestack{\tmpbox}{\stretchto{%
  \scaleto{%
    \scalerel*[\widthof{\ensuremath{#1}}]{\kern.1pt\mathchar"0362\kern.1pt}%
    {\rule{0ex}{\textheight}}
  }{\textheight}%
}{2.4ex}}%
\stackon[-6.9pt]{#1}{\tmpbox}%
}
\newcommand{\im}{\text{Im}}

\DeclareMathOperator{\re}{Re}

\def\norm#1{\left\Vert#1\right\Vert}
\def\normm#1{\Vert#1\Vert}
\def\Gammaf#1{\Gamma\left(#1\right)}

\newcommand{\cf}{{\mathcal F}}

\newcommand{\ci}{{\mathcal I}}
\newcommand{\cj}{{\mathcal J}}

\newcommand{\co}{{\mathcal O}}

\newcommand{\cs}{{\mathcal S}}

\newcommand{\sB}{{\mathscr B}}

\newcommand{\sF}{{\mathscr F}}

\newcommand{\C}{{\mathbb C}}

\newcommand{\p}{{\mathbb P}}

\newcommand{\R}{{\mathbb R}}

\newcommand{\Z}{{\mathbb Z}}

\begin{document}

\title[Leray Transform]{High frequency behavior of the Leray transform: model hypersurfaces and projective duality }
\author{David E. Barrett \& Luke D. Edholm}
\begin{abstract}
The Leray transform $\bm{L}$ is studied on a family $M_\gamma$ of unbounded hypersurfaces in two complex dimensions.  For a large class of measures, we obtain necessary and sufficient conditions for the $L^2$-boundedness of $\bm{L}$, along with an exact spectral description of $\bm{L}^*\bm{L}$.  This yields both the norm and high-frequency norm of $\bm{L}$, the latter affirming an unbounded analogue of an open conjecture relating the essential norm of $\bm{L}$ to a projective invariant on a bounded hypersurface.  $\bm{L}$ is also shown to play a central role in bridging the function theoretic and projective geometric notions of duality.  Our work leads to the construction of projectively invariant Hardy spaces on the $M_\gamma$, along with the realization of their duals as invariant Hardy spaces on the dual hypersurfaces.
\end{abstract}
\thanks{The first author was supported in part by NSF grant number DMS-1500142}
\thanks{The second author was supported in part by Austrian Science Fund (FWF): AI0455721}
\address{Department of Mathematics\\University of Michigan, Ann Arbor, MI, USA}
\address{Department of Mathematics\\Universit\"at Wien, Vienna, Austria}
\email{barrett@umich.edu}
\email{luke.david.edholm@univie.ac.at}

\maketitle

\section{Introduction}\label{S:Introduction}

This article continues a series aimed at further developing the theory of the Leray transform from a projective dual point of view.  Much of our focus here will be on the following family of real hypersurfaces.  For $ \gamma \ge 1$, define
\begin{equation}\label{D:M_gamma}
M_{\gamma} := \left\{ (\zeta_1,\zeta_2)\in \C^2\colon \im(\zeta_2) = \left|\zeta_1\right|^\gamma\right\},
\end{equation}
together with the unbounded domain lying on its $\C$-convex side
\begin{equation}\label{D:Omega_gamma}
\Omega_{\gamma} := \left\{ (z_1,z_2)\in \C^2\colon \im(z_2) > \left|z_1\right|^\gamma\right\}.
\end{equation}
\indent The Leray transform is a higher-dimensional analogue of the Cauchy transform of a planar domain; it acts by taking in data given on a real hypersurface to construct holomorphic functions on the interior of the hypersurface.  (See Section \ref{SS:LerayTransform} for precise definitions.)  As is typical in multi-dimensional complex analysis constuctions, convexity conditions play a crucial role.  In particular, the Leray transform $\bm{L}_{\cs}:= \bm{L}$ is a well-defined integral operator on any smoothly bounded $\C$-convex hypersurface $\cs$ in $\C^n$.  When $\cs$ is an unbounded hypersurface (such as $M_\gamma$), additional care must be taken.  As in the case of the Cauchy transform, knowledge of both quantitative and qualitative information related to $\bm{L}$ provides insight into holomorphic function spaces associated to $\cs$.\\
\indent A great deal of additional information can be obtained by viewing the Leray transform through the lens of projective duality. The projective dual of a hypersurface $\cs \subset \C\p^n$, denoted $\cs^*$, is the set of complex hyperplanes tangent to $\cs$.  In \cite{Bar16}, the first author demonstrates that the efficiency of a natural pairing of two dual Hardy spaces associated respectively to $\cs$ and $\cs^*$ is measured by the $L^2$-norm of $\bm{L}_\cs$.    The Cauchy transform plays an analogous role in the pairing of interior and exterior Hardy spaces associated to a planar curve; see \cite{Meyer82}.  New aspects of this theory are set forth in Section \ref{S:Duality}, where they are carefully illustrated in the setting of $M_\gamma$.\\
\indent Each $M_\gamma$ is homogeneous with respect to certain projective automorphisms (see Section \ref{SS:Symmetries}) and $\bm{L}$ admits a transformation law with respect to such maps; see \eqref{E:LerayProjTrans}.  It thus makes sense to pay special attention to $L^2$-spaces on $M_\gamma$ which are built from measures satisfying desirable transformation laws.  Our analysis begins by considering the measure $\sigma := \alpha^{\gamma-1} d\alpha \wedge d\theta \wedge ds$, with $\alpha=|\zeta_1|,\theta=\arg(\zeta_1), s=\re(\zeta_2)$.  This is a constant multiple of the Leray-Levi measure corresponding to the most natural choice of defining function for $M_\gamma$; see Section \ref{SS:LerayTransform}.  

We now highlight several results established in this paper.  Continuing the theme of \cite{BarEdh17}, the exact $L^2(M_\gamma,\sigma)$-norm of the Leray transform is obtained:

\begin{theorem}\label{T:Norm}
Let $\bm{L}$ be the Leray transform of $M_\gamma$.  Then $\bm{L}\colon L^2(M_\gamma,\sigma) \to L^2(M_\gamma,\sigma)$ is a bounded projection operator with norm
\begin{equation*}\label{E:Norm}
\norm{\bm{L}}_{L^2(M_\gamma,\sigma)} = \frac{\gamma}{2\sqrt{\gamma-1}}.
\end{equation*}
\end{theorem}

We go on to consider a more general class of measures: for $r \in \R$, let $\mu_r = \alpha^r d\alpha \wedge d\theta \wedge ds$.  Then we have

\begin{theorem}\label{T:boundedness-mu_r}
The Leray transform $\bm{L}$ is a bounded projection from $L^2(M_\gamma,\mu_r) \to L^2(M_\gamma,\mu_r)$ if and only if $r \in \ci_0 := (-1,2\gamma-1)$.
\end{theorem}

Theorem \ref{T:Norm} is proved in Section \ref{SS:proof-of-norm-formula}, Theorem \ref{T:boundedness-mu_r} in Section \ref{SS:Leray-boundedness-m_r}.  The analysis involved in proving these results relies heavily on the symmetries of $M_\gamma$.  In particular, the $S^1$ action in the $\zeta_1$ variable of \eqref{D:M_gamma} yields a (partial) Fourier series decomposition of the space 
\begin{equation*}
L^2(M_\gamma,\mu_r) = \bigoplus_{k=-\infty}^\infty L^2_k(M_\gamma,\mu_r).
\end{equation*}
(Here we use the ``subspace" notion of direct sum as set forth, for example, on page 81 of \cite{AxlBouRam}.)
The Leray transform decomposes similarly, but a holomorphic function on $M_\gamma$ may have nonzero components in $L^2_k(M_\gamma,\mu_r)$ only if $k\ge0$.  In other words,
\begin{equation}\label{E:Leray-decomp}
\bm{L} = \bigoplus_{k=0}^\infty \bm{L}_k,
\end{equation}
where the {\em sub-Leray} operator $\bm{L}_k$ is nonzero only when acting on $L^2_k(M_\gamma,\mu_r)$.  Very precise information on each $\bm{L}_k$  is obtained, including both the sharp range of $r$ for which boundedness in $L^2(M_\gamma,\mu_r)$ holds and the exact operator norm when it does.  Define the interval
\begin{equation}\label{E:def-of-interval-I_k-intro}
\ci_k = (-2k-1 , (2k+2)(\gamma-1)+1),
\end{equation}
along with the $\mu_r$-{\em symbol function}
\begin{equation*}
C_{\mu_r}(\gamma,k) = \frac{\Gamma\big(\frac{2k+1+r}{\gamma}\big) \Gamma\big(2k+2 - \frac{2k+1+r}{\gamma}\big)}{\Gamma(k+1)^2} \left( \tfrac{\gamma}{2}\right)^{2k+2} (\gamma-1)^{-\left( 2k+2 - \frac{2k+1+r}{\gamma} \right)}.
\end{equation*}

\begin{theorem}\label{T:Lk-boundedness}
Let $k \ge 0$ be an integer and $r\in \R$.  Then $\bm{L}_k$ is bounded from $L^2_k(M_\gamma,\mu_r) \to L^2_k(M_\gamma,\mu_r)$ if and only if $r \in \ci_k$.  Furthermore when $r \in \ci_k$,
\begin{equation*}
\norm{\bm{L}_k}_{L^2(M_\gamma,\mu_r)} = \sqrt{C_{\mu_r}(\gamma,k)}.
\end{equation*}
\end{theorem}

When $r \in \ci_0$ (the interval of boundedness from Theorem \ref{T:boundedness-mu_r}), the Leray transform admits a bounded adjoint $\bm{L}^{(*,\mu_r)}$ in the $L^2(M_\gamma,\mu_r)$-inner product.  For these $r$ values, we obtain a complete spectral description of the self-adjoint $\bm{L}^{(*,\mu_r)}\bm{L}$ and $\bm{L}\bm{L}^{(*,\mu_r)}$, and anti-self-adjoint $\bm{A}^{\mu_r} := \bm{L}^{(*,\mu_r)}-\bm{L}$ in terms of the symbol function.  In Section \ref{SS:related-operators} it is shown that:

\begin{theorem}\label{T:spectrum-of-L*L}
Let $r \in \ci_0$ and let $\bm{T}$ denote either $\bm{L}^{(*,\mu_r)}\bm{L}$ or $\bm{L}\bm{L}^{(*,\mu_r)}$.  Then $\bm{T}$ admits an orthogonal basis of eigenfunctions and its spectrum is given by
\begin{equation*}
\{0\} \cup \left\{ C_{\mu_r}(\gamma,k): k=0,1,2, \dots \right\}.
\end{equation*}
\end{theorem}

\begin{theorem}\label{T:spectrum-of-L*-L}
Let $r \in \ci_0$.  The operator $\bm{A}^{\mu_r} := \bm{L}^{(*,\mu_r)}-\bm{L}$ admits an orthogonal basis of eigenfunctions and its spectrum is given by
\begin{equation*}
\{0\} \cup \left\{ \pm i\sqrt{C_{\mu_r}(\gamma,k)-1}: k=0,1,2, \dots \right\}.
\end{equation*}
\end{theorem}
Theorems \ref{T:spectrum-of-L*L} and \ref{T:spectrum-of-L*-L} follow from work carried out in Section \ref{SS:related-operators}.

We now emphasize a fascinating observation about the limiting behavior of the sub-Leray operators.  Define the {\em ($\mu_r$)-high-frequency limit norm} by
\begin{equation}\label{E:HighFreqLimSup}
\norm{\bm{L}}_{L_{HF}^2(M_\gamma,\mu_r)} := \limsup_{k\to\infty} \norm{\bm{L}_k}_{L^2(M_\gamma,\mu_r)}.
\end{equation}
This is definable for any operator admitting a Fourier series decomposition and it can be viewed as a generalized essential norm in the sense of Lef\`evre \cite{Lefevre09} -- in Section \ref{SS:GenEss} we introduce the terminology {\em grade-essential-norm}.

As the value of $r$ varies, the norm $\norm{\bm{L}_k}_{L^2(M_\gamma,\mu_r)}$ expectedly changes.  But in Section \ref{SS:Leray-boundedness-m_r} it is shown that all choices of $r$ yield the same high frequency limit:

\begin{theorem}\label{T:HighFreqNorm}
For $r \in \R$, the $\mu_r$-high-frequency limit norm of $\bm{L}$ is
\begin{equation*}\label{E:HighFreqNorm}
\norm{\bm{L}}_{L_{HF}^2(M_\gamma,\mu_r)} = \sqrt{\frac{\gamma}{2\sqrt{\gamma-1}}}.
\end{equation*}
\end{theorem}

An analogous result was observed in \cite{BarLan09} on smoothly bounded, strongly $\C$-convex Reinhardt hypersurfaces in $\C^2$.  We'll see in Section \ref{SS:EssNormConj} that this quantity is closely connected to a projective geometric invariant of $M_\gamma$.  Note also that this quantity is the square-root of the $L^2(M_\gamma, \sigma)$-norm in Theorem \ref{T:Norm}; a similar observation was recently noted for the Leray transform on boundaries of $L^p$-balls $\{(\zeta_1,\zeta_2) \in \C^2 :|\zeta_1|^p + |\zeta_2|^p = 1 \}$;  see \cite{ShelahThesis}.  

Lanzani and Stein have written a series of recent articles on the Leray transform in various settings.  One of the takeaways in \cite{LanSte13, LanSte14} is that $\bm{L}$ is well behaved (one can expect both $L^2$ and $L^p$ boundedness results) as long as the hypersurface in question is both $\bm{(1)}$ strongly $\C$-convex and $\bm{(2)}$ $C^{1,1}$-smooth.  When either hypothesis is dropped, they are able to construct elementary counter-examples with no Leray $L^p$-boundedness (including $p=2$); see \cite{LanSte17c, LanSte17b}.  The positive Lanzani-Stein results don't apply in our setting since $M_\gamma$ (viewed in projective space) fails even to be $C^1$ when it meets the line at infinity (except when $\gamma=2$); see Section \ref{SS:MgammaAtInfinity}.  But unlike their counter-examples, our results show that $\bm{L}$ is $L^2$-bounded for a range of reasonable measures on $M_\gamma$.

Interest in $M_\gamma$ also stems from other considerations.  In \cite{Bar16}, Barrett defines a scalar invariant (denoted $\beta_\cs(\zeta)$ in Section \ref{SS:EssNormConj} below) associated to any strongly $\C$-convex hypersurface.  The $M_\gamma$ comprise one of the few families for which this invariant is constant. Another such family  
\begin{equation}\label{D:def-of-S_beta}
\cs_{\beta} = \{(\zeta_1,\zeta_2)\in \C^2: \im(\zeta_2) = |\zeta_1|^2 + \beta \re(\zeta_1^2)\}, \qquad 0\le \beta <1,
\end{equation}
was recently studied by Barrett and Edholm in \cite{BarEdh17}.  Both $M_\gamma$ and $\cs_\beta$ are intriguing models in connection with Conjecture \ref{C:EssentialNormConj} on the essential norm of $\bm{L}$, but this conjecture does not directly apply to either family since these hypersurfaces are unbounded (or more to the point, they fail to be $C^1$ at infinity -- except in the very special case of $M_2 = \cs_{0}$).  

Holomorphic function theory on $M_\gamma$ has been previously studied for positive even integers $\gamma = 2n$.  (These are the only $\gamma$ for which $M_\gamma$ is $C^\infty$ at the origin.)  Greiner and Stein \cite{GreSte78}  found an explicit formula for the Szeg\H{o} kernel on $L^2(M_{2n}, dx_1 \w dy_1 \w dx_2)$.  Their work was later used by Diaz \cite{Diaz87} to determine mapping properties of the Szeg\H o projection in this setting.  Note that $dx_1 \w dy_1 \w dx_2$ is the Euclidean surface measure on the parameter space $\R^3$; this measure agrees with $\mu_1$ in our above notation.  Since $r=1 \in \ci_0$ for all $\gamma>1$, a special case of Theorem \ref{T:boundedness-mu_r} implies that $\bm{L}$ is a bounded, skew (unless $\gamma=2$) projection on the same $L^2$-spaces considered by Greiner, Stein and Diaz.

Recall that the Szeg\H{o} projection is the orthogonal projection from a pre-specified $L^2$-space onto its associated Hardy space.  The projection is represented by its Szeg\H{o} kernel, a Hilbert space reproducing kernel.  The existence of this projection is guaranteed, but such kernels can only be concretely written down in a small number of situations.  The Leray transform on the other hand, is always given by an explicit integral formula and can often be shown to be a bounded, skew projection operator onto reasonable Hardy spaces.  An intimate relationship between Szeg\H o projections and ``Cauchy-like'' integral operators $\bm{C}$ (the Leray transform is just one such example) was noticed by Kerzman and Stein in \cite{KerSte78a,KerSte78b}.  They observed that detailed information about the Szeg\H o projection can be extracted from the operator $\bm{A} := \bm{C}^*-\bm{C}$.  See \cite{Bell16} for an expository treatment of these ideas in the complex plane and \cite{Bol06,Bol07,BarLan09,LanSte17a} for more recent developments.

There are important reasons to think of complex projective space (not affine space) as the ideal setting in which to study the Leray transform.  Following the point of view taken in \cite{ScandBook04}, we maintain that $\C$-convexity most naturally belongs in $\C\p^n$ due to the duality seen between points and hyperplanes.  Formula \eqref{E:LerayKernel} shows that $\C$-convexity is built into the definition of $\bm{L}$: the denominator of $\mathscr{L}_\cs$ vanishes on the supporting complex hyperplanes of $\cs$ and $\C$-convexity is precisely the condition needed for this hyperplane to avoid the domain bounded by $\cs$.  The set of complex hyperplanes -- and therefore the set of $\C$-convex hypersurfaces -- is preserved under projective automorphism.  In Section \ref{SSS:mod-Fefferman}, we define the {\em preferred measure} $\wt{\mu}_\cs$.  It follows from work of Bolt \cite{Bol05} that if the Leray kernel is written in terms of $\wt{\mu}_\cs$, then it admits a projective transformation law; cf. \eqref{E:LerayProjTrans}.


This material is consolidated into a satisfying circle of ideas in Section \ref{S:Duality}, where $\bm{L}$ is used to bridge the notions of function theoretic and projective geometric duality.  We construct a projectively invariant Hardy space on each $M_\gamma$, then show that its dual space can be naturally identified with the pullback of the invariant Hardy space on the dual hypersurface.  Unlike Bergman spaces which come naturally equipped with a transformation law, Hardy spaces only transform nicely if set up with respect to very specific measures.  This is where our preferred measure comes in: $\wt{\mu}_{M_\gamma} = \wt{\mu}$ is shown to be a constant multiple of some $\mu_r$, $r \in \ci_0$. Therefore, Theorem \ref{T:Norm} allows us to define the Hardy space
\begin{equation*}
H^2(M_\gamma,\wt{\mu}) := \bm{L}\big(L^2(M_\gamma,\wt{\mu}) \big).
\end{equation*}
By its construction, this space admits a projective transformation law.  (One important order of business is the verification that functions in this space actually correspond to holomorphic functions on the domain $\Omega_\gamma$.  This turns out to be the case, but the proof is postponed until the Appendix.)

In Section \ref{SS:dual-hypersurface-of-M_gamma} we show that the projective dual of $M_\gamma$ can be represented by $M_{\gamma^*}$ (where $\gamma^* = \tfrac{\gamma}{\gamma-1}$ is the H\"older conjugate) and construct a diffeomorphism $w:  M_\gamma \to M_{\gamma^*}$.  The pullback of this map is used to define the {\em preferred dual measure} $w^*\big( \wt{\mu}_{M_{\gamma^*}} \big) := \wt{\mu}_{M_\gamma}^* = \wt{\mu}^*$, along with the dual Hardy space
\begin{equation*}
H_{\sf dual}^2\big(M_\gamma,\wt{\mu}^*\big) := w^*\Big( H^2\big(M_{\gamma^*},\wt{\mu}_{M_{\gamma^*}}\big) \Big).
\end{equation*}

A third measure of interest facilitates a bilinear pairing of these two Hardy spaces.  The {\em pairing measure} $\nu$ appears in \cite{BarEdh17} as part of a universal formulation of the Leray transform; see Theorem \ref{T:GenMLeray} below; also see \cite{ScandBook04}.
In the $M_\gamma$ setting, $\nu$ is a constant multiple of $\sigma$ from Theorem \ref{T:Norm}.  A multiplicative constant can be specified to give a sharp Cauchy-Schwarz inequality relating $\wt{\mu}$, $\wt{\mu}^*$ and $\nu$: for $f \in L^2(M_\gamma,\wt{\mu})$ and $g \in L^2(M_\gamma,\wt{\mu}^*)$,
\begin{equation*}
\bigg| \int_{M_\gamma} f g \,\nu \bigg| \le \|f\|^2_{L^2\left(M_\gamma,\wt{\mu} \right)} \|g\|^2_{L^2\left(M_\gamma,\wt{\mu}^{*} \right)},
\end{equation*}
where equality is achieved for any such $f$ (likewise for any such $g$).  Consider now the map
$$
\chi_\gamma: L^2(M_\gamma,\wt{\mu}^* ) \to \left( L^2(M_\gamma,\wt{\mu})\right)^*
$$
given by $\chi_\gamma(g): f \mapsto  \int_{M_\gamma} f g \,\nu$, along with the companion map 
$$
\chih_\gamma\colon H^2_{\sf dual} \big(M_\gamma, \wt{\mu}^*\big) \to \left(H^2\big(M_\gamma, \wt{\mu}\big)\right)^*
$$
given by the restriction $\chih_\gamma(g)= \chi_\gamma(g)\big{|}_{H^2(M_\gamma,\wt{\mu})}$.  

In Section \ref{SS:bilinear-pairing-of-hardy-spaces} it is shown that $\chih_\gamma$ induces a faithful representation of the dual space and that the efficiency of the Hardy space pairing is measured  by the appropriate norm of $\bm{L}$:

\begin{theorem}
The operator $\chih_\gamma\colon H^2_{\sf dual} \big(M_\gamma, \wt{\mu}^* \big) \to \left(H^2\big(M_\gamma, \wt{\mu}\big)\right)^*$ is invertible with norm
\begin{equation*}
\left\| {\chih}_{\gamma}^{\,-1} \right\| = \norm{\bm{L}}_{L^2(M_\gamma,\wt{\mu})}.
\end{equation*}
\end{theorem}


The paper is organized as follows.  Sections \ref{SS:LerayTransform}-\ref{SS:EssNormConj} cover necessary background material, while Section \ref{SS:notation} is included to collect frequently used notation.  Section \ref{SS:Symmetries} is concerned with geometric symmetries of $M_\gamma$, objects that greatly assist our analysis of $\bm{L}_\gamma$ in Sections \ref{SS:LerayReparam}-\ref{SS:norm-of-subLeray}.  (The computations in Section \ref{S:AnalysisLgamma} only involve the measure $\sigma$, but the stage there is set for the larger class of measures considered later.)  In Section \ref{SS:MgammaAtInfinity}, we study the geometry of $M_\gamma$ at infinity.  Section \ref{S:SymbolFunction} mainly concerns properties of the $\sigma$-symbol function $C_\sigma(\gamma,k)$.  In Section \ref{SS:Leray-boundedness-m_r} we widen our scope and examine the action of $\bm{L}$ with respect to the more general measures $\mu_r$.  Sections \ref{SS:Leray-adjoints} and \ref{SS:related-operators} focus on adjoints and other related operators.  Section \ref{S:Duality} begins with an introduction to projective duality, leading directly into a detailed study of the projective dual of $M_\gamma$ in Section \ref{SS:dual-hypersurface-of-M_gamma}.  Section \ref{SS:distinguished-measures} is concerned with certain distinguished measures on $M_\gamma$, while \ref{SS:dual-Hardy} and \ref{SS:bilinear-pairing-of-hardy-spaces} set up the invariant Hardy spaces and their dual spaces. A computationally intensive Appendix is included at the end of the paper to establish crucial facts without disrupting the exposition of our main results.

The authors would like to emphasize their surprise at the exactness of many results in this paper.  We suspect that mathematicians interested in special function theory may be particularly drawn to computations found in Sections \ref{SS:props-of-C_sigma(gamma,k)}, \ref{SS:ProofOfC(gamma,k)DecreasingWithk}, \ref{SS:Leray-boundedness-m_r} and the Appendix.

\section{Background and preliminaries}\label{S:Prelims}

\subsection{The Leray transform}\label{SS:LerayTransform}

A domain $\Omega \subset \C^n$ is said to be $\C$-convex if its intersection with any complex line is both connected and simply connected (when non-empty).  It is said to be $\C$-linearly convex if the complement can be written as the union of complex hyperplanes.  These two notions coincide when the boundary of $\Omega$ is $C^1$ (see Section 2.5 in \cite{ScandBook04}).  A hypersurface $\cs$ bounding a domain $\Omega$ is said to be {\em $\C$-convex} (resp. {\em $\C$-linearly convex}) if $\Omega$ is $\C$-convex (resp. $\C$-linearly convex).  A hypersurface $\cs$ that is locally projectively equivalent to a strongly convex hypersurface is said to be  {\em strongly $\C$-convex}.  (Section 5.2 of \cite{Bar16} discusses equivalent characterizations of strong $\C$-convexity.)

Let $\cs\subset\C^n$ be a $\C$-linearly convex hypersurface with defining function $\rho$ and $f$ a function defined on $\cs$.  The Leray transform maps $f$ to a holomorphic function on $\Omega$ whenever the following integral makes sense (Theorem \ref{T:range-of-r-building-holomorphic-full-Leray} in the Appendix gives conditions on $f$ which guarantee the holomorphicity of $\bm{L}f$ on $\Omega_\gamma$):
\begin{subequations}\label{E:Leray}
\begin{align}
\bm{L}_{\cs}f(z) &:= \int_{\cs} f(\zeta) \, \mathscr{L}_{\cs}(z,\zeta), \label{E:LerayIntegralFormula} \\
\mathscr{L}_{\cs}(z,\zeta) &:= \frac{1}{(2\pi i)^n}  \frac{\partial\rho(\zeta) \wedge (\bar \partial\partial\rho(\zeta))^{n-1}}{\left< \partial\rho(\zeta),(\zeta - z) \right>^n} \label{E:LerayKernel}.
\end{align}
Note that the {\em Leray kernel} $\mathscr{L}_{\cs}$ is a form of bi-degree $(n,n-1)$.  Here $\left<\cdot,\cdot \right>$ denotes the natural {\em bilinear} pairing between $(1,0)$-forms and vectors.  This definition is independent of the choice of $\rho$; see Chapter IV of \cite{Range86} for more information.

Separate the Leray kernel \eqref{E:LerayKernel} into two pieces, each of which {\em does} depend on the choice of defining function:
\begin{align}
\ell_\rho(z,\zeta) &:= \frac{1}{\left< \partial\rho(\zeta),(\zeta - z) \right>^n} \label{E:LerayFunction}\\
\lambda_\rho(\zeta) &:= \frac{1}{(2\pi i)^n} \, \partial\rho(\zeta) \wedge (\bar \partial\partial\rho(\zeta))^{n-1}. \label{E:LerayLeviMeasure}
\end{align}
\end{subequations}
Refer to $\lambda_\rho$ as {\em the Leray-Levi measure for the defining function $\rho$}.  We often suppress the subscript $\rho$ when it is clear from context.

Now calculate the pieces of the Leray transform on $M_\gamma$.  Using the defining function
\begin{equation}\label{E:DefiningFunction}
\rho(z) = |z_1|^\gamma - \im{(z_2)},
\end{equation}
equations \eqref{E:LerayFunction} and \eqref{E:LerayLeviMeasure} yield
\begin{subequations}
\begin{align}
\ell_\rho(z,\zeta) &= \frac{4}{\left(\gamma\bar{\zeta}_1|\zeta_1|^{\gamma-2}(\zeta_1-z_1)+i(\zeta_2-z_2)\right)^2}\label{E:ell_M(z,zeta)},\\
\lambda_\rho(\zeta) &= \frac{\gamma^2}{32\pi^2 i} |\zeta_1|^{\gamma-2}\, d\zeta_2\wedge d\bar{\zeta}_1\wedge d\zeta_1. \label{E:LerayLevi}
\end{align}
In this paper, the Leray transform $\bm{L}_{M_\gamma}$ will frequently be denoted by $\bm{L}$, or sometimes by $\bm{L}_\gamma$ when it is important to keep track of the exponent.
Equations \eqref{E:LerayIntegralFormula} and \eqref{E:LerayKernel} give
\begin{align}
\bm{L} f(z) &= \frac{\gamma^2}{8\pi^2 i} \int_{M_\gamma} f(\zeta)\frac{|\zeta_1|^{\gamma-2}d\zeta_2\wedge d\bar{\zeta}_1\wedge d\zeta_1}{\left[\gamma\bar{\zeta}_1|\zeta_1|^{\gamma-2}(\zeta_1-z_1)+i(\zeta_2-z_2)\right]^2}. \label{E:LerayMgamma}
\end{align}
\end{subequations}

\subsection{Projective invariants and transformation laws}\label{S:ProjStuff}

Recall that, with respect to the standard affinization, automorphisms of  $\C\mathbb{P}^2$ have the form 
\begin{equation}\label{E:Xhat}
\hat X\colon(z_1,z_2)\mapsto
\left( \frac{D+Ez_1+Fz_2}{A+Bz_1+Cz_2}, \frac{G+Hz_1+Iz_2}{A+Bz_1+Cz_2} \right),
\end{equation}
where $\hat X$ is induced by the invertible matrix 
\begin{equation}\label{E:matrix-X}
X:=\begin{pmatrix}
A  & B  &  C \\
D  & E  &  F \\
 G &  H &   I
\end{pmatrix}.
\end{equation}
In homogeneous coordinates, these are linear maps.  Projective automorphisms map strongly $\C$-convex hypersurfaces to strongly $\C$-convex hypersurfaces; see Section 5.2 of \cite{Bar16}.  See also \cite{Bol08, Bol10} for a detailed treatment of M\"obius geometry on hypersurfaces.

Direct computation reveals the Jacobian determinant of $\hat X$ is given by
\[
\det\hat X'=\frac{\det X} {(A+Bz_1+Cz_2)^3}.
\]  
A similar formula holds in dimension $n$ with exponent $n+1$ in the denominator.  Below we will have occasion to refer to $(\det\hat X')^{n/(n+1)}:$  it is clear that there are $n+1$ distinct well-defined branches of this function.  It is understood that the same branch is to be used within a single computation when this expression appears repeatedly.

\subsubsection{A projective scalar invariant} \label{SSS:BetaInvariant}
For a strongly pseudoconvex hypersurface $\cs\subset\C^2$ with defining function $\rho$ and a point $\zeta\in\cs$, the scalar quantity 
\begin{equation}\label{D:BetaInvariant}
\beta_\cs(\zeta) := \left|\frac{\det
{ \begin{pmatrix}
 0  & \rho_{z_1}  & \rho_{z_2}  \\
\rho_{z_1}  &  \rho_{z_1 z_1} & \rho_{z_2 z_1}  \\
\rho_{z_2}  & \rho_{z_1 z_2}  &   \rho_{z_2 z_2} 
\end{pmatrix}
}}
{
\det \begin{pmatrix}
 0  & \rho_{z_1}  & \rho_{z_2}  \\
\rho_{\bar z_ 1 }  & \rho_{z_1 \bar z_ 1} & \rho_{z_2 \bar z_ 1}  \\
\rho_{\bar z_2 }  & \rho_{z_1 \bar z_ 2}  &   \rho_{z_2 \bar z_ 2}  
\end{pmatrix}}
(\zeta)\right|
\end{equation}
(where subscripts denote derivatives) is directly invariant under projective automorphisms.  This scalar is the absolute value of an invariant tensor ${\mathscr B}_S$ introduced by the first author in Section 5.3 of \cite{Bar16}.  A computation shows that for all $\zeta\in M_\gamma$,
\begin{equation}\label{E:beta-invariant-on-Mgamma}
\beta_{M_\gamma}(\zeta)=\frac{|\gamma-2|}{\gamma}.
\end{equation}

\subsubsection{Fefferman surface measure}\label{SSS:FeffermanMeasure}
On page 259 of \cite{Fef79}, Fefferman introduced an invariant measure $\mu_\cs^{\sf Fef}$ for an arbitrary
smooth strongly pseudoconvex hypersurface $\cs\subset\C^n$.  Viewing
$\mu_\cs^{\sf Fef}$ as a positive
$(2n-1)$-form,  it is characterized by the equation 
\begin{equation}\label{E:fefdef}
\mu_\cs^{\sf Fef}\w d\rho =  c_n \,M(\rho)^{1/(n+1)} \omega_{\C^n},
\end{equation}
where $\omega_{\C^n}$ is the Euclidean volume $2n$-form, $\rho$ is a defining
function for $\cs$  and
$M(\rho)$ denotes Fefferman's complex Monge-Amp\`ere operator defined by
\begin{equation}\label{E:madef}
M(\rho)=(-1)^{n} \det
\begin{pmatrix}
\rho & \rho_{z_j}\\
\rho_{\bar{z}_k} & \rho_{z_{j}{\bar{z}_k}}
\end{pmatrix}.
\end{equation}
(The subscripts denote differentiation and $c_n$ is a dimensional constant.  In Section \ref{SS:distinguished-measures} we specify an explicit value of this constant that is convenient for the purposes of this paper.)

Fefferman's measure satisfies the transformation law
\begin{equation}\label{E:TransGen}
\Psi^*\left( \mu_{\Psi(\cs)}^{\sf Fef}\right)=\left|\det\Psi'\right|^{2n/(n+1)}\,\mu_\cs^{\sf Fef}
\end{equation}
for CR diffeomorphisms $\Psi$; see \cite{Hirachi90}.

If  a Szeg\H{o} projection operator for a compact strongly pseudoconvex $\cs$ is defined with respect to  $\mu_\cs^{\sf Fef}$  and $\Psi$ is a CR diffeomorphism with a well-defined branch of $\left(\det\Psi'\right)^{n/(n+1)}$, then from \eqref{E:TransGen} we find that the operator 
\[
f\mapsto \left(\det\Psi'\right)^{n/(n+1)}\cdot \left(f\circ\Psi\right)
\]
maps $L^2\big(\Psi(\cs),\mu_{\Psi(\cs)}^{\sf Fef}\big)$ isometrically to $L^2\left(\cs,\mu_\cs^{\sf Fef}\right)$, preserving the corresponding Hardy spaces. This leads to the transformation law demonstrated in \cite{Hirachi90}:
\begin{equation}\label{E:Hirachi-szego-trans-law}
\bm{S}_{\cs} \left(\left( \det \Psi'\right)^{n/(n+1)}\left(f\circ \Psi\right)\right) = \left( \det \Psi'\right)^{n/(n+1)}
\left(\left( \bm{S}_{\Psi(\cs)} f\right)\circ\Psi\right).
\end{equation}

From \eqref{E:DefiningFunction}, \eqref{E:fefdef} and \eqref{E:madef}, the Fefferman measure on $M_\gamma$ is seen to be given by
\begin{align}\label{E:FefMeasureMgamma}
\mu_{M_\gamma}^{\sf Fef} &= c_2 \frac{ \gamma^{2/3}}{2^{7/3}i} |\zeta_1|^{\frac{\gamma-2}{3}}\,d\zeta_2\w d\bar{\zeta}_1\w d\zeta_1.
\end{align}

\subsubsection{The preferred measure}\label{SSS:mod-Fefferman}
For the general theory of the Leray transform it is useful to work with a modified version of Fefferman's measure.  In Section 8 of \cite{Bar16}, the first author defines a projective modification of $\mu_\cs^{\sf Fef}$ tailored to suit a natural pairing of Hardy spaces for which the Fefferman measure is not optimal; we examine this pairing in Section \ref{SS:bilinear-pairing-of-hardy-spaces}.

Throughout this paper, we refer to this modification of Fefferman's measure as the {\em preferred measure} and denote it by $\wt{\mu}_\cs$.  In two dimensions it is given by the formula
\begin{equation*}
\widetilde{\mu}_\cs(\zeta)=\dfrac{\mu_\cs^{\sf Fef}(\zeta)}{\sqrt[3]{1-\beta_\cs(\zeta)^2}},
\end{equation*}
where $\beta_\cs(\zeta)$ is defined in \eqref{D:BetaInvariant}.  By \eqref{E:beta-invariant-on-Mgamma}, this invariant is constant on $M_\gamma$.  Thus,
\begin{align}
\widetilde{\mu}_{M_\gamma}
&=\frac{\gamma^{2/3}}{2^{2/3}{(\gamma-1)}^{1/3}}
\,\mu_{M_\gamma}^{\sf Fef}\notag\\
&= c_2 \frac{\gamma^{4/3}}{8(\gamma-1)^{1/3} i}\, |\zeta_1|^{\frac{\gamma-2}{3}} \,d\zeta_2\w d\bar{\zeta}_1\w d\zeta_1. \label{E:ModifedFefMGamma}
\end{align}
Equation \eqref{E:TransGen} and the projective invariance of $\beta_\cs(\zeta)$ now show that
\begin{equation}\label{E:TransLFT}
\Psi^*\left( \widetilde{\mu}_{\Psi(\cs)}\right)=\left|\det\Psi'\right|^{4/3}\,\widetilde{\mu}_\cs
\end{equation}
for any projective automorphism $\Psi$.

For such $\Psi$ we can use \eqref{E:TransLFT} to deduce that a transformation law analogous to \eqref{E:Hirachi-szego-trans-law} holds for Szeg\H{o} projections based on the preferred measure $\wt{\mu}_\cs$.  If $\cs$ is strongly $\C$-convex and $\Psi$ is a projective automorphism, then Bolt \cite{Bol05} has shown that the Leray transform satisfies the same transformation law.  In two dimensions, this is
\begin{equation}\label{E:LerayProjTrans}
 \bm{L}_{\cs} \left(\left( \det \Psi'\right)^{2/3}\left(f\circ \Psi\right)\right) = \left( \det \Psi'\right)^{2/3}
\left(\left( \bm{L}_{\Psi(\cs)} f\right)\circ\Psi\right).
\end{equation}

\subsection{Projective invariants and the essential norm conjecture}\label{SS:EssNormConj}

In \cite{BarEdh17}, the authors conjecture a quantitative relationship between the projective invariant $\beta_\cs(\zeta)$ defined in \eqref{D:BetaInvariant} and the essential norm of the Leray transform on smoothly bounded, strongly $\C$-convex hypersurfaces $\cs \subset \C^2$.  We restate a version of this conjecture here:

\begin{conjecture}\label{C:EssentialNormConj}
Let $\cs \subset \C^2$ be a smoothly bounded, strongly $\C$-convex hypersurface and $\bm{L}$ its Leray transform.  There is a ``reasonable" family $\sF$ of measures on $\cs$ such that if $\mu \in \sF$, the essential $L^2(\cs,\mu)$-norm is given by
\begin{equation}\label{E:LerayEssentialConjecture}
\norm{\bm{L}}_{{L^2_{\sf ess}(\cs,\mu)}} = \sup_{\zeta\in \cs }\frac{1}{\sqrt[4]{1-\beta_\cs(\zeta)^2}}.
\end{equation}
\end{conjecture}
At this stage, the word ``reasonable" is left intentionally vague, but there is good reason to expect such a relationship.  This has been previously established (Theorem 3 in \cite{BarLan09}) when $\cs$ is the boundary of a $C^2$-smooth, bounded, strongly convex {\em Reinhardt} domain and $\mu$ is any continuous, positive multiple of surface measure on $\cs$.

While Conjecture \ref{C:EssentialNormConj} only pertains to {\em bounded} hypersurfaces, the invariant $\beta_\cs(\zeta)$ does not require $\cs$ to be bounded.  It was noted in \cite{Bar16} that this invariant is constant for both $\cs_\beta$ (defined in \eqref{D:def-of-S_beta}) and $M_\gamma$.  It easily follows from \eqref{D:BetaInvariant} that
\begin{equation*}
\beta_{\cs_\beta}(\zeta) = \beta \qquad \textrm{and}\qquad \beta_{M_\gamma} = \frac{|\gamma-2|}{\gamma},
\end{equation*}
and consequently, the quantity on the right hand side of \eqref{E:LerayEssentialConjecture} is equal to 
\begin{equation}\label{E:important-quantities-Sbeta-Mgamma}
\frac{1}{\sqrt[4]{1-\beta^2}} \,\, \textrm{when} \,\,  \cs=\cs_\beta \qquad \textrm{and} \qquad \sqrt{\frac{\gamma}{2\sqrt{\gamma-1}}} \,\, \textrm{when} \,\, \cs = M_\gamma.
\end{equation}
We raise the following question: 

\noindent $(\star)$ {\em Are the quantities in \eqref{E:important-quantities-Sbeta-Mgamma} connected to the Leray transform in a meaningful way?}

In \cite{BarEdh17}, the authors gave an affirmative answer to $(\star)$ on $\cs_\beta$ by proving that 
\begin{equation*}
\norm{\bm{L}}_{L^2(\cs_\beta,\sigma)} = \frac{1}{\sqrt[4]{1-\beta^2}} = \norm{\bm{L}}_{L_{\sf ess}^2(\cs_\beta,\sigma)},
\end{equation*}
where $\sigma = dx_1\w dy_1 \w dx_2$ is (a constant multiple of) the Leray-Levi measure corresponding to the defining function $\rho(\zeta) = |\zeta_1|^2 + \beta\re{(\zeta_1^2)}-\im{(\zeta_2)}$.  The norm coincides with the essential norm because the former is attained on an infinite dimensional subspace of $L^2(\cs_\beta,\sigma)$.

In the present paper, $(\star)$ is answered affirmatively on $M_\gamma$, but the interpretation of this quantity is slightly different.  Indeed, Theorem \ref{E:Norm} and Remark \ref{R:Leray-norm-achieved-on-infinite-dimensional-space} show that
\begin{equation}\label{E:norm-and-ess-norm-agree-Mgamma}
\norm{\bm{L}}_{L^2(M_\gamma,\sigma)} = \frac{\gamma}{2\sqrt{\gamma-1}} = \norm{\bm{L}}_{L_{\sf ess}^2(M_\gamma,\sigma)},
\end{equation}
where now $\sigma = \alpha^{\gamma-1}d\alpha\w d\theta\w ds$ is (a constant multiple of) the Leray-Levi measure corresponding to the defining function $\rho(\zeta) = |\zeta_1|^\gamma - \im{(\zeta_2)}$.  Notice that \eqref{E:norm-and-ess-norm-agree-Mgamma} is the {\em square} of the desired quantity in \eqref{E:important-quantities-Sbeta-Mgamma}.  
But this desired quantity does naturally arise (Theorem \ref{T:HighFreqNorm}) as the high-frequency limit norm of $\bm{L}$ in the space $L^2(M_\gamma,\mu_r)$ for {\em every} measure of the form $\mu_r = \alpha^{r}d\alpha\w d\theta\w ds$, $r\in\R$, (which clearly includes $\sigma$):
\begin{equation*}
\limsup_{k\to\infty} \norm{\bm{L}_k}_{L^2(M_\gamma,\mu_r)} = \sqrt{\frac{\gamma}{2\sqrt{\gamma-1}}}.
\end{equation*}

This suggests a connection between the geometric invariant $\beta_\cs$ on similar unbounded hypersurfaces (those admitting both $S^1$ and $\R$ actions) and the behavior of the Leray transform at high frequencies.  This connection is further developed in the coming work of Edholm and Shelah \cite{EdhShe22}, where a much more general class of such hypersurfaces is shown to satisfy an analogue of Conjecture \ref{C:EssentialNormConj}.  In the two dimensional Reinhardt setting, the high frequency limit norm coincides with the essential norm, but $M_\gamma$ shows these two notions don't necessarily agree on unbounded domains.  The high frequency limit norm can, however, be conceptualized as a generalized essential norm as we now shall see.

\subsubsection{The grade-essential-norm} \label{SS:GenEss}

Given a Hilbert space $E$, recall that the norm-closure of the space of finite-rank operators is the two-sided ideal of compact operators, and that the essential norm of a bounded operator $\bm{T}\colon E\to E$ is the distance to the ideal of compact operators; equivalently, the essential norm of $\bm{T}$ is its norm in the quotient {\em Calkin algebra}.  See \cite{CowMacBook95} for more information.

Consider on the other hand a Hilbert space $E$ admitting an orthogonal decomposition
$$E= \displaystyle{\bigoplus_{k=-\infty}^\infty E_k }.$$   (The decomposition is also known as a {\em $\Z$-grading} of $E$.) Call a bounded operator 
$\bm{T}\colon E \to E$
{\em decomposable} if it may be written as 
$$\bm{T}=  \displaystyle{ \bigoplus_{k=-\infty}^\infty \bm{T}_k}$$ with each $\bm{T}_k\colon E_k\to E_k$.
  (One way this could arise is from Fourier series decomposition based on an $S^1$ action on $E$ commuting with $\bm{T}$.)  The decomposable operators form an algebra.

We will call a decomposable $\bm{T}$ {\em grade-compact} if it is the norm-limit of operators taking values in 
$\Span\left\{ E_{-k}, E_{-k+1}, \dots, E_k\right\}$; equivalently, if $\|\bm{T}_k\|\to 0$ as $|k|\to\infty$.

We define the {\em grade-essential-norm}  of a decomposable $\bm{T}$ to be the distance 
($\limsup \|\bm{T}_k\|$) from $\bm{T}$ to the space of 
grade-compact decomposable operators.  This is consistent with Lef\`evre's notion of a {\em generalized essential norm} \cite{Lefevre09}.  It is clear that the grade-essential-norm coincides with the high-frequency limit norm defined in \eqref{E:HighFreqLimSup} above.


The grade-compact decomposable operators form a two-sided ideal in the algebra of decomposable operators; the grade-essential-norm induces a norm on the corresponding quotient algebra.

\begin{remark}\label{R:extended-grade-essential}
In Theorem \ref{T:grade-essential-for-F}, we will have occasion to consider situations in which finitely many of the $\bm{T}_k$ are unbounded; in such a setting we still refer to 
$\limsup \|\bm{T}_k\|$ as the grade-essential norm. $\lozenge$

\end{remark}

\subsection{Notation}\label{SS:notation}

Throughout this paper, a number of different measures occur and precise notation is required to prevent ambiguity.  We include this section to collect this material in one central location.

\subsubsection{Measures on $M_\gamma$}\label{S:meas}  

We begin by gathering formulas for previously discussed measures on $M_\gamma$, along with others that will appear in Sections \ref{S:AdjointsNormalOps} and \ref{S:Duality}.  The measures are given both in $(\zeta_1,\zeta_2)$ coordinates and in terms of the parametrization of $M_\gamma$ given by $(\alpha e^{i\theta},s + i\alpha^\gamma)$, where $\alpha=|\zeta_1|$, $\theta=\arg(\zeta_1)$, and $s=\re(\zeta_2)$.  

The {\em Leray-Levi measure} associated to the defining function $\rho(z) = |z_1|^\gamma - \im{(z_2)}$ is
\begin{align}
\lambda_\rho &= \frac{\gamma^2}{32\pi^2 i} |\zeta_1|^{\gamma-2}d\zeta_2\wedge d\bar{\zeta}_1\wedge d\zeta_1 \notag\\
&= \frac{\gamma^2}{16\pi^2}\, \alpha^{\gamma-1}\,ds \wedge d\alpha \wedge d\theta. \label{D:LambdaRhoReParam}
\end{align}

In \eqref{E:NuDef}, a family $\nu^A$ of {\em pairing measures} parametrized by matrices $A \in GL(n+1,\C)$ is introduced. The matrices are used to induce various affinizations of projective space and in our work each $M_\gamma$ corresponds to a matrix $A_\gamma$, defined in \eqref{A-Spec}.  For a fixed $\gamma>1$, the pairing measure we are interested in is
\begin{align}\label{E:sec2-measure-notation-nu}
\nu^{A_\gamma} = 4\lambda_\rho = \frac{\gamma^2}{4 \pi^2 }\, \alpha^{\gamma-1}ds \w d\alpha\w d\theta.
\end{align}

The {\em preferred measure} $\wt{\mu}_\cs$ defined in \eqref{E:ModifedFefMGamma} will be used in Section \ref{S:Duality} to construct a family of invariant Hardy spaces on $M_\gamma$.  It takes the form
\begin{align}\label{E:sec2-measure-notation-mu-tilde}
\widetilde{\mu}_{M_\gamma} &= c_2\frac{\gamma^{4/3}}{8(\gamma-1)^{1/3} i}\,|\zeta_1|^{\frac{\gamma-2}{3}}d\zeta_2\w d\bar{\zeta}_1\w d\zeta_1 \notag\\
&= c_2\frac{\gamma^{4/3}}{4(\gamma-1)^{1/3}}\, \alpha^{\frac{\gamma+1}{3}} ds \wedge d\alpha \wedge d\theta.
\end{align} 

Also in Section \ref{S:Duality}, we meet the counterpart to $\wt{\mu}_{\cs}$, which we call the {\em preferred dual measure}.  This is obtained by taking the preferred measure on the projective dual hypersurface and pulling it back to $M_\gamma$.  The representation of the dual hypersurface (and hence the measure) in affine space depends on an affinization of projective space, and therefore on a matrix.  These measures are used to construct the duals of the Hardy spaces mentioned above.  With respect to the aforementioned matrix $A_\gamma$, the preferred dual measure is given by
\begin{align}\label{E:sec2-measure-notation-mu-tilde-star}
\wt{\mu}_{M_\gamma}^{*,A_\gamma} &= c_2 \frac{\gamma^{4/3}}{8(\gamma-1)i} |\zeta_1|^{\frac{5(\gamma-2)}{3}} \,d\zeta_2\w d\bar{\zeta}_1\w d\zeta_1 \notag \\
&= c_2 \frac{\gamma^{4/3}}{4(\gamma-1)} \, \alpha^{\frac{5\gamma - 7}{3}} ds \wedge d\alpha \wedge d\theta.
\end{align}

Notice that these measures take a very specific form when written in $(\alpha,\theta,s)$ coordinates.  For $r \in \R$, define the measure
\begin{equation}\label{D:MeasureMu_r}
\mu_r := \alpha^{r}ds \wedge d\alpha \wedge d\theta = \frac{1}{2 i} |\zeta_1|^{r-1}d\zeta_2\wedge d\bar{\zeta}_1\wedge d\zeta_1,
\end{equation} 
together with the family of all constant positive multiples of $\mu_r$:
\begin{align}\label{D:FamilyMeasuresFr}
\sF_r = \{c \mu_r: c > 0 \}.
\end{align}
Thus, 
$$
\lambda_\rho, \nu^{A_\gamma} \in \sF_{\gamma-1}, \qquad \wt{\mu}_{M_\gamma} \in \sF_{\frac{\gamma+1}{3}}, \qquad \wt{\mu}_{M_\gamma}^{*,A_\gamma} \in \sF_{\frac{5\gamma-7}{3}}.$$
Finally, because the measure $\mu_{\gamma-1}$ occurs so frequently throughout the paper, we give it its own symbol:
\begin{align}\label{D:MeasureSigma}
\sigma = \alpha^{\gamma-1}ds \wedge d\alpha \wedge d\theta.
\end{align}

We emphasize the trivial fact that though multiplying a measure by a positive constant will uniformly scale the norms of functions, norms of operators remain unchanged.  Indeed, if $\mu$ is a measure on $M_\gamma$ and $\bm{T}$ is a bounded operator on $L^2(M_\gamma,\mu)$, then for all constants $c>0$ we have
\begin{align}\label{E:OpsUnchangedByMeasureScaling}
\norm{\bm{T}}_{L^2(M_\gamma,\mu)} = \norm{\bm{T}}_{L^2(M_\gamma,c\mu)}.
\end{align}

\subsubsection{Inner products, norms and adjoints}
Let $\mu$ be a measure on $M_\gamma$ and $f,g\in L^2(M_\gamma,\mu)$.  We often write their $L^2(M_\gamma,\mu)$-inner product by 
\begin{align}\label{E:3DimInnerProd}
\Big< f,g \Big>_\mu &:= \int_{M_\gamma} f(\zeta)\overline{g(\zeta)}\,\mu(\zeta).
\end{align}
The $L^2(M_\gamma,\mu)$-norm is similarly defined:
\begin{equation}\label{E:3DimNorm}
\norm{f}_{\mu} = \sqrt{\Big< f,f \Big>_{\mu}}.
\end{equation}

Now let $\mu_r$ be the measure in \eqref{D:MeasureMu_r}.  Throughout this paper, we frequently encounter the following one-dimensional integral related to $\mu_r$.  For $f,g \in L^2\big((0,\infty),\alpha^{r}d\alpha\big)$, define the one-dimensional inner product
\begin{align}\label{E:1DimInnerProd}
\Big< f,g \Big>_{(\mu_r,1)} := \int_0^\infty f(\alpha)\overline{g(\alpha)} \,\alpha^r d\alpha.
\end{align}
Similarly, define the norm
\begin{equation}\label{E:1DimNorm}
\norm{f}_{(\mu_r,1)} = \sqrt{\Big< f,f \Big>_{(\mu_r,1)}}.
\end{equation}

If $\bm{T} : L^2(M_\gamma,\mu) \to L^2(M_\gamma,\mu)$ is a bounded operator, it admits a bounded adjoint $\bm{T}^{(*,\mu)}$, satisfying
\begin{equation}
\Big< \bm{T} f,g \Big>_\mu  = \Big< f, \bm{T}^{(*,\mu)} g \Big>_\mu.
\end{equation}
It is trivial to see that
\begin{equation}
\bm{T}^{(*,\mu)} = \bm{T}^{(*,c \mu)} 
\end{equation}
for all constants $c > 0$, a fact that is repeatedly used in Section \ref{S:Duality}.


\section{Geometry and Analysis on $M_\gamma$}\label{S:AnalysisLgamma}

$M_\gamma$ is both strongly $\C$-convex and real analytic away from $\{ \zeta_1 = 0 \}$ for each $\gamma > 1$.  But on this set, these properties simultaneously hold only in the special case of $\gamma=2$.  $M_\gamma$ is weakly $\C$-convex here when $\gamma>2$, and it fails to be $C^2$ smooth when $1< \gamma <2$. (It is only $C^{\gamma}$ smooth).  These two notions are highly intertwined; see \cite{BarLan09} for a detailed study of this in the Reinhardt setting.  The interplay between $\C$-convexity and smoothness is further illuminated from a projective dual point of view; see Section \ref{S:Duality}.  At infinity, the behavior of $M_\gamma$ is even less regular; this is the subject of Section \ref{SS:MgammaAtInfinity}.

\subsection{Projective automorphisms of $M_\gamma$}\label{SS:Symmetries}

Certain affine maps preserving $M_\gamma$ are examined here.  These play an important role in our analysis of the Leray transform.  The three types of maps considered are

\begin{itemize}

\item Rotations in the first coordinate: $r_\theta(z_1,z_2) = (e^{i\theta}z_1,z_2)$, $\theta \in [0,2\pi)$.

\item Real translations in the second coordinate: $t_s(z_1,z_2) = (z_1,z_2+s)$, $s\in \R$.

\item Non-isotropic dilations: $\delta_\alpha(z_1,z_2) = (\alpha z_1,\alpha^\gamma z_2)$, $\alpha > 0$.

\end{itemize}

It is easily verified that the following pairs of maps commute:
\begin{itemize}

\item $r_{\theta_1} \circ r_{\theta_2} = r_{\theta_1+\theta_2} = r_{\theta_2} \circ r_{\theta_1}$

\item $t_{s_1} \circ t_{s_2} = t_{s_1+s_2} = t_{s_2} \circ t_{s_1}$

\item $\delta_{\alpha_1} \circ \delta_{\alpha_2} = \delta_{\alpha_1 \alpha_2} = \delta_{\alpha_2} \circ \delta_{\alpha_1}$

\item $r_\theta \circ t_s = t_s \circ r_\theta$

\item $r_\theta \circ \delta_\alpha = \delta_\alpha \circ r_\theta$.

\end{itemize}
(Note though, that $t_s \circ \delta_\alpha \neq \delta_\alpha \circ t_s$.)  Because $r_\theta$ commutes with both $t_s$ and $\delta_\alpha$, this rotational symmetry is the starting point of our analysis.

Define the set 
\begin{equation}\label{E:ExceptonalSetV}
V = \{(\zeta_1,\zeta_2)\in M_\gamma: \zeta_1 = 0 \}.
\end{equation}
$V$ is a copy of $\R$ sitting inside $M_\gamma$.  Each $r_\theta, t_s, \delta_\alpha$ fixes $V$; collectively, they act transitively on $M_\gamma \backslash V$.  Indeed:
\begin{theorem}
Given any pair of points $z,\zeta \in M_\gamma \backslash V$, there is a unique map of the form $t_s\circ r_\theta \circ \delta_\alpha$ sending $z \mapsto \zeta$.
\end{theorem}
\begin{proof}
The numbers $\alpha$ and $\theta$ are determined by setting $\alpha e^{i\theta} = \zeta_1 {z_1}^{-1}$.  The second variable is appropriately adjusted by setting $s = \re{(\zeta_2)} - \alpha^{\gamma}\re{(z_2)}$.
\end{proof}

\subsection{Leray reparametrization}\label{SS:LerayReparam}

The symmetries $r_\theta, t_s, \delta_\alpha$ lead to a reparametrization of $M_\gamma$.  Write $z,\zeta \in M_\gamma$ as
\begin{align}
z &= (\alpha_z e^{i\theta_z}, s_z + i\alpha_z^\gamma), \label{E:z-reparam}\\
\zeta &= (\alpha_\zeta e^{i\theta_\zeta}, s_\zeta + i\alpha_\zeta^\gamma), \label{E:zeta-reparam}
\end{align}
where each  $\alpha\ge0$, each $\theta \in [0,2\pi)$, and each $s \in \R$.  Equation \eqref{D:LambdaRhoReParam} says the Leray-Levi measure from \eqref{E:LerayLevi} takes the form
\begin{align}
\lambda_\rho = \frac{\gamma^2}{32\pi^2 i} |\zeta_1|^{\gamma-2}d\zeta_2\wedge d\bar{\zeta}_1\wedge d\zeta_1 &= \frac{\gamma^2}{16\pi^2}\, \alpha_\zeta^{\gamma-1} ds_\zeta \wedge d\alpha_\zeta \wedge d \theta_\zeta \label{D:measureLambda}.
\end{align}

Now re-write the Leray transform.  Expression \eqref{E:LerayMgamma} becomes
\begin{align}
\bm{L} f(z) &= \frac{\gamma^2}{4\pi^2} \int_{M_\gamma} f(\zeta) \frac{\alpha_\zeta^{\gamma-1} ds_\zeta \wedge d\alpha_\zeta \wedge d \theta_\zeta}{\left[\left( (\gamma-1)\alpha_\zeta^\gamma + \alpha_z^\gamma + i(s_{\zeta} - s_z) \right) - \left(\gamma \alpha_z \alpha_\zeta^{\gamma-1} e^{i(\theta_z-\theta_\zeta)}\right) \right]^2} \notag\\
&=  \frac{\gamma^2}{4\pi^2} \int_{M_\gamma}  \,\frac{f(\zeta)}{\big[A-B  e^{i(\theta_z-\theta_\zeta)}\big]^2} \,\sigma(\zeta), \label{E:IntKernelAB}
\end{align}
where $\sigma(\zeta) =  \alpha_\zeta^{\gamma-1} ds_\zeta \wedge d\alpha_\zeta \wedge d \theta_\zeta$ and
\begin{align}\label{E:DefofAandB}
A &:= (\gamma-1)\alpha_\zeta^\gamma + \alpha_z^\gamma + i(s_{\zeta} - s_z), \qquad B := \gamma \alpha_z \alpha_\zeta^{\gamma-1}.
\end{align}

\subsection{Series expansion}\label{SS:SeriesExpn}

The $S^1$ action on $M_\gamma$ yields a decomposition of the $L^2$-space 
\begin{align}\label{E:SubspaceDecomp}
L^2(M_\gamma,\sigma) = \bigoplus_{k=-\infty}^\infty  L_k^2(M_\gamma,\sigma),
\end{align}
where functions in $L_k^2(M_\gamma,\sigma)$ have the form $f_k(s,\alpha)e^{ik\theta}$.  Each $f(\zeta) = f(s_\zeta, \alpha_\zeta,\theta_\zeta) \in L^2(M_\gamma,\sigma)$ decomposes as a partial Fourier series
\begin{equation}\label{E:FourierSeries}
f(s_\zeta, \alpha_\zeta,\theta_\zeta) = \sum_{k=-\infty}^\infty f_k(s_\zeta,\alpha_\zeta)\,e^{ik\theta_\zeta},
\end{equation}
and the following version of Parseval's theorem holds
\begin{align}
\norm{f}^2_{L^2(M_\gamma,\sigma)} &= \int_{M_\gamma}|f(s_\zeta, \alpha_\zeta,\theta_\zeta)|^2\,\alpha_\zeta^{\gamma-1}ds_\zeta \, d\alpha_\zeta \, d \theta_\zeta \notag \\
&= \int_0^{2\pi} \int_0^\infty \int_{-\infty}^{\infty} \left( \sum_{j,k = -\infty}^{\infty} f_j(s_\zeta,\alpha_\zeta) \overline{f_k(s_\zeta,\alpha_\zeta)} e^{i(j-k)\theta_\zeta}\right) \alpha_\zeta^{\gamma-1}ds_\zeta \, d\alpha_\zeta \, d \theta_\zeta \notag \\
&= 2\pi \sum_{k=-\infty}^{\infty} \int_0^\infty \int_{-\infty}^{\infty}  |f_k(s_\zeta,\alpha_\zeta)|^2\, \alpha_\zeta^{\gamma-1}ds_\zeta \, d\alpha_\zeta. \label{E:ParsevalSeries}
\end{align}

Return now to the computation of the Leray transform.  The rational function appearing in the integrand of \eqref{E:IntKernelAB} may be expanded as a series as a consequence of the following:
\begin{lemma}\label{L:GeneralAMGM}
Let $x,y \ge 0$ and $\gamma>1$.  Then
\begin{equation}\label{E:GeneralAMGM}
x^\gamma + (\gamma-1)y^\gamma \ge \gamma x y^{\gamma-1},
\end{equation}
with equality if and only if $x=y$.
\end{lemma}
\begin{proof}
Divide both sides of \eqref{E:GeneralAMGM} by $y^\gamma$ and set $u = \frac xy$.  Now let $f(u) = u^\gamma - \gamma u +(\gamma - 1)$.  This function is convex and attains its global minimum of $0$ at $u=1$.  
\end{proof}

Consider the formula for $\bm{L}f$ in \eqref{E:IntKernelAB}.  Lemma \ref{L:GeneralAMGM} shows $\left|\frac{B}{A}\right| < 1$, unless both $s_z = s_\zeta$ and $\alpha_z = \alpha_\zeta$.  Away from this set of $\sigma$-measure $0$,
\begin{align}
\frac{1}{\big[A-B  e^{i(\theta_z-\theta_\zeta)}\big]^2} =  \frac{1}{A^2}\sum_{k=0}^\infty \,(k+1)\left(\frac{B}{A}\right)^k e^{ik(\theta_z-\theta_\zeta)}.\label{E:KernelAB}
\end{align}
Returning now to \eqref{E:IntKernelAB}, 
\begin{align}
\bm{L} f(s_z,\alpha_z,\theta_z) &=  \frac{\gamma^2}{4\pi^2} \int_{M_\gamma}  \,\frac{f(\zeta) \sigma(\zeta)}{\big[A-B  e^{i(\theta_z-\theta_\zeta)}\big]^2}\notag \\
&=  \sum_{k=0}^\infty  \, \frac{\gamma^2(k+1)}{4\pi^2} \int_{M_\gamma} f(\zeta) \frac{B^k}{A^{k+2}} e^{-ik\theta_\zeta} \, \sigma(\zeta) e^{ik\theta_z} \notag \\
&:= \sum_{k=0}^\infty \Lambda_k f(s_z,\alpha_z) e^{ik\theta_z}. \label{E:SumLambda_k}
\end{align}
Now replace $f$ with its Fourier expansion \eqref{E:FourierSeries}
\begin{align}
\Lambda_kf(s_z,\alpha_z) &=  \frac{\gamma^2(k+1)}{4\pi^2}  \int_{M_\gamma} \left(\sum_{j=-\infty}^\infty f_j(s_\zeta,\alpha_\zeta)\,e^{ij\theta_\zeta}\right) \frac{B^k}{A^{k+2}} e^{-ik\theta_\zeta} \, \sigma(\zeta) \notag\\
&= \frac{\gamma^2(k+1)}{4\pi^2} \sum_{j=-\infty}^\infty \int_{M_\gamma} \frac{B^k}{A^{k+2}}  f_j(s_\zeta,\alpha_\zeta) e^{i(j-k)\theta_\zeta} \, \alpha_\zeta^{\gamma-1} ds_\zeta d\alpha_\zeta d\theta_\zeta\notag\\
&= \frac{\gamma^2(k+1)}{2\pi} \int_0^\infty \int_{-\infty}^\infty \frac{B^k}{A^{k+2}}  f_k(s_\zeta,\alpha_\zeta)\,\alpha_\zeta^{\gamma-1} ds_\zeta \,d\alpha_\zeta.  \label{D:Lambda_k} \
\end{align}
This gives the series decomposition of the Leray transform:

\begin{definition}\label{D:SubLerayTransformLk}
For each nonnegative integer $k$, define the sub-Leray operator $\bm{L}_k$ to be the restriction of $\bm{L}$ to the subspace $L_k^2(M_\gamma,\sigma)$.
\end{definition}\label{D:Leray_k}

\begin{remark}
We later consider the same decomposition in the space $L^2(M_\gamma,\mu_r)$. $\lozenge$
\end{remark}

The orthogonal decomposition \eqref{E:SubspaceDecomp} shows $\bm{L} = \bigoplus_{k=0}^\infty \bm{L}_k$, and \eqref{E:SumLambda_k} and \eqref{D:Lambda_k} give that
\begin{align}\label{E:Leray_k}
\bm{L}_kf(s_z,\alpha_z,\theta_z) = \Lambda_k f(s_z,\alpha_z) e^{ik\theta_z}.
\end{align}

\subsection{Fourier transforms}\label{SS:ApplicationOf FourierTransform}
We continue the analysis of $\bm{L}_k$ by taking a closer look at $\Lambda_k f(s_z,\alpha_z)$.  The goal is to use the Fourier transform and understand the unitarily equivalent operator $\cf^{-1}\bm{L}_k\cf$.  

Importing $A$ and $B$ from \eqref{E:DefofAandB}, 
\begin{align}
\eqref{D:Lambda_k} &= \frac{\gamma^{k+2}(k+1)\alpha_z^k}{2\pi} \int_0^\infty  \alpha_\zeta^{(k+1)(\gamma-1)} \int_{-\infty}^\infty \frac{f_k(s_\zeta,\alpha_\zeta)}{[(\gamma-1)\alpha_\zeta^\gamma + \alpha_z^\gamma + i(s_{\zeta} - s_z)]^{k+2}} ds_\zeta \, d\alpha_\zeta  \notag \\
&= \frac{(i\gamma)^{k+2}(k+1)\alpha_z^k}{2\pi} \int_0^\infty \alpha_\zeta^{(k+1)(\gamma-1)}  \int_{-\infty}^\infty \frac{f_k(s_\zeta,\alpha_\zeta)}{[(s_z - s_{\zeta})+i\big((\gamma-1)\alpha_\zeta^\gamma + \alpha_z^\gamma\big)]^{k+2}} ds_\zeta \,d\alpha_\zeta \notag.
\end{align}
This shows that
\begin{align}\label{D:IkConvolution}
\Lambda_{k}f(s_z,\alpha_z) = \frac{(i\gamma)^{k+2}(k+1)\alpha_z^k}{2\pi} \int_0^\infty \alpha_\zeta^{(k+1)(\gamma-1)}  \Big( f_k(\,\cdot\, ,\alpha_\zeta) \,*\, G_k \Big)(s_z) \,d\alpha_\zeta,
\end{align}
where the function $G_k $ is given by 
\begin{align}\label{E:DefG_kANDC}
G_{k}(s) := \frac{1}{(s+i((\gamma-1)\alpha_\zeta^\gamma + \alpha_z^\gamma))^{k+2}}. 
\end{align}

To better understand the integral defining $\Lambda_k f$, consider the Fourier transform $\cf$ and its inverse $\cf^{-1}$:  For $g\in L^1(\R)\cap L^2(\R)$,
\begin{align}\label{D:FourierTransformDef}
\cf g(\xi) = \int_{-\infty}^\infty g(s) e^{- 2\pi i s \xi}\,ds, \qquad \cf^{-1}g(\xi) = \int_{-\infty}^\infty g(s) e^{2\pi i s \xi}\,ds.
\end{align}
Under this convention, $\cf$ and $\cf^{-1}$ are well known to transform convolutions to products
\begin{equation}\label{E:FourierConvolutions}
\cf(g * h)(\xi) = \cf{g}(\xi) \cdot \cf{h}(\xi), \qquad \cf^{-1}(g * h)(\xi) = \cf^{-1}{g}(\xi) \cdot \cf^{-1}{h}(\xi).
\end{equation}
These operators also extend to isometries of $L^2(\R)$
\begin{align}\label{E:FourierIsometryL2}
\norm{g}_{L^2(\R)} = \norm{\cf g}_{L^2(\R)} = \norm{\cf^{-1} g}_{L^2(\R)}.
\end{align}

Applying $\cf^{-1}$ in the $s_z$ variable to \eqref{D:IkConvolution} yields 
\begin{align}
\cf^{-1} \Lambda_kf(\xi,\alpha_z) &= \frac{(i\gamma)^{k+2}(k+1)\alpha_z^k}{2\pi} \int_0^\infty \alpha_\zeta^{(k+1)(\gamma-1)}  \cf^{-1} \Big( f_k(\,\cdot\, ,\alpha_\zeta) \,*\, G_k \Big)(\xi) \,d\alpha_\zeta \label{E:InvFtransLambda_k}.
\end{align}
By \eqref{E:FourierConvolutions}, 
\begin{align}\label{E:FourierConvolutions2}
\cf^{-1}\Big( f_k(\,\cdot\, ,\alpha_\zeta) \,*\, G_k \Big)(\xi) = \cf^{-1} f_k(\xi,\alpha_\zeta) \cdot \cf^{-1} G_k(\xi).
\end{align}

Now calculate $\cf^{-1} G_k(\xi)$.  From \eqref{E:DefG_kANDC}, notice the term $C := (\gamma-1)\alpha_\zeta^\gamma + \alpha_z^\gamma > 0$, unless both $\alpha_z=\alpha_\zeta=0$.
\begin{proposition}\label{P:InverseFourierComputation}
Let $C>0$.  The inverse Fourier transform of $G_k(s) = \frac{1}{(s+iC)^{k+2}}$ is
\begin{align}\label{E:InverseFourierGk}
	\cf^{-1}G_k(\xi) = \int_{-\infty}^{\infty} \frac{e^{2\pi i s \xi}}{(s+iC)^{k+2}} \, ds
	= \begin{dcases}
		0, & \xi \ge 0\\ \\
		-\frac{(2\pi i)^{k+2}}{(k+1)!} \xi^{k+1} e^{2\pi\xi C}, & \xi < 0.\\ 
		\end{dcases}
\end{align}
\end{proposition}
\begin{proof}
Since $C>0$, the function $H_k(z) := \frac{e^{2\pi i z \xi}}{(z+iC)^{k+2}}$ has a pole in the lower half plane.  For any $R>0$, the integral
\begin{equation*}
\int_{-R}^R \frac{e^{2\pi i s \xi}}{(s+iC)^{k+2}}\,ds.
\end{equation*}
can be thought of as a piece of a contour integral around a semicircle with base on the $x$-axis.

When $\xi>0$, consider such a semicircle in the upper half plane traversed counterclockwise.  It can be seen that the radial portion of the integral tends to $0$ as $R\to0$.  On the other hand, this function is holomorphic inside this semicircle contour.  Thus Cauchy's theorem implies $\cf^{-1}G_k(\xi) = 0$ for $\xi>0$.

When $\xi<0$, consider a semicircle in the lower half plane traversed clockwise.  For $R$ sufficiently large, the contour encloses the pole of $H_k$.  As above, the radial portion of this integral tends to $0$ as $R\to0$.  Thus, Cauchy's integral formula shows
\begin{align*}
\lim_{R\to\infty} \int_{-R}^R \frac{e^{2\pi i s \xi}}{(s+iC)^{k+2}}\,ds &= -\frac{2\pi i}{(k+1)!} \frac{d^{k+1}}{ds^{k+1}}\Big( e^{2\pi i s \xi} \Big)\Bigg|_{s=-iC} = -\frac{(2\pi i)^{k+2}}{(k+1)!} \xi^{k+1} e^{2\pi\xi C},
\end{align*}
completing the proof.
\end{proof}

Combining \eqref{E:InvFtransLambda_k} and \eqref{E:FourierConvolutions2} with Proposition \ref{P:InverseFourierComputation} shows
\begin{equation}\label{E:Lambda_k}
\cf^{-1} \Lambda_kf(\xi,\alpha_z) = \eta_k(\xi) \tau_k(\xi,\alpha_z) \int_0^\infty \cf^{-1}f_k(\xi,\alpha_\zeta)\kappa_k(\xi,\alpha_\zeta)\,\alpha_\zeta^{\gamma-1}d\alpha_\zeta,
\end{equation}
where 
\begin{align}
\eta_k(\xi) &= \frac{(-2\pi\xi)^{k+1}\gamma^{k+2}}{k!} \cdot\mathds{1}_{\{\xi<0\}}, \label{D:eta_k}\\
\tau_k(\xi,\alpha_z) &= \alpha_z^{k} e^{2\pi \xi \alpha_z^\gamma}\cdot \mathds{1}_{\{\xi<0\}}, \label{D:tau_k}\\
\kappa_k(\xi,\alpha_\zeta) &= \alpha_\zeta^{k(\gamma-1)} e^{2\pi \xi(\gamma-1) \alpha_\zeta^\gamma}\cdot \mathds{1}_{\{\xi<0\}}, \label{D:kappa_k}
\end{align}
and $\mathds{1}_{\{\xi<0\}}$ is the indicator function of the set $\{\xi<0\}$.  These three functions occur frequently throughout the rest of the paper.

Recall the one dimensional inner product $\big<\cdot,\cdot\big>_{(\sigma,1)}$ defined by formula \eqref{E:1DimInnerProd}:
\begin{equation*}
\Big< f,g \Big>_{(\sigma,1)} := \int_0^\infty f(\alpha)\overline{g(\alpha)} \,\alpha^{\gamma-1}d\alpha.
\end{equation*}
We summarize this in the following
\begin{proposition}\label{P:LkMultipliers}
The operator $\cf^{-1}\bm{L}_k\cf$ is given by
\begin{equation}\label{E:LkMultipliers}
\cf^{-1}\bm{L}_k\cf f (\xi,\alpha_z,\theta_z) = \eta_k(\xi) \tau_k(\xi,\alpha_z) \Big< f_k(\xi,\cdot),\kappa_k(\xi,\cdot) \Big>_{(\sigma,1)} \,e^{ik\theta_z},
\end{equation}
where $f (\xi,\alpha_z,\theta_z) = \sum_j f_j(\xi,\alpha_z)e^{ik\theta_z} \in L^2(M_\gamma,\sigma)$.
\end{proposition}
\begin{proof}
By its definition, $\bm{L}_k$ is only a nonzero operator on $L^2_k(M_\gamma,\sigma)$.  Since $\cf f = \sum_j \cf f_j e^{ij\theta}$, we have
\begin{align}
\cf^{-1}\bm{L}_k(\cf f) (\xi,\alpha_z,\theta_z) &= \cf^{-1} \Lambda_k\cf f_k(\xi,\alpha_z)\,e^{ik\theta_z} \notag\\
&= \eta_k(\xi) \tau_k(\xi,\alpha_z) \left(\int_0^\infty f_k(\xi,\alpha_\zeta)\kappa_k(\xi,\alpha_\zeta)\alpha_\zeta^{\gamma-1}d\alpha_\zeta\right)e^{ik\theta_z}\notag\\
&= \eta_k(\xi) \tau_k(\xi,\alpha_z) \Big< f_k(\xi,\cdot),\kappa_k(\xi,\cdot) \Big>_{(\sigma,1)} \,e^{ik\theta_z}, \notag
\end{align}
where the second equality follows from \eqref{E:Lambda_k}.
\end{proof}

\subsection{Norms of the sub-Leray operators}\label{SS:norm-of-subLeray}

Let us first record a computational lemma that will be used several times throughout the paper.
\begin{proposition}\label{P:IntegralComputation}
Fix $x>-1$, $y>0$, $\gamma\ge1$.  Then
\begin{align*}
\int_0^\infty \alpha^x e^{-y\alpha^\gamma}\,d\alpha = \tfrac{1}{\gamma}\, y^{-\frac{1+x}{\gamma}} \Gamma\big(\tfrac{1+x}{\gamma}\big).
\end{align*}
\end{proposition}
\begin{proof}
This follows from the substitution $t=y\alpha^\gamma$ and the definition of the $\Gamma$-function.
\end{proof}
This immediately implies the following
\begin{corollary}\label{C:(sigma,1)-norms-of-tau-and-kappa}
Let $\tau_k(\xi,\cdot)$ and $\kappa_k(\xi,\cdot)$ be as defined in \eqref{D:tau_k} and \eqref{D:kappa_k}.  Then for $\xi<0$,
\begin{align}
\norm{\tau_k(\xi,\cdot)}_{(\sigma,1)}^2 &= \tfrac{1}{\gamma}\cdot (-4\pi\xi)^{-1-\frac{2k}{\gamma}} \Gamma\big(\tfrac{2k}{\gamma}+1\big), \label{E:TaukNorm1} \\  
\norm{\kappa_k(\xi,\cdot)}_{(\sigma,1)}^2 &= \tfrac{1}{\gamma}\cdot (-4\pi\xi(\gamma-1))^{\frac{2k}{\gamma}-2k-1} \Gamma\big(1+ 2k -\tfrac{2k}{\gamma}\big).  \label{E:KappakNorm1}
\end{align}
\end{corollary}
\begin{proof}
Recall that $\tau_k$ and $\kappa_k$ are nonzero only for $\xi < 0$.  Now by choosing the appropriate values of $x$ and $y$ in Proposition \ref{P:IntegralComputation},
\begin{align*}
\norm{\tau_k(\xi,\cdot)}_{(\sigma,1)}^2 &= \int_0^\infty \alpha_z^{2k+\gamma-1} e^{4\pi\xi\alpha_z^\gamma}d\alpha_z = \tfrac{1}{\gamma}\cdot (-4\pi\xi)^{-1-\frac{2k}{\gamma}} \Gamma\big(\tfrac{2k}{\gamma}+1\big), \\  
\norm{\kappa_k(\xi,\cdot)}_{(\sigma,1)}^2 &= \int_0^\infty \alpha_\zeta^{(2k+1)(\gamma-1)} e^{4\pi\xi(\gamma-1)\alpha_\zeta^\gamma}d\alpha_\zeta \notag\\
&= \tfrac{1}{\gamma}\cdot (-4\pi\xi(\gamma-1))^{\frac{2k}{\gamma}-2k-1} \Gamma\big(1+ 2k -\tfrac{2k}{\gamma}\big). 
\end{align*}
\end{proof}

Now consider a function $f(\xi,\alpha_z,\theta_z) = f_k(\xi,\alpha_z) e^{i \theta_z}\in L_k^2(M_\gamma,\sigma)$.  Equation \eqref{E:LkMultipliers} and Cauchy-Schwarz show
\begin{align}
|\cf^{-1}\bm{L}_k\cf f (\xi,\alpha_z,\theta_z)|^2 &= |\eta_k(\xi)|^2 |\tau_k(\xi,\alpha_z)|^2 \left|\Big< f_k(\xi,\cdot),\kappa_k(\xi,\cdot) \Big>_{(\sigma,1)}\right|^2 \notag\\
&\le |\eta_k(\xi)|^2 |\tau_k(\xi,\alpha_z)|^2 \norm{f_k(\xi,\cdot)}_{(\sigma,1)}^2 \norm{\kappa_k(\xi,\cdot)}_{(\sigma,1)}^2, \label{E:Bound1F^-1LF} 
\end{align}
with equality if and only if $f_k(\xi, \cdot)$ is a multiple of $\kappa_k(\xi, \cdot)$, i.e.,  
\begin{align}\label{E:C-SEquality}
f_k(\xi,\alpha_z) = m_k(\xi) \kappa_k(\xi,\alpha_z),
\end{align}
where, recalling the definition of $\kappa_k$ in \eqref{D:kappa_k}, the multiplier function $m_k$ must satisfy
\begin{align}\label{E:KappaKMultiplierCondition-sigma}
\int_{-\infty}^0 | m_k(\xi)|^2 |\xi|^{\frac{2k}{\gamma}-2k-1}\,d\xi < \infty.
\end{align}

Estimate \eqref{E:Bound1F^-1LF} implies
\begin{align*}
\int_0^\infty |\cf^{-1}\bm{L}_k\cf f (\xi,\alpha_z,\theta_z)|^2 \alpha_z^{\gamma-1}d\alpha_z &\le |\eta_k(\xi)|^2 \norm{\tau_k(\xi,\cdot)}_{(\sigma,1)}^2 \norm{f_k(\xi,\cdot)}_{(\sigma,1)}^2 \norm{\kappa_k(\xi,\cdot)}_{(\sigma,1)}^2\notag\\
&:=C_\sigma(\gamma,k) \norm{f_k(\xi,\cdot)}_{(\sigma,1)}^2,
\end{align*}
where the ($\xi$ independent) quantity $C_\sigma(\gamma,k) = |\eta_k(\xi)|^2 \norm{\tau_k(\xi,\cdot)}_{(\sigma,1)}^2 \norm{\kappa_k(\xi,\cdot)}_{(\sigma,1)}^2$.  We will call this the $\sigma$-{\em symbol function}.  Combining equations \eqref{D:eta_k}, \eqref{E:TaukNorm1} and \eqref{E:KappakNorm1}, we see
\begin{align}\label{E:first-encounter-sigma-symbol-function}
C_\sigma(\gamma,k) =\frac{\Gamma\big(\frac{2k}{\gamma}+1\big) \Gamma\big(2k-\frac{2k}{\gamma}+1\big)}{\Gamma(k+1)^2} \left( \tfrac{\gamma}{2}\right)^{2k+2} (\gamma-1)^{\frac{2k}{\gamma}-2k-1}.
\end{align}
Detailed analysis of the $\sigma$-symbol function is carried out in Section \ref{S:SymbolFunction}.

\begin{remark}\label{R:XiIndependenceOfSymbolFunction}
It is remarkable that $C_\sigma(\gamma,k)$ is independent of $\xi$.  The more general symbol function $C_{\mu_r}(\gamma,k)$ calculated in Section \ref{SS:Leray-boundedness-m_r} also shares this property. $\lozenge$
\end{remark}

We now compute the norm of the operator $\bm{L}_k$.

\begin{theorem}\label{T:OpNormLk}
The operator $\bm{L}_k: L^2(M_\gamma,\sigma) \to L_k^2(M_\gamma,\sigma)$ is bounded with norm given by
\begin{equation*}
\norm{\bm{L}_k}_{L^2(M_\gamma,\sigma)} = \sqrt{C_\sigma(\gamma,k)}.
\end{equation*}
\end{theorem}
\begin{proof}
First note that $\bm{L}_k$ is unitarily equivalent to $\cf^{-1}\bm{L}_k\cf$.  Now
\begin{align}
\norm{\cf^{-1}\bm{L}_k\cf f}_\sigma^2 &= \int_0^{2\pi}\int_{-\infty}^\infty\int_0^\infty |\cf^{-1}\bm{L}_k\cf f (\xi,\alpha_z,\theta_z)|^2 \alpha_z^{\gamma-1}\,d\xi\,d\alpha_z d\theta_z \notag\\
&\le C_\sigma(\gamma,k) \int_0^{2\pi}\int_{-\infty}^\infty \norm{f_k(\xi,\cdot)}_{(\sigma,1)}^2 \,d\xi\,d\theta_z \label{E:LkC-SIneq}\\
&= C_\sigma(\gamma,k) \int_0^{2\pi}\int_{-\infty}^\infty \int_0^\infty |f_k(\xi,\alpha_z)|^2 \alpha_z^{\gamma-1}d\alpha_z \,d\xi\,d\theta_z \notag\\
&= C_\sigma(\gamma,k) \norm{f}_\sigma^2.\notag
\end{align}
Equality in \eqref{E:LkC-SIneq} holds if and only if $f_k$ is of the form given by \eqref{E:C-SEquality} and \eqref{E:KappaKMultiplierCondition-sigma}.
\end{proof}

\begin{theorem}\label{T:L_k-is-a-projection-sigma}
The operator $\bm{L}_k:L^2(M_\gamma,\sigma) \to L_k^2(M_\gamma,\sigma)$ is a projection.
\end{theorem}
\begin{proof}
Let $f \in L^2(M_\gamma,\sigma)$.  It must be shown that $\bm{L}_k \circ \bm{L}_k = \bm{L}_k$.  Proceed by conjugating by the Fourier transform.  Proposition \ref{P:LkMultipliers} says
\begin{align}
\cf^{-1}\bm{L}_k \circ \bm{L}_k \cf (f) &= \cf^{-1}\bm{L}_k \cf \circ \cf^{-1} \bm{L}_k \cf (f) \notag\\
&= \cf^{-1}\bm{L}_k \cf\Big( \eta_k(\xi) \tau_k(\xi,\alpha_z) \Big< f_k(\xi,\cdot),\kappa_k(\xi,\cdot) \Big>_{(\sigma,1)} \,e^{ik\theta_z} \Big) \notag\\
&= \eta_k(\xi) \Big< f_k(\xi,\cdot),\kappa_k(\xi,\cdot) \Big>_{(\sigma,1)} \cf^{-1}\bm{L}_k \cf\Big(\tau_k(\xi,\alpha_z) e^{ik\theta_z} \Big) \notag\\
&= \eta_k(\xi)^2 \Big< f_k(\xi,\cdot),\kappa_k(\xi,\cdot) \Big>_{(\sigma,1)} \tau_k(\xi,\alpha'_{z}) \Big< \tau_k(\xi,\cdot),\kappa_k(\xi,\cdot) \Big>_{(\sigma,1)},\label{E:L_k^2Comp}
\end{align}
where $\alpha_z'$ denotes the radial part of the $\zeta_1$ variable after two applications of $\cf^{-1} \bm{L}_k \cf$.

From the definitions of $\tau_k$ and $\kappa_k$ in \eqref{D:tau_k} and \eqref{D:kappa_k},
\begin{align}
\Big< \tau_k(\xi,\cdot),\kappa_k(\xi,\cdot) \Big>_{(\sigma,1)} &= \int_0^\infty \tau_k(\xi,\alpha) \overline{\kappa_k(\xi,\alpha)}\alpha^{\gamma-1}\,d\alpha \notag\\
&= \mathds{1}_{\{\xi<0\}}\int_0^\infty \alpha^{\gamma(k+1)-1} e^{2\pi\xi\gamma\alpha^\gamma} d\alpha \notag\\
&= \mathds{1}_{\{\xi<0\}}\frac{\Gamma(k+1)}{(-2\pi\xi)^{k+1}\gamma^{k+2}} \label{E:TauKappaInt}\\
&= \frac{1}{\eta_k(\xi)}, \label{E:TauKappaInt=Eta}
\end{align}
where \eqref{E:TauKappaInt} follows from Proposition \ref{P:IntegralComputation}.  Returning to \eqref{E:L_k^2Comp}, we now see that
\begin{align*}
\cf^{-1}\bm{L}_k \circ \bm{L}_k \cf (f) &= \eta_k(\xi)^2 \Big< f_k(\xi,\cdot),\kappa_k(\xi,\cdot) \Big>_{(\sigma,1)} \tau_k(\xi,\alpha'_{z}) \Big< \tau_k(\xi,\cdot),\kappa_k(\xi,\cdot) \Big>_{(\sigma,1)}\\
&= \eta_k(\xi) \tau_k(\xi,\alpha'_{z}) \Big< f_k(\xi,\cdot),\kappa_k(\xi,\cdot) \Big>_{(\sigma,1)}\\
&= \cf^{-1}\bm{L}_k \cf (f).
\end{align*}
This shows that $\cf^{-1} \bm{L}_k \circ \bm{L}_k \cf = \cf^{-1} \bm{L}_k \cf$ and thus $\bm{L}_k \circ \bm{L}_k = \bm{L}_k$.
\end{proof}

\subsection{$M_\gamma$ at infinity}\label{SS:MgammaAtInfinity}

To understand the behavior of $M_\gamma$ at infinity, apply the projective automorphism $\Phi$ mapping $(z_1,z_2) \mapsto (\frac{z_1}{z_2},\frac{1}{z_2})$.  This transformation swaps the line $\{z_2=0\}$ with the line at $\infty$.  Setting $\wt{M}_\gamma := \Phi({M_\gamma})$, the transformed hypersurface can be represented as follows:
\begin{align}
\wt{M}_\gamma &= \left\{ (z_1,z_2): -|z_2|^{\gamma-2}\,\im{(z_2)} = |z_1|^\gamma\right\} \notag \\
&= \left\{ (x_1,y_1,x_2,y_2): -y_2 (x_2^2 + y_2^2)^{\frac{\gamma}{2} -1} = (x_1^2 + y_1^2)^{\frac{\gamma}{2}}  \right\} \label{E:Mgamma-at-inf-xy-coords} \\
&=  \left\{ (r_1 e^{i\theta_1} , r_2 e^{i\theta_2}): -r_2^{\gamma-1} \sin{\theta_2} = r_1^\gamma\right\}. \label{E:Mgamma-at-inf-polar-coords}
\end{align}

The behavior of $\wt{M}_\gamma$ at $z_2 = 0$ gives the behavior of $M_\gamma$ at $\infty$.  For any $\gamma>1$, \eqref{E:Mgamma-at-inf-polar-coords} shows that sending $r_2 \to 0$ forces $r_1 \to 0$.  This implies that the closure of $M_\gamma$ in $\C\mathbb{P}^2$ contains a single point at infinity with homogeneous coordinates $[0:0:1]$.  We now consider the regularity here:

\begin{theorem}
Let $\gamma>1$.  The closure of $M_\gamma$ in $\C\mathbb{P}^2$ fails to be a $C^1$ submanifold near $[0:0:1]$  except in the case  $\gamma=2$.  However, $M_\gamma$ is globally Lipschitz.
\end{theorem}
\begin{proof}
The discussion above lets us transfer the problem to $\wt{M}_\gamma$.  If $\wt{M}_\gamma$ were a $C^1$ manifold at the origin, then one of the four real variables in \eqref{E:Mgamma-at-inf-xy-coords} would be expressable as a $C^1$ function of the other three in a neighborhood of the origin.  Now, given the constraint in \eqref{E:Mgamma-at-inf-xy-coords} that $y_2 \le 0$, we see that the only possibility is for $y_2$ to be a function of $x_1, y_1$ and $x_2$.  

Starting from the equation $(x_1^2 + y_1^2)^{\frac{\gamma}{2}} + y_2 (x_2^2 + y_2^2)^{\frac{\gamma}{2} -1} = 0$, implicit differentiation shows
\begin{align}\label{E:dy2/dx2}
\frac{\partial y_2}{\partial x_2} = \frac{(2-\gamma)x_2 y_2}{x_2^2 + (\gamma-1)y_2^2} &= \frac{(2-\gamma)\cos{\theta_2} \sin{\theta_2}}{1 + (\gamma-2)\sin^2{\theta_2}}.
\end{align}
Now set, for instance, $\theta_2 = \pi$ and $\theta_2 = \frac{5\pi}{4}$ in \eqref{E:dy2/dx2} and see that different values are obtained, {\em except in the case that $\gamma=2$}. (Any two distinct choices of $\theta_2 \in [\pi,2\pi]$ can be used.)  This shows that $\frac{\partial y_2}{\partial x_2}$ fails to be continuous at the origin, which in turn proves $M_\gamma$ is not $C^1$ at infinity.

Now see that $M_\gamma$ is Lipschitz.  It will again suffice to consider $\wt{M}_\gamma$ in the form of \eqref{E:Mgamma-at-inf-xy-coords} and show that the derivatives $\frac{\partial y_2}{\partial x_1}$, $\frac{\partial y_2}{\partial y_1}$, and $\frac{\partial y_2}{\partial x_2}$ are $L^\infty$ functions near the origin.

First, see that the quantity $|1 + (\gamma-2)\sin^2{\theta}| \ge \min\{1, {\gamma-1} \}$ for any choice of $\theta$.  Then implicit differentiation shows
\begin{align*}
\left| \frac{\partial y_2}{\partial x_1} \right| = \left| \frac{-\gamma x_1 (x_1^2 + y_1^2)^{\frac{\gamma}{2}-1} (x_2^2 + y_2^2)^{2-\frac{\gamma}{2}} }{x_2^2 + (\gamma-1)y_2^2} \right| \lesssim r_1^{\gamma-1} r_2^{2-\gamma} \lesssim r_2^{\frac{\gamma^2-2\gamma+1}{\gamma}+{(2-\gamma)}} = r_2^{\frac{1}{\gamma}},
\end{align*}
which is clearly bounded near the origin.  By symmetry, $\frac{\partial y_2}{\partial y_1}$ satisfies an analogous bound.  Similarly, it quickly follows from \eqref{E:dy2/dx2} that $\left| \frac{\partial y_2}{\partial x_2} \right| \lesssim 1$, finishing the proof.
\end{proof}


\begin{remark}
It is easy to check that the line at infinity $\{[0:z_1:z_2]\}$ is the only complex line in $\C\mathbb{P}^2\setminus \Omega_{\gamma}$ passing through the point $[0:0:1]$ and  thus may be viewed as the tangent line at $[0:0:1]$ (in a weak sense). $\lozenge$
\end{remark}

\begin{remark}
In the case of $\gamma = 1$, \eqref{E:Mgamma-at-inf-polar-coords} shows that the completion of the hypersuface $M_1$ in $\C\mathbb{P}^2$ consists of a closed disc at infinity, rather than just a single point. $\lozenge$
\end{remark}


\section{Analysis of the $\sigma$-symbol function}\label{S:SymbolFunction}

\subsection{Properties of $C_{\sigma}(\gamma,k)$}\label{SS:props-of-C_sigma(gamma,k)}
Analysis of $C_{\sigma}(\gamma,k)$ for each $k$ value yields precise information on $\bm{L}$.  Recall that
\begin{align}\label{E:DefCgamma(k)}
C_{\sigma}(\gamma,k) &:= \frac{\Gamma\big(\frac{2k}{\gamma}+1\big) \Gamma\big(2k-\frac{2k}{\gamma}+1\big)}{\Gamma(k+1)^2} \left( \tfrac{\gamma}{2}\right)^{2k+2} (\gamma-1)^{\frac{2k}{\gamma}-2k-1}.
\end{align}

\begin{theorem}\label{T:CgammaProps}
Let $\gamma>1$ and $k$ be a non-negative integer.  The symbol function $C_\sigma(\gamma,k)$ has the following properties:
\begin{itemize}
\item[$(a)$] $C_\sigma(\gamma,0) = \frac{\gamma^2}{4(\gamma-1)}$.

\item[$(b)$]  $C_\sigma(2,k) = 1$ for all positive integers $k$.

\item[$(c)$] $C_\sigma(\gamma,k)$ is H\"{o}lder symmetric in $\gamma$, i.e., $C_\sigma(\gamma,k) = C_\sigma\big(\frac{\gamma}{\gamma-1},k\big)$.

\item[$(d)$] For each $\gamma \neq 2$, $C_\sigma(\gamma,k)$ strictly decreases as a function of $k$.

\item[$(e)$] $\lim_{k\to\infty} C_\sigma(\gamma,k) = \frac{\gamma}{2\sqrt{\gamma-1}}$.
\end{itemize}
\end{theorem}
\begin{proof}
Parts $(a)$ and $(b)$ are immediate from the formula.  We will also quickly verify $(c)$ and $(e)$, leaving $(d)$ for the next section.

For part $(c)$, compute
\begin{align*}
C_\sigma\big(\tfrac{\gamma}{\gamma-1},k\big) &= \frac{\Gamma\big(\frac{2k(\gamma-1)}{\gamma}+1\big) \Gamma\big(2k-\frac{2k(\gamma-1)}{\gamma}+1\big)}{\Gamma(k+1)^2}\big( \tfrac{\gamma}{2\gamma-2}\big)^{2k+2} \big(\tfrac{1}{\gamma-1}\big)^{\frac{2k(\gamma-1)}{\gamma}-2k-1} \\
&= \frac{\Gamma\big(2k-\frac{2k}{\gamma}+1\big) \Gamma\big(\frac{2k}{\gamma}+1\big)}{\Gamma(k+1)^2} \big( \tfrac{\gamma}{2}\big)^{2k+2} (\gamma-1)^{\frac{2k}{\gamma}-2k-1} \\
&= C_\sigma(\gamma,k).
\end{align*}

Part $(e)$ follows from Stirling's formula.  Recall that $f$ and $g$ are said to be {\em asymptotically equivalent} when
\begin{equation}\label{E:AsymptoticEquivalence}
\lim_{x\to\infty} \frac{f(x)}{g(x)} = 1.
\end{equation}
When \eqref{E:AsymptoticEquivalence} holds, we write
\begin{align*}
f(x) \sim g(x).
\end{align*}
Stirling's formula says $\Gamma(x+1)\sim {\sqrt {2\pi x}}\left({\tfrac {x}{e}}\right)^{x}$, which implies the following asymptotic equivalences:
\begin{align}
\Gamma({k+1})^2 &\sim 2\pi e^{-2k} k^{2k+1}, \label{E:AsEquiv1}\\
\Gamma\big(\tfrac{2k}{\gamma}+1\big) &\sim \sqrt{2\pi}\cdot e^{-\frac{2k}{\gamma}} \big( \tfrac{2k}{\gamma}\big)^{\frac{2k}{\gamma}+\frac12}, \label{E:AsEquiv2}\\
\Gamma\big(2k-\tfrac{2k}{\gamma}+1\big) &\sim \sqrt{2\pi}\cdot e^{\frac{2k}{\gamma}-2k} \big( \tfrac{2k}{\gamma}\big)^{2k-\frac{2k}{\gamma}+\frac12}(\gamma-1)^{2k-\frac{2k}{\gamma}+\frac12}. \label{E:AsEquiv3}
\end{align}
Combining \eqref{E:AsEquiv1}, \eqref{E:AsEquiv2} and \eqref{E:AsEquiv3} shows
\begin{align*}
\frac{\Gamma\big(\frac{2k}{\gamma}+1\big) \Gamma\big(2k-\frac{2k}{\gamma}+1\big)}{\Gamma(k+1)^2} \sim \big( \tfrac{2}{\gamma}\big)^{2k+1}(\gamma-1)^{2k-\frac{2k}{\gamma}+\frac12}.
\end{align*}
Consequently,
\begin{align*}
C_\sigma(\gamma,k) &\sim \big( \tfrac{2}{\gamma}\big)^{2k+1}(\gamma-1)^{2k-\frac{2k}{\gamma}+\frac12} \big( \tfrac{\gamma}{2}\big)^{2k+2} (\gamma-1)^{\frac{2k}{\gamma}-2k-1}  \\
&= \frac{\gamma}{2\sqrt{\gamma-1}}.
\end{align*}
This completes the proof of item $(e)$.
\end{proof}

\subsection{Proof of item $(d)$ in Theorem \ref{T:CgammaProps}}\label{SS:ProofOfC(gamma,k)DecreasingWithk}
We will prove that the function $C_\sigma(\gamma,k)$ decreases in the integer variable $k$ by showing that  
\begin{align}\label{E:C(gamma,k)ratio<1}
\frac{C_\sigma(\gamma,k+1)}{C_\sigma(\gamma,k)} \le 1,
\end{align}
with equality holding only in the case of $\gamma=2$.  (That equality \eqref{E:C(gamma,k)ratio<1} holds when $\gamma=2$ is item $(b)$ in Theorem \ref{T:CgammaProps}.)  Note the H\"older symmetry in the variable $\gamma$ (item $(c)$ in Theorem \ref{T:CgammaProps}) lets us reduce our investigation to $1<\gamma<2$.  After cancellation, the above ratio may be written as
\begin{align}\label{E:C(gamma,k)Ratio}
\frac{C_\sigma(\gamma,k+1)}{C_\sigma(\gamma,k)} &= \frac{\Gamma\big(\frac{2}{\gamma}(k+1) \big) \Gamma\big((2-\frac{2}{\gamma})(k+1) \big)}{\Gamma\big(\frac{2k}{\gamma} +1 \big) \Gamma\big((2-\frac{2}{\gamma} )k+1 \big)} (\gamma-1)^{\frac{2}{\gamma}-1}.
\end{align}

Now observe further symmetries that simplify the situation.  Letting $x=\frac{2}{\gamma}$ (so $1<x<2$) and taking the logarithm of \eqref{E:C(gamma,k)Ratio} leads to the definition of the function
\begin{align}\label{E:A(k,x)}
A(k,x) := \log{\frac{\Gammaf{(k+1)x}}{\Gammaf{kx+1}}} +  \log\frac{\Gammaf{(k+1)(2-x)}}{\Gammaf{k(2-x)+1}} +(x-1)\log\left(\frac{2}{x}-1\right).
\end{align}
Establishing inequality \eqref{E:C(gamma,k)ratio<1} amounts to showing that $A(k,x)<0$ for each integer $k \ge 0$ and $1<x<2$.  

Notice that $A(k,1) = 0$, so the required negativity of $A(k,x)$ will follow after it is shown that
\begin{equation}\label{E:dA/dx<0}
\frac{\dee A}{\dee x}(k,x)<0, \qquad 1<x<2.
\end{equation}
With this in mind, observe that
\begin{align*}
A(k,x) &= \log\Gamma((k+1)x) - \log\Gamma(kx+1) - (x-1)\log x \notag\\
&\qquad  + \log\Gamma((k+1)(2-x)) - \log\Gamma(k(2-x)+1) - ((2-x)-1)\log (2-x) \notag \\
&:= B(k,x) + B(k,2-x),
\end{align*}
where
\begin{align}\label{E:B(k,x)}
B(k,x) = \log{\Gamma((k+1)x)} - \log{\Gamma((kx+1))} - (x-1)\log x.
\end{align}
Therefore,
\begin{align}\label{E:dA/dx_IN_TERMS_OF_dB/dx}
\frac{\dee A}{\dee x}(k,x) &= \frac{\dee B}{\dee x}(k,x) - \frac{\dee B}{\dee x}(k,2-x).
\end{align}

Note that for $1<x<2$, the inequality $0<2-x<x$ holds.  The validity of \eqref{E:dA/dx<0} will follow from a stronger claim: that for each fixed integer $k \ge 0$, $\frac{\dee B}{\dee x}(k,x)$ is a decreasing function for all $x>0$.  This of course is equivalent to  showing
\begin{align}\label{E:d^2B/dx^2<0}
\frac{\dee^2 B}{\dee x^2}(k,x) < 0,
\end{align}
for integers $k \ge 0$ and all $x>0$.  The establishment of \eqref{E:d^2B/dx^2<0} will imply \eqref{E:dA/dx<0}, which in turn will imply \eqref{E:C(gamma,k)ratio<1} and complete the proof.  In order to proceed, we need to better understand the logarithmic derivative of the gamma function.

\subsubsection{The digamma function}
Define the digamma function $\psi$ to be the logarithmic derivative of $\Gamma(x)$, i.e.,
\begin{align*}\label{E:Digamma}
\psi(x) := \frac{\Gamma'(x)}{\Gammaf{x}},
\end{align*}
and its derivative, the trigamma function $\psi'(x)$.  Properties of digamma, trigamma and further polygamma functions (obtained by taking further derivatives) have been extensively studied in special function theory; see, for instance, \cite{AndAskRoyBOOK99}.

From \eqref{E:B(k,x)}, it is seen that
\begin{equation*}
\frac{\dee B}{\dee x}(k,x) = (k+1)\psi((k+1)x) - k\psi(kx+1) -\log{x} +\frac{1}{x} -1,
\end{equation*}
and 
\begin{equation}\label{E:d^2B/dx^2_FORMULA}
\frac{\dee^2 B}{\dee x^2}(k,x) = (k+1)^2\psi'((k+1)x) - k^2\psi'(kx+1) -\frac{1}{x} -\frac{1}{x^2}.
\end{equation}

We make use of the well known formula for $\psi'(x)$, valid for all $x$ outside of the non-positive integers:
\begin{align}\label{E:TrigammaSum}
\psi'(x) = \sum_{j=1}^\infty \frac{1}{(j+x-1)^2}.
\end{align}
Substituting \eqref{E:TrigammaSum} into \eqref{E:d^2B/dx^2_FORMULA},
\begin{align*}
\frac{\dee^2 B}{\dee x^2}(k,x) &= \sum_{j=1}^\infty \frac{(k+1)^2}{(j+(k+1)x-1)^2} - \sum_{j=1}^\infty \frac{k^2}{(j+kx)^2} - \frac{1}{x} - \frac{1}{x^2} \notag\\
&= \sum_{j=1}^\infty \left[\frac{(k+1)^2}{(j+(k+1)x)^2} - \frac{k^2}{(j+kx)^2} \right] - \frac{1}{x}. \\
&:=  D(k+1,x) - D(k,x),
\end{align*}
where
\begin{equation}\label{E:DefD(k,x)}
D(k,x) = \sum_{j=1}^\infty \frac{k^2}{(j+kx)^2} - \frac{k}{x}.
\end{equation}
Thus to establish $\frac{\dee^2 B}{\dee x^2}(k,x)<0$, it is sufficient to show that for any $x>0$, $D(k,x)$ decreases as a function of $k\ge0$.  Treating $k$ as a continuous variable and differentiating, we claim
\begin{align}\label{E:partialD/partialk<0}
\frac{\dee D}{\dee k}(k,x) = \sum_{j=1}^\infty \frac{2kj}{(kx+j)^3} - \frac{1}{x} < 0.
\end{align}

Setting $a := kx$, inequality \eqref{E:partialD/partialk<0} is a consequence of the following.

\begin{proposition}\label{P:Wijit}
For all $a>0$
\begin{equation}\label{E:EulerMacSum}
\sum_{j=0}^\infty \frac{2aj}{(a+j)^3} < 1.
\end{equation}
\end{proposition}

\begin{proof}
Observe the following estimate:
\begin{align}  
\sum_{j=1}^\infty \frac{2j}{(a+j)^3} &< \sum_{j=1}^\infty \frac{2j}{(a+j-1)(a+j)(a+j+1)}\notag\\
&=\sum_{j=1}^\infty \left[\frac{j}{(a+j-1)(a+j)} - \frac{j}{(a+j)(a+j+1)}\right]\notag \\
&=\sum_{j=1}^\infty \left[\frac{1}{a+j-1} - \frac{1}{a+j}\right] + \sum_{j=1}^\infty \left[\frac{j-1}{(a+j-1)(a+j)} - \frac{j}{(a+j)(a+j+1)}\right]\notag\\
&=\frac{1}{a}.\label{E:Wijit}
\end{align}
\end{proof}

Proposition \ref{P:Wijit} shows that \eqref{E:partialD/partialk<0} holds, implying \eqref{E:d^2B/dx^2<0}.  This gives \eqref{E:C(gamma,k)ratio<1}, completing the proof of item $(d)$ in Theorem \ref{T:CgammaProps}. \qed

\begin{remark}
Many thanks to Wijit Yangjit for showing the elegant proof to Proposition \ref{P:Wijit} the second author.  The authors' original proof of this used the Euler-Maclaurin formula (see \cite{GKPBook}) to relate the Riemann sum $\sum_{j=1}^\infty \frac{2aj}{(a+j)^3}$ to the integral $\int_{0}^\infty \frac{2aj}{(a+j)^3}\,dj = 1$. $\lozenge$
\end{remark}

\subsection{Proof of Theorem \ref{T:Norm}}\label{SS:proof-of-norm-formula}
Following Section \ref{SS:ApplicationOf FourierTransform}, we denote the Fourier transform in the $s$ variable along with its inverse by $\cf$ and $\cf^{-1}$, respectively.  These operators respect the decomposition of $L^2(M_\gamma,\sigma)$ into the orthogonal sum of the spaces $L_k^2(M_\gamma,\sigma)$.  Since each $\bm{L}_k$ is a projection operator (Theorem \ref{T:L_k-is-a-projection-sigma}), the full $\bm{L}$ is as well.  We now confirm that the full operator is bounded and calculate its norm.

Since $\cf$ and $\cf^{-1}$ are isometries,
\begin{equation*}
\norm{\bm{L}}_\sigma = \norm{\cf^{-1}\bm{L} \cf}_\sigma \qquad \textrm{and} \qquad \norm{\bm{L}_k}_\sigma = \norm{\cf^{-1}\bm{L}_k \cf}_\sigma.
\end{equation*}
In what follows, let $f(s,\alpha,\theta) = \sum_{j} f_j(s,\alpha) e^{i j \theta} \in L^2(M_\gamma,\sigma).$ 
Since $\cf^{-1}\bm{L}_k \cf$ (and thus $\bm{L}_k$) is only non-zero when acting on $L^2_k(M_\gamma,\sigma)$,
\begin{align}
\norm{\cf^{-1}\bm{L}\cf f}_\sigma^2 = \sum_{k=0}^\infty \norm{\cf^{-1}\bm{L}_k \cf f}^2_\sigma &= \sum_{k=0}^\infty \norm{\cf^{-1} \bm{L}_k \cf (f_k e^{i k \theta})}^2_\sigma \notag \\ 
&\le \sum_{k=0}^\infty \norm{\cf^{-1} \bm{L}_k \cf}^2_\sigma \norm{f_k}^2_\sigma \notag\\
&= \sum_{k=0}^\infty C_\sigma(\gamma,k) \norm{f_k}^2_\sigma \label{E:NormProof1} \\
&\le  \frac{\gamma^2}{4\gamma-4} \sum_{k=0}^\infty  \norm{f_k}^2_\sigma = \frac{\gamma^2}{4\gamma-4} \norm{f}^2_\sigma, \label{E:NormProof2}
\end{align}
where \eqref{E:NormProof1} follows from Theorem \ref{T:OpNormLk} and \eqref{E:NormProof2} follows from Theorem \ref{T:CgammaProps}.
This tells us that 
\begin{equation}\label{E:NormLerayLessThanOrEqual}
\norm{\bm{L}}_\sigma \le \frac{\gamma}{2\sqrt{\gamma-1}}.
\end{equation}

To see that equality holds in \eqref{E:NormLerayLessThanOrEqual}, let $f$ be of a multiple of $\kappa_0$, i.e.,
\begin{align}\label{E:FunctionMaxing F^-1 LF norm}
f(\xi,\alpha,\theta) &= f_0(\xi,\alpha)  \notag\\
&:=  m_0(\xi) e^{2\pi \xi (\gamma-1)\alpha^{\gamma}} \mathds{1}_{\{\xi<0\}} \in L^2_0(M_\gamma,\sigma),
\end{align}
where, by \eqref{E:KappaKMultiplierCondition-sigma}, the multiplier function $m_0$ satisfies
\begin{equation}\label{m_0-integrability-condition}
\int_{-\infty}^0 | m_0(\xi)|^2 |\xi |^{-1}\,d\xi < \infty.
\end{equation}
With this choice of $f = f_0$, equations \eqref{E:Bound1F^-1LF} and \eqref{E:C-SEquality} show that
\begin{align}
\norm{\cf^{-1}\bm{L} \cf f}_\sigma^2 &= \int_{M_\gamma}|\cf^{-1}\bm{L}_0\cf f (\xi,\alpha,\theta)|^2 \,\alpha^{\gamma-1}d\xi \wedge d\alpha \w d\theta \notag\\
&=2\pi \int_{-\infty}^\infty \int_0^\infty |\cf^{-1}\bm{L}_0\cf f (\xi,\alpha,\theta)|^2 \alpha^{\gamma-1}d\alpha\,d\xi \notag \\
&= 2\pi \int_{-\infty}^\infty |\eta_0(\xi)|^2  \norm{f_0(\xi,\cdot)}_{(\sigma,1)}^2 \norm{\kappa_0(\xi,\cdot)}_{(\sigma,1)}^2 \left(\int_0^\infty |\tau_0(\xi,\alpha)|^2 \alpha^{\gamma-1} d\alpha \right) d\xi \notag \\
&= 2\pi \int_{-\infty}^\infty |\eta_0(\xi)|^2  \norm{\kappa_0(\xi,\cdot)}_{(\sigma,1)}^2 \norm{\tau_0(\xi,\alpha)}_{(\sigma,1)}^2   \norm{f_0(\xi,\cdot)}_{(\sigma,1)}^2 d\xi \notag \\
&=2\pi \,C_\sigma(\gamma,0) \int_{-\infty}^\infty \norm{f_0(\xi,\cdot)}_{(\sigma,1)}^2 d\xi \notag\\
&= \frac{\gamma^2}{4\gamma-4} \norm{f}^2_\sigma. \notag
\end{align}
This shows the upper bound in \eqref{E:NormLerayLessThanOrEqual} is actually attained.  \qed

\begin{remark}\label{R:Leray-norm-achieved-on-infinite-dimensional-space}
The space of functions $m_0(\cdot)$ satisfying \eqref{m_0-integrability-condition} is infinite dimensional, which implies the $\sigma$-norm of the $\bm{L}$ is achieved on an infinite dimensional space; cf. Section \ref{SS:related-operators}.
$\lozenge$
\end{remark}


\section{Modified measures and adjoints}\label{S:AdjointsNormalOps}

We have seen that $\bm{L}: L^2(M_\gamma,\sigma) \to L^2(M_\gamma,\sigma)$ is bounded, where $\sigma$ is a constant multiple of the Leray-Levi measure $\lambda_\rho$.  More flexibility on measures will now be allowed.  For $r\in\R$, recall the rotationally invariant measure
\begin{equation*}
\mu_r := \alpha^{r}ds \wedge d\alpha \wedge d\theta.
\end{equation*}

\subsection{Leray boundedness}\label{SS:Leray-boundedness-m_r}

Just as in \eqref{E:SubspaceDecomp}, $S^1$-invariance leads to the decomposition
\begin{equation}
L^2(M_\gamma,\mu_r) = \bigoplus_{k=-\infty}^\infty L^2_k(M_\gamma,\mu_r),
\end{equation}
where the functions in $L_k^2(M_\gamma,\mu_r)$ take the form $f_k(s,\alpha)e^{ik\theta}$.  In the spirit of Definition \ref{D:SubLerayTransformLk}, the restriction of $\bm{L}$ to each $L_k^2(M_\gamma,\mu_r)$ yields the decomposition $\bm{L} = \bigoplus_{k=0}^\infty \bm{L}_k$.

We once again work with the unitarily equivalent $\cf^{-1}\bm{L}_k\cf$.  From Section \ref{SS:ApplicationOf FourierTransform}, recall
\begin{align}
\eta_k(\xi) &= \frac{(-2\pi\xi)^{k+1}\gamma^{k+2}}{k!} \cdot\mathds{1}_{\{\xi<0\}}, \label{D:eta_k2}\\
\tau_k(\xi,\alpha) &= \alpha^{k} e^{2\pi \xi \alpha^\gamma}\cdot \mathds{1}_{\{\xi<0\}}, \label{D:tau_k2}\\
\kappa_k(\xi,\alpha) &= \alpha^{k(\gamma-1)} e^{2\pi \xi(\gamma-1) \alpha^\gamma}\cdot \mathds{1}_{\{\xi<0\}}, \label{D:kappa_k2}
\end{align}
where $\mathds{1}_{\{\xi<0\}}$ is the indicator function of the set $\{\xi<0\}$.

Let $f (s,\alpha,\theta) =  \sum_{j} f_j(s,\alpha)e^{ij\theta} \in L^2(M_\gamma,\mu_r) $.  The form of Proposition \ref{P:LkMultipliers} remains valid, i.e.,
\begin{equation}\label{E:LkMultipliers2}
\cf^{-1}\bm{L}_k\cf f (\xi,\alpha,\theta) = \eta_k(\xi) \tau_k(\xi,\alpha) \Big< f_k(\xi,\cdot),\kappa_k(\xi,\cdot) \Big>_{(\sigma,1)} e^{ik\theta},
\end{equation}
{\em as long as the inner product makes sense.}  Recall that $\left<\cdot,\cdot\right>_{(\sigma,1)}$ is a special case of (with $r = \gamma-1$):
\begin{equation}\label{E:one-dim-inner-prod-mu_r}
\left< g,h \right>_{(\mu_r,1)} := \int_0^\infty g(\alpha)\overline{h(\alpha)} \,\alpha^{r}d\alpha.
\end{equation}

When working with $\mu_r$, we will frequently encounter its counterpart measure $\mu_{r'}$, where $r$ and $r'$ are related by
\begin{equation}\label{E:relating-r-and-r'}
\frac{r+r'}{2} = \gamma-1.
\end{equation}

The following version of Cauchy-Schwarz is crucial to the investigation of the boundedness of $\bm{L}_k$ in $L^2(M_\gamma,\mu_r)$.
\begin{lemma}\label{L:GeneralizedCauchySchwarz}
Fix $\gamma >1$ and let $r$ and $r'$ be real numbers related by $\frac{r+r'}{2} = \gamma-1$.  Then
\begin{equation*}
\left| \left< g,h \right>_{(\sigma,1)}\right| \le \norm{g}_{(\mu_r,1)} \norm{h}_{(\mu_{r'},1)},
\end{equation*}
where the inner product and norms are defined by \eqref{E:one-dim-inner-prod-mu_r}.  Equality holds if and only if there exists a constant $c$ such that $g(\alpha)\alpha^{r/2} = c\cdot h(\alpha)\alpha^{r'/2}$.  When this happens, 
\begin{equation}\label{E:GenCauchySchwarzNormEquality}
\norm{g}_{(\mu_r,1)} = |c| \cdot \norm{h}_{(\mu_{r'},1)}.
\end{equation}
\end{lemma}
\begin{proof}
Start by setting $r=\gamma-1$ in \eqref{E:one-dim-inner-prod-mu_r}.  Split the integrand into $g_1(\alpha) = g(\alpha)\alpha^{r/2}$ and $h_1(\alpha) = h(\alpha)\alpha^{r'/2}$ and apply the usual Cauchy-Schwarz inequality.
\end{proof}

To make use of Lemma \ref{L:GeneralizedCauchySchwarz}, let us generalize Corollary \ref{C:(sigma,1)-norms-of-tau-and-kappa}.  The definitions of $\tau_k$ and $\kappa_k$ in \eqref{D:tau_k2} and \eqref{D:kappa_k2} ensure that we are only interested in negative $\xi$.  Combining \eqref{E:relating-r-and-r'} and Proposition \ref{P:IntegralComputation} shows that for $\xi<0$:
\begin{align}
&\left\Vert\tau_k(\xi,\cdot)\right\Vert_{(\mu_r,1)}^2 = \int_0^\infty \alpha^{2k+r} e^{4\pi\xi\alpha^\gamma} d\alpha 
= \begin{cases}
\tfrac{1}{\gamma} (-4\pi\xi)^{-\left( \frac{2k+r+1}{\gamma}\right)} \Gamma\big(\tfrac{2k+r+1}{\gamma}\big)
&\text{ for }r > -2k-1 \\
\infty &\text{ otherwise. }
\end{cases}
&\label{E:TauKNormMu_r} \\
&\left\Vert\kappa_k(\xi,\cdot)\right\Vert_{(\mu_{r'},1)}^2 = \int_0^\infty \alpha^{(2k+2)(\gamma-1)-r} e^{4\pi\xi(\gamma-1)\alpha^\gamma} d\alpha \label{E:KappaKNormMu_r'} \\
&= \begin{cases}
\tfrac{1}{\gamma} (-4\pi\xi(\gamma-1))^{-\left( \frac{(2k+2)(\gamma-1)-r+1}{\gamma}\right)} \Gamma\left(\tfrac{(2k+2)(\gamma-1)-r+1}{\gamma}\right)
&\text{ for }r < (2k+2)(\gamma-1)+1 \\
\infty &\text{ otherwise. }
\end{cases}
\notag
\end{align}
These computations suggest the definition of the interval
\begin{equation}\label{D:def-interval-Ik}
\ci_k = (-2k-1 , (2k+2)(\gamma-1)+1).
\end{equation}

We are now able to prove the general boundedness result stated in Theorem \ref{T:Lk-boundedness}.
\begin{theorem}\label{T:leray-boundedness-mu_r}
Let $k \ge 0$ be an integer and $r\in \R$.  Then $\bm{L}_k$ is bounded from $L^2(M_\gamma,\mu_r) \to L^2_k(M_\gamma,\mu_r)$ if and only if $r \in \ci_k$.  Furthermore, when $r \in \ci_k$,
\begin{equation}\label{E:leray-norm-mu_r}
\norm{\bm{L}_k}_{L^2(M_\gamma,\mu_r)} = \sqrt{C_{\mu_r}(\gamma,k)},
\end{equation}
where the $\mu_r$-symbol function is
\begin{equation}\label{E:symbol-function-mu_r}
C_{\mu_r}(\gamma,k) = \frac{\Gamma\big(\frac{2k+1+r}{\gamma}\big) \Gamma\big(2k+2 - \frac{2k+1+r}{\gamma}\big)}{\Gamma(k+1)^2} \left( \tfrac{\gamma}{2}\right)^{2k+2} (\gamma-1)^{-\left( 2k+2 - \frac{2k+1+r}{\gamma} \right)}.
\end{equation}
\end{theorem}

\begin{proof} Consider three cases: $(a)$ $r \in \ci_k$, $(b)$ $r \le -2k-1$ and $(c)$ $r \ge (2k+2)(\gamma-1)+1$.

{$(a)$:  Let $r \in \ci_k$} and, without loss of generality, assume $f =  f_k(\xi,\alpha)e^{ik\theta} \in L_k^2(M_\gamma,\mu_r)$. From \eqref{E:LkMultipliers2},
\begin{align}
\norm{\cf^{-1}\bm{L}_k\cf f}^2_{L^2(M_\gamma,\mu_r)} &= \int_{M_\gamma} \left| \eta_k(\xi) \tau_k(\xi,\alpha) \Big< f_k(\xi,\cdot),\kappa_k(\xi,\cdot) \Big>_{(\sigma,1)} \,e^{ik\theta} \right|^2 \mu_r \notag \\
&= 2\pi \int_{-\infty}^0 \int_0^\infty \left| \eta_k(\xi) \tau_k(\xi,\alpha) \Big< f_k(\xi,\cdot),\kappa_k(\xi,\cdot) \Big>_{(\sigma,1)}\right|^2 \alpha^r\, d\alpha\, d\xi \notag \\
&= 2\pi \int_{-\infty}^0 \left| \eta_k(\xi) \right|^2 \left| \Big< f_k(\xi,\cdot),\kappa_k(\xi,\cdot) \Big>_{(\sigma,1)}\right|^2 \int_0^\infty  \left| \tau_k(\xi,\alpha) \right|^2  \alpha^r\, d\alpha\, d\xi \label{E:rearrangement-mu_r-norm-calculation} \\
&\le 2\pi \int_{-\infty}^0 \left| \eta_k(\xi) \right|^2 \norm{\tau_k(\xi,\cdot)}_{(\mu_r,1)}^2 \norm{\kappa_k(\xi,\cdot)}_{(\mu_{r'},1)}^2  \norm{f_k(\xi,\cdot)}_{(\mu_r,1)}^2 \,d\xi, \label{E:AppOfGenCS1}
\end{align}
where \eqref{E:AppOfGenCS1} follows from Lemma \ref{L:GeneralizedCauchySchwarz}.  The norms appearing in the last integral are finite since $r \in \ci_k$.  By combining \eqref{D:eta_k2}, \eqref{E:TauKNormMu_r} and \eqref{E:KappaKNormMu_r'} it is easily verified that the quantity 
$$
\left| \eta_k(\xi) \right|^2 \norm{\tau_k(\xi,\cdot)}_{(\mu_r,1)}^2  \norm{\kappa_k(\xi,\cdot)}_{(\mu_{r'},1)}^2
$$
appearing in the integrand of \eqref{E:AppOfGenCS1} equals $C_{\mu_r}(\gamma,k)$.  Note that this more general symbol function is once again independent of $\xi$.  Returning now to \eqref{E:AppOfGenCS1},
\begin{align}
\norm{\cf^{-1}\bm{L}_k\cf f}^2_{L^2(M_\gamma,\mu_r)} &\le 2\pi \,C_{\mu_r}(\gamma,k) \int_{-\infty}^0 \norm{f_k(\xi,\cdot)}_{(\mu_r,1)}^2\,d\xi \label{E:UpperBoundOnLkMu_r} \\
&=\,C_{\mu_r}(\gamma,k) \int_0^{2\pi} \int_{-\infty}^0  \int_0^\infty |f_k(\xi,\alpha)|^2 \alpha^r d\alpha \,d\xi\,  d\theta\notag\\
&= C_{\mu_r}(\gamma,k) \norm{f}_{L^2(M_\gamma,\mu_r)}^2. \notag
\end{align}

Thus, $\sqrt{C_{\mu_r}(\gamma,k)}$ is an upper bound on $\norm{\bm{L}_k}_{L^2(M_\gamma,\mu_r)}$.  We now show this to be sharp.  Lemma \ref{L:GeneralizedCauchySchwarz} says that equality in \eqref{E:AppOfGenCS1} holds if and only if 
\begin{equation}\label{E:mu_r-norm-acieving-condition1}
f_k(\xi,\alpha)\alpha^{r/2} = m_k(\xi) \kappa_k(\xi,\alpha) \alpha^{r'/2},
\end{equation}
where $m_k(\xi)$ must be chosen so that $f(\xi,\alpha,\theta) = f_k(\xi,\alpha)e^{i k \theta} \in L^2(M_\gamma,\mu_r)$.  This happens if and only if
\begin{equation}\label{E:mu_r-norm-acieving-condition2}
\int_{-\infty}^0 |m_k(\xi)|^2 |\xi|^{\left(\frac{2k+1+r}{\gamma} -2k -2 \right)}d\xi < \infty.
\end{equation}
These conditions generalize \eqref{E:C-SEquality} and \eqref{E:KappaKMultiplierCondition-sigma}.  When $f$ takes this form equality in \eqref{E:UpperBoundOnLkMu_r} must hold.  This establishes \eqref{E:leray-norm-mu_r} and concludes case $(a)$.
\smallskip

{$(b)$:  Let $r \le -2k -1$}.  Combining \eqref{E:LkMultipliers2} and \eqref{E:rearrangement-mu_r-norm-calculation}, we see that for any $f(\xi,\alpha,\theta) = f_k(\xi,\alpha) e^{i k \theta} \in L^2_k(M_\gamma,\mu_r)$ with $\langle f_k(\xi,\cdot),\kappa_k(\xi,\cdot) \rangle_{(\sigma,1)} \neq 0$, $\cf^{-1}\bm{L}_k\cf f \notin L^2(M_\gamma,\mu_r)$, since $\norm{\tau_k(\xi,\cdot)}_{(\mu_r,1)}^2 = \infty$ by \eqref{E:TauKNormMu_r}.
\smallskip

{$(c)$:  Let $r \ge (2k+2)(\gamma-1)+1$}.  For $j\in \Z^+$ set 
\begin{equation*}
\kappa_k^j(\xi,\alpha) := \kappa_k(\xi,\alpha)  \cdot \mathds{1}_{\{\frac{1}{j}<\alpha \}},
\end{equation*}
where $\mathds{1}_{\{\frac{1}{j}<\alpha \}}$ is the indicator function for the set $\{(\xi,\alpha,\theta):\frac{1}{j}<\alpha \}$.  Now define
\begin{equation*}
f_k^j(\xi,\alpha) := \tfrac{\kappa_k^j(\xi,\alpha)\alpha^{\gamma-1-r}}{\norm{\kappa_k^j(\xi,\cdot)}_{(\mu_{r'},1)}} \cdot \mathds{1}_{\{-\frac{1}{2\pi} < \xi < 0 \}} \qquad \textrm{and} \qquad f^j(\xi,\alpha,\theta) := f_k^j(\xi,\alpha) e^{i k \theta}.
\end{equation*}
It is easily shown that for any $j$,
\begin{equation*}
\norm{f_k^j(\xi, \cdot) }_{(\mu_r,1)} = \mathds{1}_{\{-\frac{1}{2\pi} < \xi < 0 \}}.
\end{equation*}
Therefore
\begin{equation}\label{E:the-norm-of-f^j-is-equal-to-1}
\norm{f^j}_{L^2(M_\gamma,\mu_r)} = \left( \int_0^{2\pi} \int_{-\infty}^\infty \norm{f_k^j(\xi, \cdot)}^2_{(\mu_r,1)} \,d\xi\,d\theta \right)^{\frac12} = 1.
\end{equation}
From \eqref{E:LkMultipliers2},
\begin{align*}
\cf^{-1}\bm{L}_k\cf f^j(\xi,\alpha,\theta) &= \eta_k(\xi) \tau_k(\xi,\alpha) \Big< f^j_k(\xi,\cdot),\kappa_k(\xi,\cdot) \Big>_{(\sigma,1)} e^{i k \theta} \\
&= \mathds{1}_{\{-\frac{1}{2\pi} < \xi < 0 \}}  \cdot \eta_k(\xi) \tau_k(\xi,\alpha) \norm{\kappa_k^j(\xi,\cdot)}_{(\mu_{r'},1)} e^{i k \theta},
\end{align*}
and consequently
\begin{align*}
\norm{\cf^{-1}\bm{L}_k\cf f^j}_{L^2(M_\gamma,\mu_r)}^2 &= 2\pi \int_{-\frac{1}{2\pi}}^0 |\eta_k(\xi)|^2 \norm{\tau_k(\xi,\cdot)}^2_{(\mu_r,1)} \norm{\kappa_k^j(\xi,\cdot)}^2_{(\mu_{r'},1)} \,d\xi.
\end{align*}
But the monotone convergence theorem says
\begin{align*}
\lim_{j\to\infty} \norm{\kappa_k^j(\xi,\cdot)}^2_{(\mu_{r'},1)} &= \lim_{j\to\infty} \int_{\frac{1}{j}}^\infty |\kappa_k(\xi,\alpha)|^2 \alpha^{r'}\,d\alpha \\
&= \lim_{j\to\infty} \int_{\frac{1}{j}}^\infty  \alpha^{(2k+2)(\gamma-1)-r} e^{4\pi\xi(\gamma-1)\alpha^\gamma}  \,d\alpha \\
&= \infty
\end{align*}
for any $-\frac{1}{2\pi}<\xi<0$.  This means
\begin{equation*}
\lim_{j\to\infty} \norm{\cf^{-1}\bm{L}_k\cf f^j}_{L^2(M_\gamma,\mu_r)} = \infty,
\end{equation*}
which together with \eqref{E:the-norm-of-f^j-is-equal-to-1} shows that $\bm{L}_k$ fails to be bounded.  This completes $(c)$.
\end{proof}

Fix any $r \in \R$.  Despite the limited boundedness range in Theorem \ref{T:leray-boundedness-mu_r},  the interval $\ci_k$ tends to $(-\infty,\infty)$ as $k\to\infty$.  This means that the tail of the sequence of operators $\{\bm{L}_k\}$ is always bounded in the $L^2(M_\gamma,\mu_r)$ norm.  In fact, a much stronger result holds:

\begin{theorem} \label{T:grade-essential-for-F}
Fix $r\in\R$.  The high frequency limit norm of $\bm{L}$ in $L^2(M_\gamma,\mu_r)$ is given by the ($r$-independent) quantity
\begin{equation}\label{E:grade-essential-norm-M_gamma}
\norm{\bm{L}}_{L_{HF}^2(M_\gamma,\mu_r)} = \lim_{k \to \infty} \norm{\bm{L}_k}_{L^2(M_\gamma,\mu_r)} = \sqrt{\frac{\gamma}{2\sqrt{\gamma-1}}}.
\end{equation}
\end{theorem}
\begin{proof}
For $k$ large enough, Theorem \ref{T:leray-boundedness-mu_r} says
\begin{equation}\label{E:norm-squared-of-Lk-wrt-mu_r}
\norm{\bm{L}_k}_{\mu_r}^2 = \frac{\Gamma\big(\frac{2k+1+r}{\gamma}\big) \Gamma\big(2k+2 - \frac{2k+1+r}{\gamma}\big)}{\Gamma(k+1)^2} \left( \tfrac{\gamma}{2}\right)^{2k+2} (\gamma-1)^{-\left( 2k+2 - \frac{2k+1+r}{\gamma} \right)}.
\end{equation}
First focus on the following quotient:
\begin{equation}\label{E:GammaFunctionQuotientC_mu_r}
\frac{\Gamma\big(\frac{2k+1+r}{\gamma}\big) \Gamma\big(2k+2 - \frac{2k+1+r}{\gamma}\big)}{\Gamma(k+1)^2}.
\end{equation}
Sterling's formula says
\begin{align*}
\Gamma({k+1})^2 &\sim 2\pi e^{-2k} k^{2k+1}, \\
\Gamma\big(\tfrac{2k+1+r}{\gamma}\big) &\sim \sqrt{2\pi}\cdot e^{\left(1-\frac{2k +1 + r}{\gamma}\right)} \left( \tfrac{2k+1+r}{\gamma} -1 \right)^{\left(\frac{2k+1+r}{\gamma} - \frac{1}{2} \right)}, \\
\Gamma\big(2k+2 -\tfrac{2k+1+r}{\gamma}\big) &\sim \sqrt{2\pi}\cdot e^{\left( -2k -1+\frac{2k+1+r}{\gamma}\right)} \left( 2k+1 -\tfrac{2k+1+r}{\gamma}\right)^{\left(2k+\frac{3}{2} -\frac{2k+1+r}{\gamma}\right)}.
\end{align*}
Combining these asymptotic equivalences yields 
\begin{align}
\eqref{E:GammaFunctionQuotientC_mu_r} &\sim \frac{1}{k^{2k+1}}\left( \tfrac{2k+1+r}{\gamma} - 1\right)^{\left(\frac{2k+1+r}{\gamma} - \frac{1}{2}\right)}  \left( 2k+1 -\tfrac{2k+1+r}{\gamma}\right)^{\left(2k+\frac{3}{2} -\frac{2k+1+r}{\gamma}\right)} \notag \\
&= \frac{1}{\gamma^{2k+1}} \cdot \left(2 + \tfrac{1+r-\gamma}{k}\right)^{\left(\frac{2k+1+r}{\gamma} - \frac{1}{2}\right)} \left( 2\gamma-2 + \tfrac{\gamma-1-r}{k}  \right)^{\left(2k+\frac{3}{2} -\frac{2k+1+r}{\gamma}\right)} \notag \\
&\sim \frac{1}{\gamma^{2k+1}} \cdot 2^{\left(\frac{2k+1+r}{\gamma} - \frac{1}{2}\right)} \left( 2\gamma-2 \right)^{\left(2k+\frac{3}{2} -\frac{2k+1+r}{\gamma}\right)} \notag \\
&=  \left(\tfrac{2}{\gamma}\right)^{2k+1} \left(\gamma-1 \right)^{\left(2k+\frac{3}{2} -\frac{2k+1+r}{\gamma}\right)}. \label{E:GammaQuotientAsympEqiv}
\end{align}
Now combine \eqref{E:norm-squared-of-Lk-wrt-mu_r} and \eqref{E:GammaQuotientAsympEqiv} to see that
\begin{equation}
\norm{\bm{L}_k}_{\mu_r}^2 \sim \frac{\gamma}{2\sqrt{\gamma-1}},
\end{equation}
which completes the proof.
\end{proof}

We now establish the sharp interval of Leray boundedness as stated in Theorem 1.4.

\begin{corollary}\label{C:r-cond}
The Leray transform $\bm{L}$ is bounded from $L^2(M_\gamma,\mu_r) \to L^2(M_\gamma,\mu_r)$ if and only if $r \in \ci_0 = (-1,2\gamma-1)$.
\end{corollary}
\begin{proof}
If $r \in \ci_0$, then each $\bm{L}_k$ is bounded on $L^2(M_\gamma,\mu_r)$ since $\ci_0 \subset \ci_k$ for all $k\in\Z^+$.  And since $\lim_{k\to\infty} \norm{\bm{L}_k}_{\mu_r}$ is finite, we conclude that there is some $C>0$ with $\norm{\bm{L}_k}_{\mu_r} \le C$ for all $k$.  Let $f = \sum_k f_k e^{i k \theta}$.  Then,
\begin{equation*}
\norm{\bm{L}f}_{\mu_r}^2 = \sum_{k=0}^\infty \norm{\bm{L}_k f}_{\mu_r}^2 \le \sum_{k=0}^\infty \norm{\bm{L}_k}_{\mu_r}^2 \norm{f_k}_{\mu_r}^2 \le C^2 \sum_{k=0}^\infty \norm{f_k}_{\mu_r}^2 = C^2 \norm{f}_{\mu_r}^2.
\end{equation*}
On the other hand if $r \notin \ci_0$, then $\bm{L}_0$ is unbounded, implying the same for $\bm{L}$.
\end{proof}

\begin{remark}\label{R:symmetry-of-interval-I_0}
Notice that $r_0 = \gamma-1$ (the $r$-value for the measure $\sigma$) lies exactly at the midpoint the interval $\ci_0$.  This says that if $\frac{r+r'}{2} = \gamma -1$, then $r \in \ci_0$ if and only if $r' \in \ci_0$. In other words, $\bm{L}$ is bounded on $L^2(M_\gamma,\mu_r)$ if and only if it is bounded on $L^2(M_\gamma,\mu_{r'})$.  $\lozenge$
\end{remark}

\subsection{Adjoints}\label{SS:Leray-adjoints}
Theorem \ref{T:leray-boundedness-mu_r} says $\bm{L}_k$ is bounded on $L^2_k(M_\gamma,\mu_r)$ when $r \in \ci_k$.  It therefore admits a bounded adjoint $\bm{L}_k^{(*,\mu_r)}$ which satisfies
\begin{align}
\Big<\bm{L}_k f,g \Big>_{\mu_r} = \Big< f,\bm{L}_k^{(*,\mu_r)} g \Big>_{\mu_r}.
\end{align}
We wish to describe this adjoint more explicitly.

The following lemma justifies later applications of Fubini's theorem.

\begin{lemma}\label{L:justification-of-fubini:adjoint-comp}
Choose $r$ and $k$ so that $r \in \ci_k$, and take $\eta_k, \tau_k,\kappa_k$ as in \eqref{D:eta_k2}, \eqref{D:tau_k2}, \eqref{D:kappa_k2}.  

\noindent(a) Let $f_k(\xi,\alpha) e^{i k \theta}, g_k(\xi,\alpha) e^{i k \theta} \in L^2_k(M_\gamma,\mu_r)$.  Then the integral
\begin{align*}
I_1 &:= \int_0^{2\pi} \int_{-\infty}^\infty \int_0^\infty \int_0^\infty \Big| \eta_k(\xi)\tau_k(\xi,\alpha_z) \kappa_k(\xi,\alpha_\zeta) f_k(\xi,\alpha_\zeta) g_k(\xi,\alpha_z) \Big|  \alpha_\zeta^{\gamma-1} \alpha_z^{r} \, d\alpha_z \,d\alpha_\zeta\, d\xi\, d\theta \\
&\le \sqrt{C_{\mu_r}(\gamma,k)} \norm{f_k}_{\mu_r} \norm{g_k}_{\mu_r}.
\end{align*}

\noindent(b) Let $\phi_k(\xi,\alpha) e^{i k \theta}  \in L^2_k(M_\gamma,\mu_{r})$ and  $\psi_k(\xi,\alpha) e^{i k \theta} \in L^2_k(M_\gamma,\mu_{r'})$.  Then the integral
\begin{align*}
I_2 &:= \int_0^{2\pi} \int_{-\infty}^\infty \int_0^\infty \int_0^\infty \Big| \eta_k(\xi)\tau_k(\xi,\alpha_z) \kappa_k(\xi,\alpha_\zeta) \phi_k(\xi,\alpha_\zeta) \psi_k(\xi,\alpha_z) \Big|  \alpha_\zeta^{\gamma-1} \alpha_z^{\gamma-1} \, d\alpha_z \,d\alpha_\zeta\, d\xi\, d\theta \\
&\le \sqrt{C_{\mu_r}(\gamma,k)} \norm{\phi_k}_{\mu_r} \norm{\psi_k}_{\mu_{r'}}.
\end{align*}

\end{lemma}
\begin{proof}
Recall that $\eta_k, \tau_k,\kappa_k$ are everywhere non-negative and identically $0$ for $\xi \ge 0$.  Write
\begin{align*}
I_1 &= 2\pi \int_{-\infty}^0 \eta_k(\xi)  \left[ \int_0^\infty  |g_k(\xi,\alpha_z)| \tau_k(\xi,\alpha_z) \alpha_z^{r} d\alpha_z \right]\left[ \int_0^\infty \kappa_k(\xi,\alpha_\zeta) |f_k(\xi,\alpha_\zeta) |  \alpha_\zeta^{\gamma-1} d\alpha_\zeta \right] d\xi \notag \\
&\le 2\pi \int_{-\infty}^0 \eta_k(\xi)   \norm{g_k(\xi,\cdot)}_{(\mu_r,1)} \norm{\tau_k(\xi,\cdot)}_{(\mu_{r},1)} \norm{f_k(\xi,\cdot)}_{(\mu_r,1)}  \norm{\kappa_k(\xi,\cdot)}_{(\mu_{r'},1)} d\xi \notag \\
&= 2\pi \sqrt{C_{\mu_r}(\gamma,k)} \int_{-\infty}^0  \norm{g_k(\xi,\cdot)}_{(\mu_r,1)} \norm{f_k(\xi,\cdot)}_{(\mu_r,1)} d\xi  \notag \\
& \le \sqrt{C_{\mu_r}(\gamma,k)} \norm{f_k}_{\mu_r} \norm{g_k}_{\mu_r}. \notag
\end{align*}
In the first inequality we used both the standard version of Cauchy-Schwarz and the variant from Lemma \ref{L:GeneralizedCauchySchwarz}.  Also, recall that $ \eta_k(\xi)  \norm{\tau_k(\xi,\cdot)}_{(\mu_r,1)}  \norm{\kappa_k(\xi,\cdot)}_{(\mu_{r'},1)}$ is equal to the ($\xi$-independent) term $\sqrt{C_{\mu_r}(\gamma,k)}$ encountered in Theorem \ref{T:leray-boundedness-mu_r}.  

Similarly,
\begin{align*}
I_2 &= 2\pi \int_{-\infty}^0 \eta_k(\xi)  \left[ \int_0^\infty  |\psi_k(\xi,\alpha_z)| \tau_k(\xi,\alpha_z) \alpha_z^{\gamma-1} d\alpha_z \right]\left[ \int_0^\infty \kappa_k(\xi,\alpha_\zeta) |\phi_k(\xi,\alpha_\zeta) |  \alpha_\zeta^{\gamma-1} d\alpha_\zeta \right] d\xi \notag \\
&\le 2\pi \int_{-\infty}^0 \eta_k(\xi)   \norm{\psi_k(\xi,\cdot)}_{(\mu_{r'},1)} \norm{\tau_k(\xi,\cdot)}_{(\mu_{r},1)} \norm{\phi_k(\xi,\cdot)}_{(\mu_r,1)}  \norm{\kappa_k(\xi,\cdot)}_{(\mu_{r'},1)} d\xi \notag \\
&= 2\pi \sqrt{C_{\mu_r}(\gamma,k)} \int_{-\infty}^0  \norm{\psi_k(\xi,\cdot)}_{(\mu_{r'},1)} \norm{\phi_k(\xi,\cdot)}_{(\mu_r,1)} d\xi  \notag \\
& \le \sqrt{C_{\mu_r}(\gamma,k)} \norm{\phi_k}_{\mu_r} \norm{\psi_k}_{\mu_{r'}}, \notag
\end{align*}
completing the proof.
\end{proof}

We now compute a formula for $(\cf^{-1} \bm{L}_k \cf)^{(*,\mu_r)} = \cf^{-1} \bm{L}_k^{(*,\mu_r)} \cf$ when $r \in \ci_k$:

\begin{proposition}\label{P:computation-of-adjoint-by-soft-analysis}
Choose $k$ and $r$ so that $r \in \ci_k$.  Then if $g(\xi,\alpha,\theta) = \sum_j g_j(\xi,\alpha) e^{ij\theta} \in L^2(M_\gamma,\mu_r),$
\begin{equation}\label{E:fourier-transformed-adjoint-formula}
\cf^{-1}\bm{L}^{(*,\mu_r)}_k\cf g (\xi,\alpha,\theta) = \alpha^{\gamma-1-r}\eta_k(\xi) \kappa_k(\xi,\alpha)  \Big< g_k(\xi,\cdot),\tau_k(\xi,\cdot) \Big>_{(\mu_r,1)}  e^{ik\theta},
\end{equation}
where $\eta_k, \tau_k,\kappa_k$ are the real valued functions given by \eqref{D:eta_k2}, \eqref{D:tau_k2}, \eqref{D:kappa_k2}, respectively.
\end{proposition}
\begin{proof}
For $r \in \ci_k$, the boundedness of $\cf^{-1} \bm{L}_k \cf$ shows there exists an adjoint operator $\cf^{-1 }\bm{L}_k^{(*,\mu_r)} \cf$ satisfying 
\begin{equation}
\Big< \cf^{-1 }\bm{L}_k \cf f,g \Big>_{\mu_r} = \Big< f, \cf^{-1 }\bm{L}_k^{(*,\mu_r)} \cf g \Big>_{\mu_r}
\end{equation}
 for all $f,g \in L^2(M_\gamma,\mu_r)$.  Without loss of generality, assume that $f(\xi,\alpha,\theta) =f_k(\xi,\alpha) e^{ik\theta}$.  
 
 Lemma \ref{L:justification-of-fubini:adjoint-comp} part (a) justifies the use of Fubini's theorem in what follows:
\begin{align}
\Big<\cf^{-1} \bm{L}_k \cf f,g \Big>_{\mu_r} &= \int_{M_\gamma} \eta_k(\xi) \tau_k(\xi,\alpha_z) \Big< f_k(\xi,\cdot),\kappa_k(\xi,\cdot) \Big>_{(\sigma,1)} \,e^{ik\theta} \overline{g_k(\xi,\alpha_z)e^{i k \theta}}\, \mu_r(z) \notag \\
&= \int_{M_\gamma} \eta_k(\xi) \tau_k(\xi,\alpha_z) \overline{g_k(\xi,\alpha_z)} \int_0^\infty f_k(\xi,\alpha_\zeta)\kappa_k(\xi,\alpha_\zeta)\alpha_\zeta^{\gamma-1}d\alpha_\zeta \alpha_z^{r} d\xi\,d\alpha_z\,d\theta \notag\\
&= \int_{M_\gamma}  f_k(\xi,\alpha_\zeta) \alpha_\zeta^{\gamma-1} \eta_k(\xi) \kappa_k(\xi,\alpha_\zeta) \int_0^\infty  \overline{g_k(\xi,\alpha_z)} \tau_k(\xi,\alpha_z)   \alpha_z^{r} d\alpha_z  d\xi\,d\alpha_\zeta\,d\theta \notag \\
&= \int_{M_\gamma}  f_k(\xi,\alpha_\zeta) e^{i k \theta} \overline{\alpha_\zeta^{\gamma-1-r}  \eta_k(\xi) \kappa_k(\xi,\alpha_\zeta)  \Big< g_k(\xi,\cdot),\tau_k(\xi,\cdot) \Big>_{(\mu_r,1)} e^{i k \theta}}  \mu_r(\zeta) \notag \\
&:= \Big<f, \cf^{-1} \bm{L}_k^{(*,\mu_r)} \cf g \Big>_{\mu_r}, \label{E:fourier-transformed-adjoint-comp}
\end{align}
completing the proof.
\end{proof}

We now obtain a formula for the adjoint of the full operator $\bm{L}$ in the space $L^2(M_\gamma,\sigma)$.

\begin{proposition}\label{P:leray-adjoint-wrt-sigma}
The adjoint of the Leray transform with respect to the inner product $\langle \cdot,\cdot \rangle_\sigma$ is given by the following integral:
\begin{equation}\label{E:leray-adjoint-wrt-sigma}
\bm{L}^{(*,\sigma)} g(z) = \frac{\gamma^2}{8\pi^2 i} \int_{M_\gamma}   g(\zeta)\frac{|\zeta_1|^{\gamma-2}d\zeta_2\wedge d\bar{\zeta}_1\wedge d\zeta_1}{\left[\gamma {z}_1|z_1|^{\gamma-2}(\bar{z}_1-\bar{\zeta}_1)-i(\bar{z}_2-\bar{\zeta}_2)\right]^2}.
\end{equation}
\end{proposition}
\begin{proof}
Define an integral operator $\bm{T}$, where $\bm{T}g(z)$ is equal to the right hand side of \eqref{E:leray-adjoint-wrt-sigma} for all choices of $g$ for which this integral converges.  Referring back to \eqref{E:ell_M(z,zeta)} and \eqref{E:LerayLevi}, we see that $\bm{T}$ can be expressed as 
\begin{equation*}
\bm{T}g(z) = \int_{M_\gamma} \bar{\ell_\rho(\zeta,z)}\lambda_\rho(\zeta).
\end{equation*}

Revisit Section \ref{SS:LerayReparam}, but this time start with $\bm{T}$ and carry out the reparametrization \eqref{E:z-reparam} and \eqref{E:zeta-reparam} together with all subsequent steps through Section \ref{SS:SeriesExpn}.  In particular, $\bm{T}$ admits an orthogonal decomposition into $\bigoplus_{k=0}^\infty \bm{T}_k$.

Now for each $\bm{T}_k$ (following the outline in Section \ref{SS:ApplicationOf FourierTransform}), conjugate by the Fourier transform to obtain the operator $\cf^{-1} \bm{T}_k \cf$.  This leads to a result in the spirit of Proposition \ref{P:LkMultipliers}:
\begin{equation}\label{E:operator-T_k-conj-by-Fourier-trans}
\cf^{-1}\bm{T}_k\cf g (\xi,\alpha,\theta) = \eta_k(\xi) \kappa_k(\xi,\alpha)  \Big< g_k(\xi,\cdot),\tau_k(\xi,\cdot) \Big>_{(\sigma,1)} e^{ik\theta}.
\end{equation}

Since $\sigma = \mu_{\gamma-1}$, we see that \eqref{E:operator-T_k-conj-by-Fourier-trans} is exactly \eqref{E:fourier-transformed-adjoint-formula} when $r=\gamma-1$.  This says that for each $k$, $\cf^{-1} \bm{T}_k \cf = \cf^{-1} \bm{L}_k^{(*,\sigma)}\cf$, implying each $\bm{T}_k = \bm{L}_k^{(*,\sigma)}$ and therefore $\bm{T} = \bm{L}^{(*,\sigma)}$.
\end{proof}

\begin{lemma}\label{L:correspondence-between-r-and-r'-norms}
Let $\frac{r+r'}{2} = \gamma-1$.  The map $\bm{R}_r: L^2(M_\gamma,\mu_{r'}) \to L^2(M_\gamma,\mu_r)$ given by
\begin{equation}\label{D:def-of-R_r}
g(\xi,\alpha,\theta) \longmapsto \alpha^{\gamma-1-r}g(\xi,\alpha,\theta)
\end{equation}
is an isometry.
\end{lemma}
\begin{proof}
This is essentially a reformulation of the condition which tells when equality holds in Lemma \ref{L:GeneralizedCauchySchwarz}.  Indeed,
\begin{align*}
\norm{\bm{R}_r g}_{\mu_{r}}^2 = \int_{M_\gamma} |\alpha^{\gamma-1-r}g(\xi,\alpha,\theta)|^2 \alpha^r\,d\xi\,d\alpha\,d\theta = \int_{M_\gamma} |g(\xi,\alpha,\theta)|^2 \alpha^{r'}\,d\xi\,d\alpha\,d\theta  = \norm{g}_{\mu_{r'}}^2.
\end{align*}
It is clear this map is bijective with inverse $\bm{R}_r^{-1}$ given by $h(\xi,\alpha,\theta) \mapsto \alpha^{r+1-\gamma} h(\xi,\alpha,\theta)$.
\end{proof}

For $r \in \ci_0$, we now give an explicit formula for the more general adjoint $\bm{L}^{(*,\mu_r)}$:

\begin{theorem}\label{T:leray-adjoint-wrt-mu_r}
Let $r \in \ci_0 = (-1,2\gamma-1)$.  The adjoint of the Leray transform in the inner product $\langle \cdot,\cdot\rangle_{\mu_r}$ is found by computing the following integral
\begin{equation}\label{E:leray-adjoint-wrt-mu_r}
\bm{L}^{(*,\mu_r)} g(z) = \frac{\gamma^2 |z_1|^{\gamma-1-r}}{8\pi^2 i} \int_{M_\gamma}   g(\zeta)\frac{|\zeta_1|^{r-1}d\zeta_2\wedge d\bar{\zeta}_1\wedge d\zeta_1}{\left[\gamma {z}_1|z_1|^{\gamma-2}(\bar{z}_1-\bar{\zeta}_1)-i(\bar{z}_2-\bar{\zeta}_2)\right]^2}.
\end{equation}
\end{theorem}
\begin{proof}
Let $r \in \ci_0$, $f\in L^2(M_\gamma,\mu_r)$ and $g \in L^2(M_\gamma,\mu_{r'})$. Remark \ref{R:symmetry-of-interval-I_0} guarantees $r' \in \ci_0$, so $\bm{L}$ is bounded in both of these function spaces.  

The Fourier transform $\cf$ (acting in the variable $s = \re(\zeta_2)$) is obiviously an isometry of both $L^2(M_\gamma,\mu_r)$ and $L^2(M_\gamma,\mu_{r'})$.  Now write $f = \cf(\phi), g = \cf(\psi)$ for $\phi = \sum_j \phi_j e^{ij\theta} \in L^2(M_\gamma,\mu_{r})$, $\psi = \sum_j \psi_k e^{ij\theta}  \in L^2(M_\gamma,\mu_{r'})$.  On one hand,
\begin{equation}\label{E:mu_r-adoint-comp-1}
\langle \bm{L} f, g \rangle_{\sigma} = \big\langle \bm{L} \cf(\phi), \cf(\psi) \big\rangle_{\sigma} = \big\langle \cf^{-1 }\bm{L} \cf(\phi), \psi \big\rangle_{\sigma} = \sum_{k=0}^\infty \big\langle \cf^{-1 }\bm{L}_k \cf(\phi), \psi \big\rangle_{\sigma},
\end{equation}
where the summand can be re-written as
\begin{align}
\big\langle \cf^{-1 }\bm{L}_k \cf(\phi), \psi \big\rangle_{\sigma} &= \int_{M_\gamma} \eta_k(\xi) \tau_k(\xi,\alpha_z) \Big< \phi_k(\xi,\cdot),\kappa_k(\xi,\cdot) \Big>_{(\sigma,1)} \bar{\psi_k(\xi,\alpha_z)} \, \alpha_z^{\gamma-1}d\xi\,d\alpha_z\,d\theta \notag \\
&= \int_{M_\gamma} \phi_k(\xi,\alpha_z) \bar{\eta_k(\xi) \kappa_k(\xi,\alpha_\zeta) \Big< \psi_k(\xi,\cdot),\tau_k(\xi,\cdot) \Big>_{(\sigma,1)}}  \, \alpha_\zeta^{\gamma-1}d\xi\,d\alpha_\zeta\,d\theta \label{E:mu_r-adoint-comp-2}\\
&= \big\langle \phi, \cf^{-1 }\bm{L}_k^{(*,\sigma)} \cf(\psi) \big\rangle_{\sigma}. \notag
\end{align}
The rearrangement in \eqref{E:mu_r-adoint-comp-2} is justified by Lemma \ref{L:justification-of-fubini:adjoint-comp} part (b) together with the fact that $\eta_k,\tau_k$ and $\kappa_k$ are real valued.  Thus,
\begin{align}
\eqref{E:mu_r-adoint-comp-1} = \sum_{k=0}^\infty \big\langle \phi, \cf^{-1 }\bm{L}_k^{(*,\sigma)} \cf(\psi) \big\rangle_{\sigma} = \big\langle \phi, \cf^{-1 }\bm{L}^{(*,\sigma)} \cf(\psi) \big\rangle_{\sigma} &= \langle  f, \bm{L}^{(*,\sigma)}g \rangle_{\sigma} \notag \\
&= \langle  f, \bm{R}_r \bm{L}^{(*,\sigma)}g \rangle_{\mu_r}. \label{E:mu_r-adoint-comp-3}
\end{align}
Notice that Lemma \ref{L:correspondence-between-r-and-r'-norms} guarantees that $\bm{R}_r \bm{L}^{(*,\sigma)}g \in L^2(M_\gamma,\mu_r)$.  

On the other hand,
\begin{equation}\label{E:mu_r-adoint-comp-4}
\langle \bm{L} f, g \rangle_{\sigma} = \langle \bm{L} f, \bm{R}_r g \rangle_{\mu_r} = \langle f, \bm{L}^{(*,\mu_r)} \bm{R}_r g\rangle_{\mu_r}.
\end{equation}
Equating \eqref{E:mu_r-adoint-comp-3} and \eqref{E:mu_r-adoint-comp-4} shows $\bm{R}_r\bm{L}^{(*,\sigma)}g = \bm{L}^{(*,\mu_r)} \bm{R}_r g$ for all $g \in L^2(M_\gamma,r')$.  This is equivalent to saying that, as an operator on $L^2(M_\gamma,\mu_r)$,
\begin{equation}
\bm{L}^{(*,\mu_r)}  = \bm{R}_r\bm{L}^{(*,\sigma)} \bm{R}_r^{-1}.
\end{equation}
Writing this out as an integral equation yields \eqref{E:leray-adjoint-wrt-mu_r}.
\end{proof}

\subsection{Related operators}\label{SS:related-operators}
In the discussion below, we often simplify notation and write $\bm{L}^{(*,\mu_r)} = \bm{L}^*$.  This convention is also extended to other operators.

\begin{proposition}\label{P:L_k-is-a-projection-mu_r}
$\bm{L}:L^2(M_\gamma,\mu_r) \to L^2(M_\gamma,\mu_r)$ is a projection for $r \in \ci_0 = (-1,2\gamma-1)$.
\end{proposition}
\begin{proof}
If $r \in \ci_0$, Corollary \ref{C:r-cond} says $\bm{L}$ is bounded on $L^2(M_\gamma,\mu_r)$.  Formula \eqref{E:LkMultipliers2} is valid for $f \in L^2(M_\gamma,\mu_r)$ and the same argument given in Theorem \ref{T:L_k-is-a-projection-sigma} shows that each $\bm{L}_k$, and consequently the full operator $\bm{L}$, is a projection.  
\end{proof}
 
Let $r \in \ci_0$ so that $\bm{L} :L^2(M_\gamma,\mu_r) \to L^2(M_\gamma,\mu_r)$ is bounded.  Define a Hardy space in the following way: 
\begin{equation}\label{E:def-of-hardy-space-H_r}
H^2(M_\gamma,\mu_r) := \bm{L}\left( L^2(M_\gamma,\mu_r) \right) = \ker\left({{\bm L}}-I\right).
\end{equation}
It is verified in the Appendix that $H^2(M_\gamma,\mu_r)$ consists entirely of boundary values of holomorphic functions.  In what follows, write $H^2(M_\gamma,\mu_r) = H$.

Let $\bm{K}\colon H^\perp\to H$ be the restriction of $\bm{L}$ to $H^\perp$; thus, with respect to the decomposition $L^2(M_\gamma,\mu_r)=H\oplus H^\perp$, $\bm{L}$ is given by the operator matrix
$\begin{pmatrix}
I  &  \bm{K} \\ 0 & 0
\end{pmatrix}.$
Similarly, $\bm{L}^*$ is given by 
$\begin{pmatrix}
I  &  0 \\ \bm{K}^* & 0
\end{pmatrix},$  
$\bm{L}^*\bm{L}$ is given by 
$\begin{pmatrix}
I  &  \bm{K}\\ \bm{K}^* & 0
\end{pmatrix},$
$\bm{L}\bm{L}^*$ is given by
$\begin{pmatrix}
I  +  \bm{K}\bm{K}^* & 0 \\ 0 & 0
\end{pmatrix},$
and $\bm{A} := \bm{L}^*-\bm{L}$ is given by 
$\begin{pmatrix}
0  &  -\bm{K} \\ \bm{K}^* & 0
\end{pmatrix}.$
These representations are standard operator theory facts (see, e.g., (2.1.6) and (2.1.34) in \cite{SimonBookOpTheory}) and it is easy to verify that
\begin{subequations}\label{E:AK}
\begin{align}
\|\bm{A}\|&=\|\bm{K}\|=\|\bm{K}^*\|\\
\|\bm{L}\|&=\|\bm{L}^*\|=\sqrt{1+\|\bm{K}\|^2}\\
\|\bm{L}\bm{L}^*\|&=\|\bm{L}^*\bm{L}\|=\|\bm{L}\|^2=1+\|\bm{K}\|^2.
\end{align}
\end{subequations}

The operator $\bm{A}$ stems from work of Kerzman and Stein \cite{KerSte78a,KerSte78b} examining the relation between certain Cauchy-Fantappi\`e projections and the self-adjoint Szeg\H{o} projection
${\bm S}$ (corresponding to $\bm{K}=0$).  See \cite{Bol06,Bol07,BarBol07,BolRai15,Bell16} for results on $\bm{A}$ in the complex plane and \cite{BarLan09} for results on Reinhardt domains in $\C^2$.

The formulas in \eqref{E:AK} respect the decomposition $L^2(M_\gamma,\mu_r) = \bigoplus_{k=0}^\infty L_k^2(M_\gamma,\mu_r)$, i.e., identical statements hold with $\bm{L}$, $\bm{A}$ and $\bm{K}$ replaced by $\bm{L}_k$, $\bm{A}_k$ and $\bm{K}_k$, respectively.  To be more specific, write 
$$
\bm{L}^{(*,\mu_r)}\bm{L} = \bigoplus_{k=0}^\infty \bm{L}_k^{(*,\mu_r)}\bm{L}_k, \qquad \bm{L} \bm{L}^{(*,\mu_r)} = \bigoplus_{k=0}^\infty \bm{L}_k \bm{L}_k^{(*,\mu_r)}
$$ 
and 
$$
\bm{A}^{\mu_r} = \bigoplus_{k=0}^\infty \bm{A}_k^{\mu_r}, \qquad \mathrm{where} \qquad \bm{A}_k^{\mu_r} = \bm{L}^{(*,\mu_r)}_k-\bm{L}_k.
$$  
We devote the rest of the section to the study of these operators.  Computations are omitted, but can be easily reconstructed by an interested reader.

\subsubsection{Spectra of $\bm{L}^*\bm{L}$ and $\bm{L}\bm{L}^*$}\label{SSS:spectra-L*L-and-LL*}
Let $k \ge 0$, $r \in \ci_k$ and $f = \sum_{k=-\infty}^\infty f_k e^{ik\theta} \in L^2(M_\gamma,\mu_r)$.  

We analyze the $k^{th}$ piece of $\bm{L}^{(*,\mu_r)}\bm{L}$ by considering the action of the following unitarily equivalent operator on $f$.  Equations \eqref{E:LkMultipliers2} and \eqref{E:fourier-transformed-adjoint-formula} show
$$
\cf^{-1}\bm{L}_k^{(*,\mu_r)}\bm{L}_k\cf(f)(\xi,\alpha,\theta) = C_{\mu_r}(\gamma,k) \left< f_k(\xi,\cdot), \tfrac{\kappa_k(\xi,\cdot)}{\norm{\kappa_k(\xi,\cdot)}^2_{(\mu_{r'},1)}} \right>_{(\sigma,1)} \alpha^{\gamma-1-r} \kappa_k(\xi,\alpha) e^{i k \theta}.
$$
It is not hard to verify that the related operator $\frac{1}{C_{\mu_r}(\gamma,k)} \cdot \cf^{-1}\bm{L}_k^{(*,\mu_r)}\bm{L}_k\cf$ represents the following orthogonal projection (in the $\langle \cdot,\cdot \rangle_{\mu_r}$ inner product):
$$
L^2_k(M_\gamma,\mu_r) \xrightarrow{\perp_{\mu_r}} X_k,
$$
where
\begin{equation*}\label{E:def-of-space-X_k}
X_k = \left\{ m(\xi) \alpha^{\gamma-1-r} \kappa_k(\xi,\alpha)e^{i k \theta} : \int_{-\infty}^0 |m(\xi)|^2 |\xi|^{\left(\frac{2k+1+r}{\gamma}-2k-2\right)}d\xi \right\}.
\end{equation*}
We emphasize the following observations:
\begin{itemize}
\item Infinite dimensional $X_k$ is constructed around the single function $\alpha^{\gamma-1-r} \kappa_k(\xi,\alpha)e^{i k \theta}$.
\item $X_k$ coincides with the eigenspace $E_{C_{\mu_r}(\gamma,k)}$ of the operator $\cf^{-1}\bm{L}_k^{(*,\mu_r)}\bm{L}_k\cf$.
\item $X_k$ is precisely the space of functions which achieve the norm of the operator $\cf^{-1}\bm{L}_k\cf$, as described by \eqref{E:mu_r-norm-acieving-condition1} and \eqref{E:mu_r-norm-acieving-condition2}.
\end{itemize}

Now analyze the $k^{th}$ piece of $\bm{L}\bm{L}^{(*,\mu_r)}$ by considering the action of the following unitarily equivalent operator on $f$.  Equations \eqref{E:LkMultipliers2} and \eqref{E:fourier-transformed-adjoint-formula} show
$$
\cf^{-1}\bm{L}_k\bm{L}_k^{(*,\mu_r)}\cf(f)(\xi,\alpha,\theta) = C_{\mu_r}(\gamma,k) \left< f_k(\xi,\cdot), \tfrac{\tau_k(\xi,\cdot)}{\norm{\tau_k(\xi,\cdot)}^2_{(\mu_{r},1)}} \right>_{(\mu_r,1)} \tau_k(\xi,\alpha) e^{i k \theta}.
$$
It is clear that the related operator $\frac{1}{C_{\mu_r}(\gamma,k)} \cdot \cf^{-1}\bm{L}_k \bm{L}^{(*,\mu_r)}_k\cf$ represents the following orthogonal projection (in the $\langle \cdot,\cdot \rangle_{\mu_r}$ inner product):
$$
L^2_k(M_\gamma,\mu_r) \xrightarrow{\perp_{\mu_r}} Y_k,
$$
where
\begin{equation*}\label{E:def-of-space-Y_k}
Y_k = \left\{ m(\xi) \tau_k(\xi,\alpha)e^{i k \theta} : \int_{-\infty}^0 |m(\xi)|^2 |\xi|^{-\left(\frac{2k+r+1}{\gamma}\right)}d\xi \right\}.
\end{equation*}

We emphasize the following observations:
\begin{itemize}
\item Infinite dimensional $Y_k$ is constructed around the single function $\tau_k(\xi,\alpha)e^{i k \theta}$.
\item $Y_k$ coincides with the eigenspace $E_{C_{\mu_r}(\gamma,k)}$ of the operator $\cf^{-1}\bm{L}_k\bm{L}_k^{(*,\mu_r)}\cf$.
\item $Y_k$ is precisely the space of functions in the image of the operator $\cf^{-1}\bm{L}_k$.
\end{itemize}
In summary,
\begin{theorem}\label{T:spectrum-of-Lk*Lk}
Fix an integer $k \ge 0$, $r \in \ci_k$ and let $\bm{T}_k$ be either $\bm{L}_k^{(*,\mu_r)}\bm{L}_k$ or $\bm{L}_k\bm{L}_k^{(*,\mu_r)}$.  Then $\bm{T}_k$ admits an orthogonal basis of eigenfunctions and its spectrum is given by
\begin{equation*}
\left\{0, C_{\mu_r}(\gamma,k) \right\}.
\end{equation*}
\end{theorem}

If $r \in \ci_0$, then $r \in \ci_k$ for all $k\ge0$.  Consequently, for this range of $r$ values, the conclusion of Theorem \ref{T:spectrum-of-Lk*Lk} holds for each non-negative $k$.  This immediately implies Theorem \ref{T:spectrum-of-L*L}.

\subsubsection{The spectrum of $\bm{A}$}\label{SSS:spectrum-of-A}
Again let $k \ge 0$, $r \in \ci_k$ and $f = \sum_{k=-\infty}^\infty f_k e^{ik\theta} \in L^2(M_\gamma,\mu_r)$.
We analyze the $k^{th}$ piece of the anti-self-adjoint $\bm{A}^{\mu_r}$ by considering the action of the following unitarily equivalent operator on $f$.  Equations \eqref{E:LkMultipliers2} and \eqref{E:fourier-transformed-adjoint-formula} show
\begin{align}
\cf^{-1}\bm{A}^{\mu_r}_k\cf(f)(\xi,\alpha,\theta) &= \cf^{-1} \left(\bm{L}^{(*,\mu_r)}_k - \bm{L}_k \right) \cf(f)(\xi,\alpha,\theta) \notag \\ 
&= \eta_k(\xi) \Big( \alpha^{\gamma-1-r}\kappa_k(\xi,\alpha) \langle f_k(\xi,\cdot), \tau_k(\xi,\cdot) \rangle_{(\mu_r,1)} \label{E:kerzman-stein-op-1}  \\
& \qquad \qquad \qquad \qquad \qquad - \tau_k(\xi,\alpha) \left< f_k(\xi,\cdot), \kappa_k(\xi,\cdot) \right>_{(\sigma,1)} \Big)  e^{ik\theta}. \notag
\end{align}

For our purposes, it is convenient to introduce the following function
\begin{equation}\label{E:def-of-function-lambda_k}
\lambda_k(\xi,\alpha) := \alpha^{\gamma-1-r} \kappa_k(\xi,\alpha) - \tfrac{\tau_k(\xi,\alpha)}{\eta_k(\xi) \norm{\tau_k(\xi,\alpha)}^2_{(\mu_r,1)}}.
\end{equation}
(The notation $\lambda_k(\xi,\alpha)$ is used only in this section and should not be confused with our notation for the Leray-Levi measure used elsewhere in the paper.)  It is easily checked that for each fixed $\xi < 0$, $\lambda_k(\xi,\cdot) \perp_{(\mu_r,1)} \tau_k(\xi,\cdot)$, and that
\begin{equation*}
\norm{\lambda_k(\xi,\cdot)}^2_{(\mu_r,1)} = \norm{\kappa_k(\xi,\cdot)}^2_{(\mu_{r'},1)} - \tfrac{1}{\eta_k(\xi)^2 \norm{\tau_k(\xi,\cdot)}^2_{(\mu_r,1)}}.
\end{equation*}
It follows that we can re-express \eqref{E:kerzman-stein-op-1} in a more symmetric fashion:
\begin{align}
\cf^{-1}\bm{A}^{\mu_r}_k\cf(f)(\xi,\alpha,\theta) &= \cf^{-1} \left(\bm{L}^{(*,\mu_r)}_k - \bm{L}_k \right) \cf(f)(\xi,\alpha,\theta) \notag \\ 
&= \eta_k(\xi) \Big( \lambda_k(\xi,\alpha) \left< f_k(\xi,\cdot), \tau_k(\xi,\cdot) \right>_{(\mu_r,1)} \label{E:kerzman-stein-op-2}  \\
& \qquad \qquad \qquad \qquad \qquad - \tau_k(\xi,\alpha) \left< f_k(\xi,\cdot), \lambda_k(\xi,\cdot) \right>_{(\mu_r,1)} \Big)  e^{ik\theta}. \notag
\end{align}
A computation now shows that
\begin{align}
\norm{\cf^{-1}\bm{A}^{\mu_r}_k\cf(f)}^2_{(\mu_r,1)} &= \left(C_{\mu_r}(\gamma,k)-1\right) \Bigg( \Big| \Big\langle f_k(\xi,\cdot),  \tfrac{\tau_k(\xi,\cdot)}{\norm{\tau_k(\xi,\cdot)}_{(\mu_r,1)}} \Big\rangle_{(\mu_r,1)} \Big|^2 + \notag \\
& \qquad \qquad \qquad \qquad \qquad \qquad \qquad  \Big| \Big\langle f_k(\xi,\cdot),  \tfrac{\lambda_k(\xi,\cdot)}{\norm{\lambda_k(\xi,\cdot)}_{(\mu_r,1)}} \Big\rangle_{(\mu_r,1)} \Big|^2 \Bigg),
\end{align}
meaning $\cf^{-1}\bm{A}^{\mu_r}_k\cf$ annihilates $L^2(M_\gamma,\mu_r) \backslash Z_k$, where
$$
Z_k := \text{span}\{ \tau_k(\xi,\alpha)e^{ik\theta}, \lambda_k(\xi,\alpha)e^{ik\theta} \}.
$$

Now decompose $Z_k = Z^1_k \bigoplus Z^2_k$, where
\begin{align*}
Z^1_k &= \left\{ m(\xi)\left( \lambda_k(\xi,\alpha) + i \tfrac{\norm{\lambda_k(\xi,\cdot)}_{(\mu_r,1)}}{\norm{\tau_k(\xi,\cdot)}_{(\mu_r,1)}} \tau_k(\xi,\alpha) \right)e^{i k \theta} : \int_{-\infty}^0 |m(\xi)|^2 |\xi|^{\left(\frac{2k+r+1}{\gamma} -2k-2\right)}d\xi \right\}, \\
Z^2_k &= \left\{ m(\xi)\left( \lambda_k(\xi,\alpha) - i \tfrac{\norm{\lambda_k(\xi,\cdot)}_{(\mu_r,1)}}{\norm{\tau_k(\xi,\cdot)}_{(\mu_r,1)}} \tau_k(\xi,\alpha) \right)e^{i k \theta} : \int_{-\infty}^0 |m(\xi)|^2 |\xi|^{\left(\frac{2k+r+1}{\gamma} -2k-2\right)}d\xi \right\}. 
\end{align*}
It is easily verified that
\begin{itemize}
\item $Z^1_k$ coincides with the eigenspace $E_{i \sqrt{C_{\mu_r}(\gamma,k)-1}}$ of $\cf^{-1}\bm{A}_k^{\mu_r}\cf$.
\item $Z^2_k$ coincides with the eigenspace $E_{-i \sqrt{C_{\mu_r}(\gamma,k)-1}}$ of $\cf^{-1}\bm{A}_k^{\mu_r}\cf$.
\item These two spaces are infinite dimensional, but both are constructed around its own particular linear combination of $\tau_k(\xi,\alpha)e^{ik\theta}$ and $\lambda_k(\xi,\alpha)e^{ik\theta}$.
\end{itemize}
In summary,
\begin{theorem}\label{T:spectrum-of-Ak}
Fix an integer $k \ge 0$ and let $r \in \ci_k$.  The operator $\bm{A}_k^{\mu_r}$ admits an orthogonal basis of eigenfunctions and its spectrum is given by
\begin{equation*}
\left\{0, \pm i \sqrt{C_{\mu_r}(\gamma,k) - 1} \right\}.
\end{equation*}
\end{theorem}

If $r \in \ci_0$, then $r \in \ci_k$ for all $k\ge0$.  Consequently, for this range of $r$ values the conclusion of Theorem \ref{T:spectrum-of-Ak} holds for each non-negative $k$.  This immediately implies Theorem \ref{T:spectrum-of-L*-L}.


\section{Projective Duality}\label{S:Duality}

We now introduce the notion of {\em projective duality}.  Define $\C\mathbb{P}^{n*}$ to be the set of hyperplanes in $\C\p^n$.  It is easily seen that this set can be isomorphically identified with $\C\mathbb{P}^{n}$.  Indeed, each hyperplane in $\C\p^n$ corresponds to a unique point $\zeta = [\zeta_0:\dots:\zeta_n] \in \C\mathbb{P}^{n}$:
\begin{align*}
\mathfrak{h}_{\zeta} = \{\zeta^* \in \C\p^n: \zeta_0 \zeta_0^* + \dots + \zeta_n \zeta_n^* = 0 \},
\end{align*}
and it is clear that this identification is reflexive.

Given a domain $\Omega \subset \C\p^n$ with smooth boundary $\cs$, we define the dual $\cs^* \subset \C\p^{n*}$ to be the set of complex hyperplanes tangent to $\cs$.  See \cite{ScandBook04} for a detailed treatment of this topic; also see \cite{BarGru18,BarGru20} for related recent developments.  When $\cs$ is a smooth and strongly $\C$-convex hypersurface, $\cs^*$ is also smooth and strongly $\C$-convex and $\cs^{**}=\cs$. However, such regularity properties for $\cs^*$ often fail under weaker hypotheses on $\cs$; see Remark \ref{R:M1}.

From the denominator of $\mathscr{L}_{\cs}$ in \eqref{E:LerayKernel}, we see that the complex tangent hyperplanes to a smooth hypersurface $\cs$ play a critical role in the Leray transform.  Strong $\C$-convexity ensures that each supporting hyperplane intersects $\overline{\Omega}$ at the point of tangency and nowhere else (see Section 2.5 in \cite{ScandBook04}).    In \cite{Bar16}, the first author develops the connection between the Leray transforms $\bm{L}_{\cs}$ and $\bm{L}_{\cs^*}$ and shows that these operators respect a natural bilinear pairing of two ``dual" Hardy spaces when $\cs$ and $\cs^*$ are smoothly bounded, strongly $\C$-convex hypersurfaces.  The Leray transform plays a role in higher dimensions analogous to the role played by the Cauchy transform of a planar domain in pairing two Hardy spaces of holomorphic boundary values associated to the interior and exterior of a planar curve, respectively; see \cite{Meyer82}.  We extend this work in Section \ref{SS:bilinear-pairing-of-hardy-spaces}.


\subsection{The dual hypersurface of $M_\gamma$}\label{SS:dual-hypersurface-of-M_gamma}
 If $\cs$ is a smooth strongly $\C$-convex hypersurface, we can choose an appropriate matrix $A \in GL(n+1,\C)$ to induce an affinization of $\C\p^{n*}$ and thereby represent $\cs^*$ as a hypersurface in $\C^{n}$.  The details below were first introduced by the authors in Section 4 of \cite{BarEdh17}.

Define the map $\Phi_A:\C^n \times \C^n \to \C$ by calculating the following matrix product:
\begin{equation}\label{E:Phi_A-map}
\Phi_A((z_1,\dots,z_n),(w_1,\dots,w_n)) = [1\,\,\,\, w_1\,\,\dots\,\, w_n]\,A\,[1\,\,\,\, z_1\,\,\dots\,\, z_n]^T.
\end{equation}

The following proposition is stated for two dimensional $\cs$, but it admits a straightforward generalization to higher dimensions.

\begin{proposition}\label{P:Admissible}
Given a smooth real hypersurface $\cs\subset\C^2$ with defining function $\rho$
and a matrix 
\begin{equation}\label{E:affinization-matrix}
A = \begin{pmatrix}
c_1  & a_1  &  a_2 \\
b_1  &  m_{11} &  m_{12} \\
b_2  & m_{21}  &  m_{22} 
\end{pmatrix} \in GL(3,\C)
\end{equation}
the following conditions are equivalent:
\refstepcounter{equation}\label{N:Admissible}
\begin{enum}
\item There are functions $w_1^A$ and $w_2^A$ on $\cs$ with the property that the complex tangent line to $\cs$ at $\zeta$ is given by
\begin{equation*}\label{E:TgtLineEq}
\left\{(z_1,z_2)\colon\Phi_A\big( (z_1,z_2), (w_{1}^A(\zeta),w_{2}^A(\zeta)) \big)=0\right\}.
\end{equation*} \label{I:w-cond}
\item Same as \itemref{N:Admissible}{E:TgtLineEq} with the additional conditions that 
$w_1^A$ and $w_2^A$ are uniquely determined and smooth. \label{I:w-cond+}
\item The matrix $\begin{pmatrix}
   m_{11}\zeta_1+m_{12}\zeta_2+b_1  & m_{21}\zeta_1+m_{22}\zeta_2+b_2 \\
  m_{11}\frac{\dee\rho}{\dee\zeta_2}-m_{12}\frac{\dee\rho}{\dee\zeta_1} &     m_{21}\frac{\dee\rho}{\dee\zeta_2}-m_{22}\frac{\dee\rho}{\dee\zeta_1}    
\end{pmatrix}$ is invertible for all $\zeta\in\cs$. \label{I:matrix-cond}
\end{enum}
\end{proposition}

 \begin{proof}
The complex tangent line to $\cs$ at $\zeta$ is the unique (affine) line joining $\zeta$ to the point $\zeta+\big(\frac{\dee\rho}{\dee\zeta_2},-\frac{\dee\rho}{\dee\zeta_1}\big)$.  This follows from the fact that the vector field $\frac{\dee\rho}{\dee\zeta_2}\frac{\dee}{\dee\zeta_1}-\frac{\dee\rho}{\dee\zeta_1}\frac{\dee}{\dee\zeta_2}$ is tangent to $\cs$ at $\zeta$.  To try to represent the tangent line in the desired form we seek solutions $(w_1^A,w_2^A)$ of the inhomogeneous linear system
\begin{align*}
(1,w_1^A,w_2^A)\,A\,(1, \zeta_1,\zeta_2)^T &=0\notag\\
(1, w_1^A,w_2^A)\,A\,\left(1,\zeta_1+\tfrac{\dee\rho}{\dee\zeta_2},\zeta_2-\tfrac{\dee\rho}{\dee\zeta_1}\right)^T&=0.
\end{align*}
Uniqueness of the line implies that the system has at most one solution, so if there is any solution the associated homogeneous system is non-singular; conversely, non-singularity of the associated homogeneous system implies the existence of a unique smooth solution of the inhomogeneous system.  A computation reveals that the matrix representing the associated homogeneous system is \[\begin{pmatrix}
   m_{11}\zeta_1+m_{12}\zeta_2+b_1  & m_{21}\zeta_1+m_{22}\zeta_2+b_2 \\
  m_{11}\frac{\dee\rho}{\dee\zeta_2}-m_{12}\frac{\dee\rho}{\dee\zeta_1} &     m_{21}\frac{\dee\rho}{\dee\zeta_2}-m_{22}\frac{\dee\rho}{\dee\zeta_1}    
\end{pmatrix}.\]
\end{proof}

\begin{remark}
A corresponding result holds in higher dimensions, in which condition \itemref{N:Admissible}{I:matrix-cond} is replaced by
a maximal rank condition on a certain $n \times \left(\binom{n}{2}+1\right)$ matrix.
$\lozenge$\end{remark}

\begin{definition}
When the equivalent conditions in Proposition \ref{P:Admissible} hold, we say that $\cs$ is {\em $A$-admissible} and denote the image of $\cs$ under the map $w^A=(w_1^A,w_2^A)$ by $\cs^{*,A}$.  When $A$ is clear from context, we may drop the superscripts and simply write $w=(w_1,w_2)$ and $\cs^*$.
\end{definition}

It is easily checked that any hypersurface that can written as a smooth graph over the variables $z_1$ and $\re(z_2)$ is $A$-admissible for any invertible matrix $A$ whose only non-zero entries occur on the anti-diagonal.  For $M_\gamma$, $\gamma>1$, we'll use the matrix
\begin{equation}\label{A-Spec}
A_\gamma=\begin{pmatrix}
0  & 0  &  i \\
0  &  \gamma &  0 \\
i(\gamma-1)  & 0  &  0 
\end{pmatrix},
\end{equation}
which yields
\begin{align}
\Phi_{A_\gamma}((z_1,z_2),(w_1(\zeta),w_2(\zeta))) = i z_2 + \gamma w_1(\zeta) z_1 + i(\gamma-1) w_2(\zeta) = 0.
\end{align}
Rewriting the tangent line 
\begin{equation}\label{E:TgtLineEqFromRho}
\left\{(z_1,z_2)\colon \frac{\dee\rho}{\dee \zeta_1}(\zeta)\cdot(z_1-\zeta_1)+\frac{\dee\rho}{\dee \zeta_2}(\zeta)\cdot(z_2-\zeta_2)=0\right\}
\end{equation}
to any $A_\gamma$-admissible hypersurface $\cs$ at $\zeta$ in the form
\begin{equation}\label{E:TangentLineWithWs}
\left\{(z_1,z_2): \gamma w_1^{A_\gamma}(\zeta)z_1 + i z_2 = i(1-\gamma) w_2^{A_\gamma}(\zeta) \right\},
\end{equation} 
we find  that 
\begin{align*}
w_1^{A_\gamma}(\zeta)&=\frac{i\rho_{\zeta_1}}{\gamma \rho_{\zeta_2}}, \qquad w_2^{A_\gamma}(\zeta)=\frac{\zeta_1 \rho_{\zeta_1} + \zeta_2 \rho_{\zeta_2}}{(1-\gamma) \rho_{\zeta_2}}.
\end{align*}

Specializing now to $M_\gamma$ by setting $\rho(\zeta) = |\zeta_1|^\gamma - \im(\zeta_2)$,
we obtain
\begin{align*}
\frac{\dee\rho}{\dee\zeta_1}(\zeta) = \frac{\gamma}{2}\bar{\zeta}_1|\zeta_1|^{\gamma-2}, \qquad \frac{\dee\rho}{\dee\zeta_2}(\zeta) = \frac{i}{2},
\end{align*}
and thus
\begin{subequations}\label{E:wj}
\begin{align}
w_1^{A_\gamma}(\zeta) &= \bar{\zeta}_1|\zeta_1|^{\gamma-2}\\
 \qquad w_2^{A_\gamma}(\zeta) &= \frac{\zeta_2 -i\gamma|\zeta_1|^\gamma}{1-\gamma}\\
 &= \frac{\left(1-\frac{\gamma}{2}\right)\zeta_2 +\frac{\gamma}{2}\bar \zeta_2}{1-\gamma}.
 \notag
\end{align}
\end{subequations}
This shows that the defining equation for the dual hypersurface $M_\gamma^{*,A_\gamma}$ is
\begin{equation}\label{E:Mgamma-dual-hypersurface-w-eqns}
\im\left(w_2^{A_\gamma}(\zeta)\right) = \left|w_1^{A_\gamma}(\zeta)\right|^{\tfrac{\gamma}{\gamma-1}}.
\end{equation}
Thus via the matrix $A_\gamma$,  the dual hypersurface $M_\gamma^* = M_{\gamma^*}$, where  $\gamma^*=\tfrac{\gamma}{\gamma-1}$ is the H\"older conjugate of the exponent $\gamma$.





Note for future reference that
\begin{subequations}\label{E:reverse}
\begin{align}
\zeta_1&=\overline{w_1^{A_\gamma}} \left|w_1^{A_\gamma}\right|^{\frac{2-\gamma}{\gamma-1}}\\
|\zeta_1|&= \left|w_1^{A_\gamma}\right|^{\frac{1}{\gamma-1}}\\
\zeta_2 &= \left(1-\frac{\gamma}{2}\right) w_2^{A_\gamma} - \frac{\gamma}{2}\overline{w_2^{A_\gamma}}= (1-\gamma)w_2^{A_\gamma} + i\gamma\left|w_1^{A_\gamma}\right|^{\frac{\gamma}{\gamma-1}} \\
\re \zeta_2 &= (1-\gamma)\re w_2^{A_\gamma}.
\end{align}
\end{subequations}

\begin{remark}\label{R:adm}
Using the $(s,\alpha,\theta)$ parametrization of $M_\gamma$ from Section \ref{SS:LerayReparam}, the matrix in \itemref{N:Admissible}{I:matrix-cond} has determinant
\begin{equation*}
\Delta:=P+Q e^{-i\theta}\alpha^{\gamma-1} + R\left(\alpha^\gamma+i\frac{s}{\gamma-1}\right)
\end{equation*}
where
\begin{align*}
P&=\frac{i}{2}\left( b_1m_{21}-b_2m_{11}\right)\\
Q&=\frac{\gamma}{2}\left( b_2m_{12}-b_1m_{22}\right)\\
R&=\frac{\gamma-1}{2}\left( m_{12}m_{21}-m_{11}m_{22}\right).
\end{align*}
Thus, the matrix $A$ is admissible if and only if $\Delta$ is non-zero for all $s,\theta\in\R, \alpha>0$.  
We make the following observations.
\begin{itemize}
\item Invertibility of $A$ implies that $P,Q,R$ are not all zero.
\item If $R=0$ then it is easy to see that $M_\gamma$ is $A$-admissible if and only if one of $P,Q$ vanishes.
\item For $R\ne0$ we can apply Young's inequality in the form $\alpha^{\gamma-1}\le \frac{1}{\gamma}+\frac{\gamma-1}{\gamma}\alpha^\gamma$ to obtain  
\begin{align*}
\re \Delta\bar R &= \re P\bar R + R\bar R\alpha^\gamma +  \re(Q\bar R e^{-i\theta})\cdot \alpha^{\gamma-1}\\
&\ge \left( \re P\bar R-\frac{|QR|}{\gamma}\right)+ \left( |R|^2 - \frac{(\gamma-1)|QR|}{\gamma} \right) \alpha^\gamma;
\end{align*}
so if $|Q|\le \gamma \min\left\{  \re\frac{P\bar R}{|R|},
\frac{|R|}{\gamma-1} \right\}$ (with strict inequality if $\re\frac{P\bar R}{|R|}=\frac{|R|}{\gamma-1}$) then $\re \Delta\bar R$ is strictly positive and thus $M_\gamma$ is $A$-admissible.
\item On the other hand, if $Q=0$ and $P\bar R$ is negative then it is easy to see that $M_\gamma$ is not $A$-admissible.  $\lozenge$
\end{itemize}


\end{remark}

\begin{remark}\label{R:adm-compete}
In general, when a hypersurface $\cs$ is both $A$-admissible and $A'$-admissible then the dualization maps $w^A$ and $w^{A'}$ are related by $w^{A'} = w^{A}\circ \reallywidehat{\left({A(A')^{-1}}\right)^T}$, where $\reallywidehat{\left({A(A')^{-1}}\right)^T}$ is the projective automorphism corresponding to the matrix $\left({A(A')^{-1}}\right)^T$ as in \eqref{E:Xhat}.  (See Lemma 4.28 in \cite{BarEdh17}.)
\end{remark}

\begin{remark}\label{R:M1}
As $\gamma\searrow 1$ the $M_\gamma$ tend to the hypersurface $M_1$ which is $\C$-convex but not strongly $\C$-convex.  The matrix $A_1$ is not invertible, but using the matrix
$$
A_*:= \begin{pmatrix}
0  & 0  &  i \\
0  &  \gamma &  0 \\
i  & 0  &  0 
\end{pmatrix}
$$ 
instead we obtain $w_1^{A_*}(\zeta) = \bar{\zeta}_1/|\zeta_1|$ and $w_2^{A_*}(\zeta) = i|\zeta_1|-\zeta_2 = -\re \zeta_2$.  Thus $M_1$ is mapped onto the 2-manifold $S^1\times \R$.
$\lozenge$
\end{remark}


\subsection{Distinguished measures on $M_\gamma$} \label{SS:distinguished-measures}

In Section \ref{S:meas} we previewed several measures now to be discussed in detail.  Each of these belongs to one of the families
\begin{equation*}
\sF_r = \{c \mu_r : c > 0 \},
\end{equation*}
where $\mu_r = \alpha^r \,ds \w d\alpha\w d\theta$, so Leray boundedness results obtained in previous sections are immediately applicable.

\subsubsection{The pairing measure}\label{SSS:pairing-measure}

Let $A \in GL(3,\C)$ be as in \eqref{E:affinization-matrix} and $\cs$ an $A$-admissible, strongly $\C$-convex hypersurface in $\C^2$.  Define the pairing measure (with respect to $A$) on $\cs$ by the following formula:
\begin{multline}\label{E:NuDef}
\nu^A:= \frac{1}{(2\pi i)^2}\Bigg(\left(  d w_{1,A}, d w_{2,A}\right) 
\begin{pmatrix} m_{11} & m_{12} \\ m_{21} & m_{22}\end{pmatrix}
\begin{pmatrix} a_2 \\ -a_1\end{pmatrix}\\
+
\begin{vmatrix} m_{11} & m_{12} \\ m_{21} & m_{22}\end{vmatrix}
\left( w_{2,A}\,d w_{1,A} - w_{1,A}\,d w_{2,A}\right)\Bigg)\w d\zeta_1\w d\zeta_2.
\end{multline}

The authors previously introduced this measure in \cite{BarEdh17}.   (Related objects also appear in Section 7 of  \cite{Bar16} and Section 3.2 of \cite{ScandBook04}.)  This measure allows for the Leray transform to be written in terms of the dual variables.  We restate Proposition 4.30 from \cite{BarEdh17}, which gives a universal formula for the Leray transform of any $A$-admissible $\cs$:

\begin{proposition}\label{T:GenMLeray}
Let $A \in GL(3,\C)$ be as in \eqref{E:affinization-matrix} and $\cs$ an $A$-admissible, strongly $\C$-convex hypersurface in $\C^2$.  Then the Leray integral from \eqref{E:LerayIntegralFormula} may be written as
\begin{equation}\label{E:GenMLeray}
\bm{L}_{\cs}f(z) =  \int_{\zeta\in \cs} \frac{f(\zeta)}{\big[ \Phi_A(z,w^A(\zeta)) \big]^2} \nu^A(\zeta),
\end{equation}
where $\Phi_A$ is defined in \eqref{E:Phi_A-map}.
\end{proposition}

In the special case of  $\cs = M_\gamma$, $A=A_\gamma$ we have
\begin{align}\label{E:NuLam}
\nu^{A_\gamma} = - \frac{i\gamma}{4 \pi^2} \,dw_{1}^{A_\gamma}\w d\zeta_1\w d\zeta_2 &= \frac{\gamma^2}{8 \pi^2 i} |\zeta_1|^{\gamma-2}\,d\zeta_2\w d\bar \zeta_1\w d\zeta_1 \notag \\ 
&= \frac{\gamma^2}{4 \pi^2} \alpha^{\gamma-1}ds \wedge d\alpha \wedge d\theta.
\end{align}

\subsubsection{The preferred measure and its dual}\label{SS:mod-Fefferman-and-its-dual}

In Section \ref{SSS:mod-Fefferman}, we defined (up to a constant) the preferred measure $\wt{\mu}_{\cs}$.  This measure was introduced by the first author in \cite{Bar16} to define a pair of dual Hardy spaces with desirable projective transformation laws.  Employing the Fefferman measure \eqref{E:fefdef} to construct such spaces still leads to such transformation laws, but the preferred measure $\wt{\mu}_\cs$ dovetails with a pairing of these spaces in the spirit of \eqref{E:GenMLeray} in a way that $\mu_\cs^{\sf Fef}$ does not.  

One Hardy space is defined using $\wt{\mu}_{\cs}$.  To define the other Hardy space, introduce the {\em preferred dual measure} by pulling back the preferred measure on the dual $\cs^{*,A}$ via the map $w^A$: 
\begin{equation}\label{E:def-of-dual-fefferman-measure}
\wt{\mu}_{\cs}^{*,A} := \left(w^A\right)^*\left( \wt{\mu}_{\cs^{*,A}} \right).
\end{equation}
When the hypersurface and matrix are clear from context, we may simply write $\wt{\mu}$ and $\wt{\mu}^*$ to denote the preferred measure and its dual.  

The preferred measure on $M_\gamma$ was already recorded in \eqref{E:ModifedFefMGamma}:
\begin{equation*}
\widetilde{\mu}_{M_\gamma} = c_2 \frac{\gamma^{4/3}}{8(\gamma-1)^{1/3} i}\, |\zeta_1|^{\frac{\gamma-2}{3}} \,d\zeta_2\w d\bar{\zeta}_1\w d\zeta_1.
\end{equation*}
We will chose a convenient normalization constant $c_2$ shortly.  But first let us compute the dual measure.  From \eqref{E:wj} we obtain
\begin{equation*}
d\overline{w}_1^{A_\gamma}\w dw_1^{A_\gamma} = (1-\gamma)|\zeta_1|^{2\gamma-4}\,d\bar{\zeta}_1\w d\zeta_1,
\end{equation*}
so from \eqref{E:def-of-dual-fefferman-measure}, we see
\begin{align}\label{E:mu*}
\wt{\mu}_{M_\gamma}^{*,A_\gamma}
&= \frac{c_2\cdot \left(\gamma^*\right)^{4/3}
|w_1^{A_\gamma}|^{\frac{\gamma^*-2}{3}}}{8(\gamma^*-1)^{1/3}i} 
\,dw_2^{A_\gamma}\w d\bar{w}_1^{A_\gamma}\w dw_1^{A_\gamma}\notag\\
&= \frac{c_2\cdot \gamma^{4/3}}{8(\gamma-1)i}
|\zeta_1|^{\frac{5(\gamma-2)}{3}}
\,d\zeta_2\w d\bar{\zeta}_1\w d\zeta_1.
\end{align}
Note the Radon-Nikodym derivative
\begin{equation}\label{E:RadNik}
\frac{ {\rm d} \wt{\mu}_{M_\gamma}^{*,A_\gamma}}{ {\rm d}\wt{\mu}_{M_\gamma}}
= \frac{|\zeta_1|^{\frac43(\gamma-2)}}{(\gamma-1)^{2/3}}.
\end{equation}
%
%
%
%
Now define the geometric mean $\sqrt{ \wt{\mu}_{M_\gamma} \, \wt{\mu}_{M_\gamma}^{*,A_\gamma} }$ of measures $\wt{\mu}_{M_\gamma}$ and $\wt{\mu}_{M_\gamma}^{*,A_\gamma}$ by 
\begin{equation*}
{\rm d} \sqrt{ \wt{\mu}_{M_\gamma} \, \wt{\mu}_{M_\gamma}^{*,A_\gamma} }
:=  \sqrt{\frac{ {\rm d} \wt{\mu}_{M_\gamma}^{*,A_\gamma}}{ {\rm d}\wt{\mu}_{M_\gamma}}}\:
{\rm d} \wt{\mu}_{M_\gamma}.
\end{equation*}
Therefore
\[
\sqrt{\wt{\mu}_{M_\gamma} \, \wt{\mu}_{M_\gamma}^{*,A_\gamma} }=\dfrac{c_2\gamma^{4/3}
|\zeta_1|^{\gamma-2}}{8(\gamma-1)^{2/3} i}\,d\zeta_2\w d\bar \zeta_1\w d\zeta_1
\]
and thus,
\begin{align*}
\left| \int_{M_\gamma} f g \,{\rm d}\nu^{A_\gamma} \right|
&\le \sqrt{ \int_{M_\gamma} |f|^2\,{\rm d}\wt{\mu}_{M_\gamma} }
\sqrt{ \int_{M_\gamma} |g|^2\,{\rm d}\wt{\mu}_{M_\gamma}^{*,A_\gamma}  }
\frac{{\rm d}\nu^{A_\gamma }}{{\rm d} \sqrt{\wt{\mu}_{M_\gamma} \, \wt{\mu}_{M_\gamma}^{*,A_\gamma} }}\\
&= \frac{\left(\gamma(\gamma-1)\right)^{2/3}}{c_2 \pi^2}
\|f\|^2_{L^2\left(M_\gamma, \wt{\mu}\right)}
\|g\|^2_{L^2\left(M_\gamma, \wt{\mu}^*  \right)}.
\end{align*}
(Here we identify the positive three-form $\nu^{A_\gamma}$ with a measure so that $\int f g \,{\rm d}\nu^{A_\gamma}$ and $\int f g \,\nu^{A_\gamma}$ have the same meaning.)

In order to obtain a sharp  Cauchy-Schwarz inequality
\begin{equation}\label{E:cleanCS}
\left| \int_{M_\gamma} f g \,\nu^{A_\gamma} \right| \le 
\|f\|^2_{L^2\left(M_\gamma,\wt{\mu} \right)}
\|g\|^2_{L^2\left(M_\gamma,\wt{\mu}^{*} \right)},
\end{equation}
we set 
$$
c_2=\frac{\left(\gamma(\gamma-1)\right)^{2/3}}{\pi^2}
$$
so that
\begin{equation}\label{E:spec-meas}
\sqrt{\wt{\mu}_{M_\gamma} \, \wt{\mu}_{M_\gamma}^{*,A_\gamma} }
= \dfrac{\gamma^{2}
|\zeta_1|^{\gamma-2}}{8 \pi^2 i}\,d\zeta_2\w d\bar \zeta_1\w d\zeta_1
= \nu^{A_\gamma}.
\end{equation}
For each $f$, there is some non-zero $g$ such that equality holds in \eqref{E:cleanCS}.  (The same is true with the roles of $f$ and $g$ reversed.)  Here, we've allowed $c_2$ to depend not only on the dimension but also on the choice of matrix $A_\gamma$.

This choice of $c_2$ will be fixed throughout the rest of the paper.  We therefore settle on the convention that
\begin{align}\label{E:modified-fefferman-with-constant-settled}
\wt{\mu}_{M_\gamma} &= \frac{\gamma^2 (\gamma-1)^{1/3}}{8\pi^2 i} |\zeta_1|^{\frac{\gamma-2}{3}}\,d\zeta_2\w d\bar \zeta_1\w d\zeta_1 \notag \\
&=  \frac{\gamma^2 (\gamma-1)^{1/3}}{4 \pi^2} \alpha^{\frac{\gamma+1}{3}}ds \wedge d\alpha \wedge d\theta
\end{align}
and
\begin{align}\label{E:modified-dual-fefferman-with-constant-settled}
\wt{\mu}_{M_\gamma}^{*,A_\gamma} &= \frac{\gamma^2 }{8\pi^2 (\gamma-1)^{1/3} i} |\zeta_1|^{\frac{5\gamma-10}{3}}\,d\zeta_2\w d\bar \zeta_1\w d\zeta_1 \notag \\
&=  \frac{\gamma^2 }{4\pi^2 (\gamma-1)^{1/3}} \alpha^{\frac{5\gamma-7}{3}}ds \wedge d\alpha \wedge d\theta.
\end{align}

\subsection{Dual Hardy spaces and the dual Leray transform} \label{SS:dual-Hardy}

Formulas \eqref{E:NuLam}, \eqref{E:modified-fefferman-with-constant-settled} and \eqref{E:modified-dual-fefferman-with-constant-settled} show that each of the distinguished measures from the previous section are contained in some $\sF_r$.  In particular,
\begin{equation}
\nu^{A_\gamma} \in \sF_{\gamma-1}, \qquad \widetilde{\mu}_{M_\gamma} \in \sF_{\frac{\gamma+1}{3}}, \qquad \widetilde{\mu}_{M_\gamma}^{*,A_\gamma} \in \sF_{\frac{5\gamma-7}{3}}.
\end{equation}
Notice that for any $\gamma > 1$, the numbers $\gamma-1, \tfrac{\gamma+1}{3}, \tfrac{5\gamma-7}{3} \in \ci_0 = (-1,2\gamma-1)$.  Therefore, Corollary \ref{C:r-cond} implies that the Leray transform is a bounded operator on all three spaces $L^2(M_\gamma,\nu^{A_\gamma})$, $L^2(M_\gamma, \widetilde{\mu}_{M_\gamma})$ and $L^2(M_\gamma, \widetilde{\mu}_{M_\gamma}^{*,A_\gamma})$.  
 
In what follows, $\left(w^{A_\gamma}\right)^*$ and $\left(w^{A_\gamma}\right)_*$ refer to the pullback and pushforward of $w^{A_\gamma}$, respectively.  In the spirit of \eqref{E:def-of-hardy-space-H_r}, define the Hardy space
\begin{equation*}
H^2(M_\gamma,\wt{\mu}_{M_\gamma}):=\bm{L}_\gamma\big( L^2(M_\gamma,\wt{\mu}_{M_\gamma}) \big) = \ker({{\bm L}_\gamma}-I)
\end{equation*}
and the {\em dual Hardy space}
\begin{equation}\label{E:def-of-dual-Hardy-space}
H^2_{\sf dual} \big(M_\gamma,\wt{\mu}_{M_\gamma}^{*,A_\gamma}\big):=\left(w^{A_\gamma}\right)^*\Big( H^2 \big(M_{\gamma^*}, \wt{\mu}_{{M_{\gamma^*}}} \big)\Big).
\end{equation}
Also define the {\em dual Leray transform}
 \begin{equation}\label{E:def-of-dual-Leray-transform}
 \widetilde{\bm L}_{\gamma}:= \left(w^{A_\gamma}\right)^*\circ {\bm L}_{\gamma^*}\circ \left(w^{A_\gamma}\right)_*.
 \end{equation}

\begin{proposition}\label{P:dual-Hardy-space-alt-rep}
The dual Leray transform $\wt{\bm{L}}_\gamma$ is a bounded projection operator from $L^2\big(M_\gamma,\wt{\mu}_{M_\gamma}^{*,A_\gamma}\big)$ onto $H^2_{\sf dual} \big(M_\gamma,\wt{\mu}_{M_\gamma}^{*,A_\gamma}\big)$.
\end{proposition}
\begin{proof}
We occasionally write $\wt{\mu}^*$ for $\wt{\mu}_{M_\gamma}^{*,A_\gamma}$ when context is clear. Letting $g \in L^2(M_\gamma,\wt{\mu}^*)$ and writing the purely imaginary constant in \eqref{E:modified-dual-fefferman-with-constant-settled} as $\frac{\gamma^2 }{8\pi^2 (\gamma-1)^{1/3} i} := -Ci$, we have
\begin{align}
\norm{g}^2_{\wt{\mu}^*} &= -C i\int_{M_\gamma} |g(\zeta)|^2 |\zeta_1|^{\frac{5\gamma-10}{3}}\,d\zeta_2\w d \bar{\zeta}_1 \wedge d \zeta_1 \notag \\ 
&=  -C i\int_{M_{\gamma^*}} |g \circ w^{-1}|^2 |w_1|^{\frac{2-\gamma}{3\gamma-3}}\,dw_2\w d \bar{w}_1 \wedge d w_1 < \infty.
\end{align}
Since $-C i\, |w_1|^{\frac{2-\gamma}{3\gamma-3}}\,dw_2\w d \bar{w}_1 \wedge d w_1 \in \cf_{\frac{2\gamma-1}{3\gamma-3}}$, we see $g \circ w^{-1} \in L^2\Big(M_{\gamma^*}, \mu_{\frac{2\gamma-1}{3\gamma-3}}\Big)$ and that each function in this space corresponds to a unique such $g$.  Now considering that
$$
\tfrac{2\gamma-1}{3\gamma-3} \in (-1,2\gamma^*-1) = \ci_0
$$ 
(the interval of $r$ values with Leray boundedness on $L^2(M_{\gamma^*},\mu_r)$), Theorem \ref{T:boundedness-mu_r} says that $\bm{L}_{\gamma^*}$ is bounded on $L^2\Big(M_{\gamma^*}, \mu_{\frac{2\gamma-1}{3\gamma-3}}\Big)$. Using $w^{A_\gamma} = w$ to pull the computation back to $M_\gamma$, (since $\wt{\bm{L}}_\gamma g= w^* \circ \bm{L}_{\gamma^*} \circ w_* \circ g$) we conclude that $\norm{ \wt{\bm{L}}_\gamma g }_{\wt{\mu}^*} \lesssim \norm{g}_{\wt{\mu}^*}$, where the implied constant is independent of the particular choice of $g$.  We are also now able to invoke Proposition \ref{P:L_k-is-a-projection-mu_r} to conclude that the dual Leray transform is a projection.

Since $\wt{\mu}_{M_{\gamma^*}} \in \cf_{\frac{2\gamma-1}{3\gamma-3}}$, $H^2 \big(M_{\gamma^*}, \wt{\mu}_{{M_{\gamma^*}}} \big)$ is the image of $L^2\Big(M_{\gamma^*}, \mu_{\frac{2\gamma-1}{3\gamma-3}}\Big)$ under $\bm{L}_{\gamma^*}$.  Definition \eqref{E:def-of-dual-Hardy-space} now lets us conclude
$$
\wt{\bm{L}}_\gamma\Big( L^2\big(M_\gamma,\wt{\mu}_{M_\gamma}^{*,A_\gamma}\big) \Big) = H^2_{\sf dual} \big(M_\gamma,\wt{\mu}_{M_\gamma}^{*,A_\gamma}\big).
$$
\end{proof}

\begin{remark}
It is shown in the Appendix that functions in the Hardy space $H^2(M_\gamma,\wt{\mu}_{M_\gamma})$ are boundary values of holomorphic functions on $\Omega_\gamma$.  Functions in $H^2_{\sf dual} \big(M_\gamma,\wt{\mu}_{M_\gamma}^{*,A_\gamma}\big)$ are pullbacks via $w^{A_\gamma}$ of holomorphic boundary values on $\Omega_{\gamma^*}$.  $\lozenge$
\end{remark}

\begin{remark}
If $Y$ is a non-vanishing vector field annihilating CR functions on $M_{\gamma^*}$, then $H^2_{\sf dual} \big(M_\gamma,\wt{\mu}_{M_\gamma}^{*,A_\gamma}\big)$ can be equivalently be thought of as the set of functions in $L^2\big(M_\gamma, \wt{\mu}_{M_\gamma}^{*,A_\gamma}\big)$ annihilated by $\left(w^{A_\gamma}\right)^* Y$.  (So $Y$ induces the ``projective dual CR structure on $M_\gamma$".  See Section 3 of \cite{BarGru18} for a general discussion.)  Alternatively, Lemma 4.40 in \cite{BarEdh17} leads to the characterizing vector field of the dual Hardy space.  It can be shown that
\begin{equation}
\bar{L}_{\sf dual} := \frac{i}{2} \frac{\dee}{\dee \zeta_1} + \frac{(2-\gamma)i}{2\gamma} \frac{\bar{\zeta}_1}{\zeta_1} \frac{\dee}{\dee \bar{\zeta}_1} - \frac{\gamma}{2}\, \bar{\zeta}_1 |\zeta_1|^{\gamma-2} \frac{\dee}{\dee \zeta_2} + \left( 1- \frac{\gamma}{2} \right) \bar{\zeta}_1 |\zeta_1|^{\gamma-2} \frac{\dee}{\dee \bar{\zeta}_2}
\end{equation}
annihilates all functions in $H^2_{\sf dual} \big(M_\gamma,\wt{\mu}_{M_\gamma}^{*,A_\gamma}\big)$. 
$\lozenge$
\end{remark}

$\widetilde{\bm L}_{\gamma}$ is now re-expressed as a single integral:

\begin{proposition}\label{P:dual-Leray-as-an-integral}
Let $g \in L^2\big(M_\gamma,\wt{\mu}_{M_\gamma}^{*,A_\gamma}\big)$.  The dual Leray transform is obtained by calculating the following integral
\begin{equation}\label{E:dual-Leray-as-an-integral}
\widetilde{\bm L}_{\gamma}(g)(z) =  \frac{\gamma^2}{8\pi^2 i} \int_{M_\gamma} g(\zeta) \frac{|\zeta_1|^{\gamma-2} d\zeta_2\wedge d\bar{\zeta}_1\wedge d\zeta_1 }{\left[\gamma\bar{z}_1|z_1|^{\gamma-2}(z_1-\zeta_1)+i(z_2-\zeta_2)\right]^2}.
\end{equation}
\end{proposition}
\begin{proof}
The map $w^{A_\gamma} = w$ given by \eqref{E:Mgamma-dual-hypersurface-w-eqns} is a diffeomorphism from $M_\gamma$ onto $M_{\gamma^*}$.  To help keep track of notation, write
\begin{align*}
\zeta &= (\zeta_1,\zeta_2) \in M_\gamma, \qquad w(\zeta) := \zeta^* = (\zeta_1^*,\zeta_2^*) \in M_{\gamma^*}, \\
z &= (z_1,z_2) \in M_\gamma, \qquad w(z) := z^* = (z_1^*,z_2^*) \in M_{\gamma^*}.
\end{align*}

Now transfer the computation from $M_{\gamma}$ to $M_{\gamma^*}$ by considering the pushforward of $g$, denoted $G := w_*(g)$.  The proof of Proposition \ref{P:dual-Hardy-space-alt-rep} shows both that $G \in L^2\Big(M_{\gamma^*}, \mu_{\frac{2\gamma-1}{3\gamma-3}}\Big)$ and that $\bm{L}_{\gamma^*}$ is a bounded operator on this space.
\begin{align}\label{E:Leray-with-everything-starred}
\wt{\bm{L}}_\gamma(g)(z) = (w^* \circ \bm{L}_{\gamma^*} \circ w_* \circ g)(z) &=\left( \bm{L}_{\gamma^*}(g \circ w^{-1})\right)(w(z)) = \bm{L}_{\gamma^*}(G)(z^*).
\end{align}
Now,
\begin{align}
\bm{L}_{\gamma^*}(G)(z^*) &= \frac{(\gamma^*)^2}{8\pi^2 i} \int_{M_{\gamma^*}} G(\zeta^*) \frac{|\zeta_1^*|^{{\gamma^*}-2} d\zeta^*_2\wedge d\bar{\zeta^*_1} \wedge d\zeta^*_1 }{\left[\gamma^*\bar{\zeta_1^*}|\zeta_1^*|^{\gamma^*-2}(\zeta_1^*-z_1^*)+i(\zeta_2^*-z_2^*)\right]^2} \notag \\
&= \frac{\gamma^2}{8\pi^2(\gamma-1)^2 i} \int_{M_{\gamma^*}} G(\zeta^*) \frac{|\zeta_1^*|^{\frac{2-\gamma}{\gamma-1}} d\zeta^*_2\wedge d\bar{\zeta^*_1}\wedge d\zeta^*_1 }{\left[(\frac{\gamma}{\gamma-1})\bar{\zeta^*_1}|\zeta_1^*|^{\frac{2-\gamma}{\gamma-1}}(\zeta_1^*-z_1^*)+i(\zeta_2^*-z_2^*)\right]^2} \notag \\
&= \frac{\gamma^2}{8\pi^2 i} \int_{M_{\gamma^*}} G(\zeta^*) \frac{|\zeta_1^*|^{\frac{2-\gamma}{\gamma-1}} d\zeta^*_2\wedge d\bar{\zeta^*_1}\wedge d\zeta^*_1 }{\left[\gamma |\zeta_1^*|^{\frac{\gamma}{\gamma-1}} - \gamma z_1^* \bar{\zeta^*_1} |\zeta_1^*|^{\frac{2-\gamma}{\gamma-1}}+i(\gamma-1)(\zeta_2^*-z_2^*)\right]^2}. \label{E:integral-in-*variables}
\end{align}

Now we wish to move the computation back to $M_\gamma$.  Writing \eqref{E:wj} and \eqref{E:reverse} in the notational convention adopted above we see that
\begin{align}
\zeta_1^* &= \bar{\zeta_1} |\zeta_1|^{\gamma-2}, \qquad \zeta_2^* = \frac{\zeta_2-i \gamma |\zeta_1|^\gamma}{1-\gamma}, \label{E:zeta*-in-terms-of-zeta} \\
\zeta_1 &= \bar{\zeta_1^*} |\zeta_1^*|^{\frac{2-\gamma}{\gamma-1}}, \qquad \zeta_2 = (1-\gamma) \zeta_2^* + i \gamma |\zeta_1|^{\frac{\gamma}{\gamma-1}}.\label{E:zeta-in-terms-of-zeta*}
\end{align}
The same relations also hold between the $z$ and $z^*$ variables. From here, it follows that
\begin{equation}\label{E:zeta*-differential}
d\zeta^*_2\wedge d\bar{\zeta^*_1}\wedge d\zeta^*_1 = |\zeta_1|^{2\gamma-4} d\zeta_2\wedge d\bar{\zeta_1}\wedge d\zeta_1.
\end{equation}

Using \eqref{E:zeta*-in-terms-of-zeta} and \eqref{E:zeta-in-terms-of-zeta*}, rewrite the term inside the brackets appearing in the denominator in \eqref{E:integral-in-*variables}: 
\begin{align}
& \gamma |\zeta_1^*|^{\frac{\gamma}{\gamma-1}} - \gamma z_1^* \bar{\zeta^*_1} |\zeta_1^*|^{\frac{2-\gamma}{\gamma-1}}+i(\gamma-1)(\zeta_2^*-z_2^*) \notag\\
= \,\,& \gamma |\zeta_1|^\gamma - \gamma \zeta_1 \bar{z_1}|z_1|^{\gamma-2} + \gamma|z_1|^\gamma - \gamma |\zeta_1|^\gamma + i (z_2 -\zeta_2) \notag \\
= \,\,& \gamma \bar{z_1} |z_1|^{\gamma-2} (z_1 - \zeta_1) + i (z_2 - \zeta_2) \label{E:denom-rewritten}
\end{align}

Inserting equations \eqref{E:zeta*-differential} and \eqref{E:denom-rewritten} into \eqref{E:integral-in-*variables} shows
\begin{equation*}
\eqref{E:integral-in-*variables} =  \frac{\gamma^2}{8\pi^2 i} \int_{M_\gamma} g(\zeta) \frac{|\zeta_1|^{\gamma-2} d\zeta_2\wedge d\bar{\zeta}_1\wedge d\zeta_1 }{\left[\gamma\bar{z}_1|z_1|^{\gamma-2}(z_1-\zeta_1)+i(z_2-\zeta_2)\right]^2},
\end{equation*}
which is equal to $\widetilde{\bm L}_{\gamma}(g)(z)$ by \eqref{E:Leray-with-everything-starred}.
\end{proof}

The preceding proposition demonstrates the close connection between the dual Leray transform and the adjoint of the Leray transform with respect to the measure $\sigma$:

\begin{corollary}\label{C:dual-Leray-in-terms-of-adjoint}
The dual Leray transform on $M_\gamma$ can be expressed as
\begin{align*}
\wt{\bm L}_{\gamma} = \mathfrak{c} \circ \bm{L}^{(*,\sigma)}_{\gamma} \circ \mathfrak{c},
\end{align*}
where $\mathfrak{c}$ is the conjugation operator and $\bm{L}^{(*,\sigma)}$ is given by the integral in \eqref{E:leray-adjoint-wrt-sigma}.
\end{corollary}
\begin{proof}
This is immediate from formulas \eqref{E:leray-adjoint-wrt-sigma} and \eqref{E:dual-Leray-as-an-integral}.
\end{proof}

\begin{remark}\label{R:boundedness-of-the-dual-Leray-transform}
We've seen that $\bm{L}_\gamma$ is bounded on $L^2(M_\gamma,\mu_r)$ if and only if $r \in \ci_0 = (-1,2\gamma-1)$. For a fixed $r$ in this range, choose $r'$ with $\frac{r+r'}{2} = \gamma-1$, and let $f \in L^2(M_\gamma,\mu_r)$, $g \in L^2(M_\gamma,\mu_{r'})$.  By symmetry (Remark \ref{R:symmetry-of-interval-I_0}), $r'$ is also contained in $\ci_0$.  General theory (or, alternatively, the proof of Theorem \ref{T:leray-adjoint-wrt-mu_r}) shows 
$$
\langle \bm{L}_\gamma f, g \rangle_\sigma = \langle  f, \bm{L}_\gamma^{(*,\sigma)} g \rangle_\sigma,
$$ 
implying that
\begin{equation*}
\norm{\bm{L}_\gamma}_{\mu_r} = \normm{\bm{L}_\gamma^{(*,\sigma)}}_{\mu_{r'}} = \normm{\wt{\bm{L}}_\gamma}_{\mu_{r'}}.
\end{equation*}
This shows that the dual Leray transform $\wt{\bm L}_{\gamma}$ is bounded on $L^2(M_\gamma,\mu_{r'})$ for all $r' \in \ci_0$. In particular, 
\begin{equation*}
\norm{\bm{L}_\gamma}_{\wt{\mu}}  = \normm{\wt{\bm{L}}_\gamma}_{\wt{\mu}^*}.
\end{equation*}
where $\wt{\mu}$ and $\wt{\mu}^*$ are the preferred measure and preferred dual measure, respectively.
$\lozenge$
\end{remark}

\subsection{Pairing Hardy spaces}\label{SS:bilinear-pairing-of-hardy-spaces}
For $f\in L^2\left(M_\gamma,\wt{\mu}_{M_\gamma} \right)$, $g\in  L^2\big(M_\gamma,\wt{\mu}_{M_\gamma}^{*,A_\gamma} \big)$, define the {\em bilinear} pairing
\begin{equation}\label{E:def-of-bilinear-pairing}
\langle\!\langle f, g  \rangle\!\rangle = \int_{M_\gamma} f(\zeta) g(\zeta)\, \nu^{A_\gamma}(\zeta).
\end{equation}
Continue to write $\wt{\mu}$  and $\wt{\mu}^*$  for $\wt{\mu}_{M_\gamma}$ and $\wt{\mu}_{M_\gamma}^{*,A_\gamma}$, respectively.  This pairing facilitates the representation of linear functionals on $H^2(M_\gamma,\wt{\mu})$ by functions in $H^2_{\sf dual} \left(M_\gamma,\wt{\mu}^*\right)$.  In Theorem \ref{T:InvNorm} below we prove that the efficiency of this representation is closely connected to the Leray transform.  We point out that:
\begin{itemize}
\item[(1)]  The preferred measure $\wt{\mu}$ and preferred dual measure $\wt{\mu}^*$ are projectively invariant, so the corresponding Hardy spaces are too.  Compare this to both Theorem 2 and Section 8 in \cite{Bar16}.
\item[(2)]  Pairing $H^2(M_\gamma,\wt{\mu})$ and $H^2_{\sf dual} \left(M_\gamma,\wt{\mu}^*\right)$ via \eqref{E:def-of-bilinear-pairing} is closely tied to the universal expression of the Leray transform given in \eqref{E:GenMLeray}.
\end{itemize}

\begin{lemma}\label{L:PairingAdjoint}
For $f\in L^2\left(M_\gamma,\wt{\mu} \right)$, $g\in  L^2\big(M_\gamma,\wt{\mu}^* \big)$ we have
\begin{equation}\label{E:PairingAdjoint}
\langle\!\langle \bm{L}_\gamma f, g  \rangle\!\rangle = \langle\!\langle f, \widetilde{\bm L}_{\gamma} g  \rangle\!\rangle  = \langle\!\langle \bm{L}_\gamma f, \widetilde{\bm L}_{\gamma} g  \rangle\!\rangle.
\end{equation}
\end{lemma}

\begin{proof} Since $\bm{L}_\gamma$ and $\wt{\bm{L}}_\gamma$ are projections, it suffices to prove the first equality.  Notice that 
\begin{align*}\label{E:dual-Leray-pairing-step-one}
\langle\!\langle \bm{L}_\gamma f, g  \rangle\!\rangle  = \langle \bm{L}_\gamma f, \overline{g}  \rangle_{\nu} = \langle f, \bm{L}_\gamma^{(*\sigma)} \overline{g}  \rangle_{\nu} = \langle\!\langle  f, \mathfrak{c} \circ \bm{L}^{(*,\sigma)}_{\gamma} \circ \mathfrak{c}(g) \rangle\!\rangle = \langle\!\langle f, \widetilde{\bm L}_{\gamma} g  \rangle\!\rangle. 
\end{align*}
The second equality is justified by the proof of Theorem \ref{T:leray-adjoint-wrt-mu_r} and the final equality is just Corollary \ref{C:dual-Leray-in-terms-of-adjoint}.
\end{proof}

\begin{remark}
Formula \eqref{E:PairingAdjoint} could also be proved by adapting the Plemelj-formula-based argument found in Theorem 25 of \cite{Bar16} to the current unbounded setting. $\lozenge$
\end{remark}

Now consider the map
\begin{align*}
\chi_\gamma\colon L^2\big(M_\gamma, \wt{\mu}^* \big) \to \left(L^2\left(M_\gamma, \wt{\mu} \right)\right)^*
\end{align*}
given by
\begin{align*}
\chi_\gamma(g)\colon f \to \langle\!\langle f, g  \rangle\!\rangle,
\end{align*}
along with the companion map
\begin{align*}
\chih_\gamma\colon H^2_{\sf dual} \big(M_\gamma, \wt{\mu}^*\big) \to \left(H^2\big(M_\gamma, \wt{\mu}\big)\right)^*
\end{align*}
given by
\begin{align*}
\chih_\gamma(g)= \chi_\gamma(g)\Big{|}_{H^2(M_\gamma,\wt{\mu})}.
\end{align*}
Note that $\chi_\gamma$ and $\chih_\gamma$ are linear ({\em not} conjugate-linear).  From \eqref{E:cleanCS} we see that $\chi_\gamma$ is an isometry and that 
\begin{align*}
\left\|\chih_\gamma\right\|&\le 1.
\end{align*}

Define a map $\wt{\bm{R}} : L^2\big(M_\gamma,\wt{\mu}^* \big) \to L^2\left(M_\gamma,\wt{\mu} \right)$ by
\begin{equation}
f(\zeta) \longmapsto \sqrt{ \frac{{\rm d} \wt{\mu}_{M_\gamma}^{*,A_\gamma} }{ {\rm d}\wt{\mu}_{M_\gamma}} } f(\zeta) = \frac{|\zeta_1|^{\frac23(\gamma-2)}}{(\gamma-1)^{1/3}} f(\zeta) := \wt{\bm{R}}f,
\end{equation}
where the Radon-Nikodym derivative is calculated in \eqref{E:RadNik}.
\begin{lemma}
The map $\wt{\bm{R}} : L^2\big(M_\gamma,\wt{\mu}_{M_\gamma}^{*,A_\gamma}  \big) \to L^2\left(M_\gamma,\wt{\mu}_{M_\gamma} \right)$ is an isometry.
\end{lemma}
\begin{proof}
The appearance of the Radon-Nikodym derivative makes this clear.
\end{proof}

\begin{theorem}\label{T:InvNorm}
The operator $\chih_\gamma\colon H^2_{\sf dual} \big(M_\gamma, \wt{\mu}^* \big) \to \left(H^2\big(M_\gamma, \wt{\mu}\big)\right)^*$ is invertible with norm
\begin{equation}\label{E:InvNorm}
\left\| {\chih}_{\gamma}^{\,-1} \right\| = \norm{\bm{L}_\gamma}_{\wt{\mu}}.
\end{equation}
\end{theorem}

\begin{remark}
If $M_\gamma$ were the boundary of a smoothly bounded strongly $\C$-convex domain in $\C^2$  (or $\C\p^{2}$), then Theorem \ref{T:InvNorm} would just be a paraphrase of Corollary 26 from \cite{Bar16}; the argument offered below is an adaptation of that earlier proof.
$\lozenge$
\end{remark}

\begin{proof}

Let $g\in\ker\chih_\gamma$, i.e., $\chih_\gamma(g)$ is the zero functional on $H^2\big(M_\gamma, \wt{\mu}\big)$.  It can be easily checked that $\bm{L}_\gamma \wt{\bm{R}}\,\bar{g} \in H^2 \big(M_\gamma, \wt{\mu}\big)$.  This means that
\begin{align*}
0 = \chih_\gamma(g)\left( \bm{L}_\gamma \wt{\bm{R}}\,\bar{g} \right) = \left\langle\!\left\langle \bm{L}_\gamma \wt{\bm{R}}\,\bar{g} , g  \right\rangle\!\right\rangle = \left\langle\!\left\langle  \wt{\bm{R}}\,\bar{g} , \wt{\bm{L}}_\gamma g  \right\rangle\!\right\rangle = \left\langle\!\left\langle  \wt{\bm{R}}\,\bar{g} ,  g  \right\rangle\!\right\rangle = \norm{g}_{\wt{\mu}^*}^2.
\end{align*}
Thus $g\equiv0$ and $\chih_\gamma$ is injective.

Let $T \in \left(H^2\big(M_\gamma, \wt{\mu} \big)\right)^*$.  There is some $h \in L^2\big(M_\gamma, \wt{\mu}^*  \big)$ representing $T$ and satisfying
\begin{equation*}
\|h\|_{L^2(M_\gamma, \wt{\mu}^{*})} =\|T\|_{H^2(M_\gamma, \wt{\mu})^*}.
\end{equation*}
Now for all $f\in H^2(M_\gamma, \wt{\mu})$,
\begin{align*}
Tf &= \left\langle\!\left\langle f , h  \right\rangle\!\right\rangle = \left\langle\!\left\langle \bm{L}_\gamma f , h  \right\rangle\!\right\rangle = \left\langle\!\left\langle  f , \wt{\bm{L}}_\gamma h  \right\rangle\!\right\rangle = \chih_\gamma\left(\widetilde{\bm L}_{\gamma}h\right)(f)
\end{align*}
In other words, $T=\chih_\gamma\left(\widetilde{\bm L}_{\gamma}h\right)$ with $ \|\wt{\bm L}_{\gamma}h\|_{\wt{\mu}^*} \le \|\widetilde{\bm L}_{\gamma}\|_{\wt{\mu}^*} \|h\|_{\wt{\mu}^*} =  \|{\bm L}_{\gamma}\|_{\wt{\mu}} \|T\|$.
So $ \chih_\gamma$ is surjective with $\| \chih_\gamma^{\,-1} \| \le \|{\bm L}_{\gamma}\|_{\wt{\mu}} $.

To prove the reverse estimate let $f\in H^2(M_\gamma,\wt{\mu})$ and pick $h \in L^2\big(M_\gamma, \wt{\mu}^*  \big)$ such that $\norm{h}_{\wt{\mu}^*} = 1$.  Then
\begin{align*}
\left|\chih_\gamma(\widetilde{\bm L}_{\gamma}h)(f)\right| = \left| \langle\!\langle f, \widetilde{\bm L}_{\gamma} h  \rangle\!\rangle \right| = \left| \langle\!\langle \bm{L}_\gamma f, h  \rangle\!\rangle \right| = \left| \langle\!\langle f, h  \rangle\!\rangle \right| \le \|f\|_{\wt{\mu}}
\end{align*}
and so $\|\chih_\gamma(\widetilde{\bm L}_{\gamma}h)\|\le 1$ for any such $h$.  Let $\epsilon>0$ and choose a specific such $h$ so that $\|\wt{\bm L}_{\gamma}h\|_{\wt{\mu}^*} \ge \|{\bm L}_{\gamma}\|_{\wt{\mu}} -\epsilon$.  It easily follows that $\|\chih_\gamma^{-1}\| \ge \|\wt{\bm L}_{\gamma}h\|_{\wt{\mu}^*} \ge \|{\bm L}_{\gamma}\|_{\wt{\mu}} -\epsilon$. 
Since $\epsilon$ was arbitrary we have $\|\chih_\gamma^{-1}\|\ge\|{\bm L}_{\gamma}\|_{\wt{\mu}}$, concluding the proof.
\end{proof}

See Section 9 in \cite{Bar16} for related results.

\appendix

\section{Constructing holomorphic functions on $\Omega_\gamma$}\label{S:Appendix}

We verify here that for certain values of $r$, the Leray transform maps functions $f\in L^2(M_\gamma,\mu_r)$ to holomorphic functions on the domain $\Omega_{\gamma} = \left\{ (z_1,z_2)\in \C^2\colon \im(z_2) > \left|z_1\right|^\gamma\right\}.$

The main result is the following:
\begin{theorem}\label{T:range-of-r-building-holomorphic-full-Leray}
Let $\bm{L}$ denote the Leray transform.  $\bm{L}$ maps $L^2(M_\gamma,\mu_r) \to \co(\Omega_\gamma)$ for each $r \in \cj_0 := (-\gamma-1, 2\gamma-1)$.
\end{theorem}
Notice that $\cj_0 \supsetneq \ci_0 = (-1,2\gamma-1)$, the interval of $r$ values for which $\bm{L}$ is bounded from $L^2(M_\gamma,\mu_r) \to L^2(M_\gamma,\mu_r)$.  In Section \ref{SS:related-operators}, we defined the Hardy space $H^2(M_\gamma,\mu_r) := \bm{L}\big(L^2(M_\gamma,\mu_r)\big)$ for each $r \in \ci_0$.  Theorem \ref{T:range-of-r-building-holomorphic-full-Leray} confirms that these Hardy spaces can be viewed as consisting entirely of boundary values of holomorphic functions.

\subsection{Symmetries of the Leray kernel}

Partition $\Omega_\gamma$ into translated copies of $M_\gamma$ by defining
\begin{equation}\label{D:MGammaEpsilon}
M_\gamma^\epsilon = \left\{ (z_1,z_2) :  \im(z_2) = |z_1|^\gamma + \epsilon \right\}.
\end{equation}
It is clear that $\Omega_\gamma = \bigcup_{\epsilon > 0} M_\gamma^\epsilon$ and that this is a union of disjoint sets.

Let $z \in \Omega_\gamma$.  Equations \eqref{E:ell_M(z,zeta)}, \eqref{E:LerayLevi} and \eqref{E:LerayMgamma} provide the formula for the Leray transform:
\begin{equation}\label{E:Leray-transform-appendix}
\bm{L}f(z) = \int_{M_\gamma} f(\zeta)\ell_\rho(z,\zeta) \lambda_\rho(\zeta) = \frac{\gamma^2}{{16 \pi^2 } } \int_{M_\gamma} f(\zeta)\ell_\rho(z,\zeta) \sigma(\zeta),
\end{equation}
where the kernel $\ell_\rho = \ell$ and measure $\sigma$ are
\begin{align*}
\ell(z,\zeta) &= \frac{4}{\left(\gamma\bar{\zeta}_1|\zeta_1|^{\gamma-2}(\zeta_1-z_1)+i(\zeta_2-z_2)\right)^2}, \qquad \sigma(\zeta) = \frac{1}{2 i} |\zeta_1|^{\gamma-2}\, d\zeta_2\wedge d\bar{\zeta}_1\wedge d\zeta_1.
\end{align*}

The measure $\sigma$ is special instance of $\mu_r$, given in rectangular coordinates for $r\in \R$ by
\begin{align*}
\mu_r(\zeta) = \frac{1}{2 i} |\zeta_1|^{r-1}\, d\zeta_2\wedge d\bar{\zeta}_1\wedge d\zeta_1.
\end{align*}
Recall the three types of automorphisms of $M_\gamma$ from Section \ref{SS:Symmetries}: $t_s$ (translations in $\re(\zeta_2)$), $r_\theta$ (rotations in $\zeta_1$), and $\delta_a$ (non-isotropic dilations).  These maps transform $\ell(z,\zeta)$ in the following way:
\begin{subequations}
\begin{align}
\ell(t_s(z),t_s(\zeta)) &= \ell(z,\zeta), \qquad\qquad s\in \R \label{E:Leray-interacting-with-t_s}\\
\ell(r_\theta(z),r_\theta(\zeta)) &= \ell(z,\zeta), \qquad\qquad \theta\in[0,2\pi) \label{E:Leray-interacting-with-r_theta}\\
\ell(\delta_a(z),\delta_a(\zeta)) &= a^{-2\gamma}\ell(z,\zeta), \qquad a >0. \label{E:Leray-interacting-with-delta_alpha}
\end{align}
The measure $\mu_r$ is easily seen to be invariant under such translations and rotations.  Further computation reveals that
\begin{align}\label{E:mu_r-interacting-with-delta_alpha}
\delta_a^*\mu_r(\zeta) &= a^{\gamma+r+1}\mu_r(\zeta).
\end{align}
\end{subequations}

\subsection{$L^2$-norms of the kernel function}
For $r \in \R$, define $\Theta_r:\Omega_\gamma \to [0,\infty]$ by
\begin{equation}\label{E:def-of-Theta_r}
\Theta_r(z) := \int_{M_\gamma} |\ell(z,\zeta)|^2\,\mu_r(\zeta).
\end{equation}

Applying $t_s$ and $r_\theta$ for appropriate values of $s$ and $\theta$, we see from \eqref{E:Leray-interacting-with-t_s} and \eqref{E:Leray-interacting-with-r_theta} that for $z \in \Omega_\gamma$,
\begin{equation}\label{E:Theta_r-invariance-translations-rotations}
\Theta_r(z) = \Theta_r(z_1,z_2) = \Theta_r(|z_1|,i\, \im(z_2)).
\end{equation}
Also observe that for $a>0$, \eqref{E:Leray-interacting-with-delta_alpha} and \eqref{E:mu_r-interacting-with-delta_alpha} imply
\begin{align}
\Theta_r(\delta_a(z)) = \int_{M_\gamma} |\ell(\delta_a(z),\zeta)|^2\,\mu_r(\zeta) &= \int_{M_\gamma} |\ell(\delta_a(z),\delta_a(\zeta))|^2\,\delta_a^*\mu_r(\zeta) \notag\\
&= a^{r+1-3\gamma}\, \Theta_r(z). \label{E:Theta_r-transformation-dilations}
\end{align}

Now let $\epsilon > 0$ be the unique number with $z \in M_\gamma^\epsilon$.  If $z_1 \neq 0$, set $a = |z_1|^{-1}$ and combine \eqref{E:Theta_r-invariance-translations-rotations} and \eqref{E:Theta_r-transformation-dilations} to see
\begin{align}
\Theta_r(z) = \Theta_r(|z_1|,i\, \im(z_2)) &= a^{3\gamma-r-1} \Theta_r(\delta_a(|z_1|,i\, \im(z_2))) \notag \\
&= |z_1|^{r+1-3\gamma} \Theta_r\Big(1,i\, \tfrac{\im(z_2)}{|z_1|^\gamma}\Big) \notag \\
&= |z_1|^{r+1-3\gamma} \Theta_r\Big(1,i \left( 1 + \tfrac{\epsilon}{|z_1|^\gamma}\right) \Big). \label{E:Theta_r_normalized_z1=/=0}
\end{align}
Otherwise if $z_1 = 0$,
\begin{equation}\label{E:Theta_r_z1=0}
\Theta_r(z) = \Theta_r(0,i\,\im{z_2}) = \Theta_r(0,i\epsilon).
\end{equation}

Now endow $\Omega_\gamma$ with a coordinate system based on the $(\alpha, \theta, s)$-coordinatization of $M_\gamma$ that was introduced in Section \ref{SS:LerayReparam}.  
For $z \in M_\gamma^\epsilon$, write
\begin{align}\label{E:Omega_gammaParam}
z &= (\alpha e^{i\theta}, s + i(\alpha^\gamma+\epsilon)).
\end{align}
Points in $\Omega_\gamma$ may now be described in $(\alpha, \theta, s, \epsilon)$-coordinates, where $\alpha \ge 0$, $\theta \in [0,2\pi)$, $s \in \R$, $\epsilon >0$.  In these coordinates, denote the quantity appearing in the innermost set of parentheses of \eqref{E:Theta_r_normalized_z1=/=0} as follows:
\begin{equation}\label{E:def-of-q(alpha,epsilon)}
1 + \frac{\epsilon}{|z_1|^\gamma} = 1 + \frac{\epsilon}{\alpha^\gamma} := q(\alpha,\epsilon).
\end{equation}
The goal is now to understand the function $z \mapsto \Theta_r(z)$ from the perspective of
\begin{equation}\label{E:Theta_r-written-in-alpha-epsilon-variables}
(\alpha, \theta, s, \epsilon) \mapsto \alpha^{r+1-3\gamma} \Theta_r(1, i q(\alpha,\epsilon)).
\end{equation}

\begin{proposition}\label{P:Theta_r(1,iq)-calculation}
Let $r \in \cj'_0 := (-1, 3\gamma-1)$ and $q >1$.  Then 
$$
\Theta_r\left(1,i q \right) = q^{\frac{1+r}{\gamma} - 3} F_r(q),
$$
where $F_r$ is positive, strictly decreasing and real-analytic on $(1,\infty)$.  Furthermore,
\begin{itemize}
\item[$(a)$] $\lim_{q \to 1^+ }F_r(q) = \infty.$
\item[$(b)$] $\lim_{q \to \infty} F_r(q)$ is a positive number.
\end{itemize}
If $r \notin \cj'_0$, then $\Theta_r\left(1,i q \right) = \infty$.
\end{proposition}
\begin{proof}
Denote $\zeta \in M_\gamma$ by $\zeta = (\alpha_\zeta e^{i\theta_\zeta}, s_\zeta + i\alpha_\zeta^\gamma)$, which yields
\begin{equation*}
\ell((1,iq),\zeta) = \frac{4}{\big((\gamma-1)\alpha_\zeta^\gamma + q + is_\zeta - \gamma \alpha_\zeta^{\gamma-1}e^{-i \theta_\zeta}\big)^2}.
\end{equation*}
Therefore
\begin{align}
\Theta_r(1,iq) &= 16 \int_{M_\gamma} \frac{\alpha_\zeta^{r}\, ds_\zeta\wedge d\alpha_\zeta\wedge d\theta_\zeta}{\big|(\gamma-1)\alpha_\zeta^\gamma + q + i s_\zeta - \gamma \alpha_\zeta^{\gamma-1} e^{-i\theta_\zeta} \big|^4} \notag\\
&= 16 \int_0^\infty \alpha_\zeta^{r} \int_{-\infty}^\infty \left\{ \int_0^{2\pi} \frac{d\theta_\zeta}{|A-B e^{-i \theta_\zeta}|^4} \right\}\,ds_\zeta\,d\alpha_\zeta \label{E:AppendixInt1},
\end{align}
where
\begin{equation}\label{D:AppendixDefAB}
A = (\gamma-1)\alpha_\zeta^\gamma + q + i s_\zeta, \qquad B = \gamma \alpha_\zeta^{\gamma-1}.
\end{equation}
Since $q>1$, Lemma \ref{L:GeneralAMGM} shows that $|B| < |A|$ for all choices of $\alpha_\zeta, s_\zeta$.  Focus now on the innermost integral in \eqref{E:AppendixInt1}:
\begin{align}
\int_0^{2\pi} \frac{d\theta_\zeta}{|A-B e^{i \theta_\zeta}|^4} &= \int_0^{2\pi} \frac{1}{(A-B e^{-i \theta_\zeta})^2} \cdot \bar{\frac{1}{(A-B e^{-i \theta_\zeta})^2}}\,  d\theta_\zeta \notag\\
&= \frac{1}{|A|^4} \int_0^{2\pi} \sum_{j,k=0}^\infty (j+1)(k+1) \left( \frac{B}{A} \right)^j \bar{\left( \frac{B}{A} \right)^k} e^{i(k-j)\theta_\zeta}\,d\theta_\zeta \notag \\
&= \frac{2\pi}{|A|^4} \sum_{k=0}^\infty (k+1)^2 \left|\frac{B}{A}\right|^{2k}. \label{E:AppendixThetaIntegral}
\end{align}
Plugging \eqref{E:AppendixThetaIntegral} into \eqref{E:AppendixInt1} gives
\begin{align}
\Theta_r(1,iq) &= 32\pi \sum_{k=0}^\infty (k+1)^2 \int_0^\infty \alpha_\zeta^{r} \left|B^{2k}\right| \left\{\int_{-\infty}^\infty \frac{d s_\zeta}{|A|^{2k+4}} \right\}\,d\alpha_\zeta \notag \\
& =32\pi \sum_{k=0}^\infty (k+1)^2 \gamma^{2k} \int_0^\infty \alpha_\zeta^{2k(\gamma-1)+r} \left\{\int_{-\infty}^\infty \frac{d s_\zeta}{|(\gamma-1)\alpha_\zeta^\gamma + q + i s_\zeta|^{2k+4}} \right\}\,d\alpha_\zeta. \label{E:AppendixIntegral2}
\end{align}

We now import the first of two integral formulas which are established in Section \ref{SS:appendix-integral-formulas}.  Proposition \ref{P:AppendixSIntegral} shows that for any $C>0$,
\begin{equation}\label{E:appendix-s-integral-formula}
\int_{-\infty}^\infty \frac{ds}{|C+is|^{2k+4}} = \frac{\pi}{4^{k+1}C^{2k+3}} \frac{\Gamma(2k+3)}{\Gamma(k+2)^2}.
\end{equation}
Substituting this into \eqref{E:AppendixIntegral2} with $C = (\gamma-1)\alpha_\zeta^\gamma + q$ shows
\begin{equation}\label{E:AppendixIntegral3}
\Theta_r(1,iq) =  8\pi^2 \sum_{k=0}^\infty \frac{\Gamma(2k+3)}{\Gamma(k+1)^2} \left( \frac{\gamma}{2} \right)^{2k} \int_0^\infty \frac{\alpha_\zeta^{2k(\gamma-1)+r}}{((\gamma-1)\alpha_\zeta^\gamma + q)^{2k+3}} d\alpha_\zeta.
\end{equation}

For each fixed $k$, two conditions on $r$ are required for the integral in \eqref{E:AppendixIntegral3} to converge:  
\begin{itemize}
\item[(1)] For $\alpha_\zeta$ near $0$, we need $2k(\gamma-1)+r > -1$.
\item[(2)] As $\alpha_\zeta$ tends to $\infty$, we need $\gamma(2k+3) - 2k(\gamma-1)-r > 1$.  
\end{itemize}
These conditions suggest the definition of the following interval of $r$ values
\begin{equation}\label{def-of-interval-J'_k}
\cj'_k := (- 2k(\gamma-1) -1  , 3\gamma - 1 +2k).
\end{equation}
Note that $\cj'_0 \subset \cj'_k$ for all $k$, so the convergence of the $k=0$ integral implies the convergence of all other integrals in the sum.  But if the $k=0$ integral diverges, i.e., $r \notin \cj'_0$, then $\Theta_r(1,iq) = \infty$.  Now make use of the second integral formula established in Section \ref{SS:appendix-integral-formulas}.  Proposition \ref{P:AppendixAlphaIntegral} shows that, for $q>0$, $\gamma>1$, $r \in \cj'_k$ and $k$ a non-negative integer,
\begin{equation}\label{E:appendix-alpha-integral-formula}
\int_0^\infty \frac{\alpha^{2k(\gamma-1)+r}}{((\gamma-1)\alpha^\gamma + q)^{2k+3}} d\alpha =  \frac{\Gamma\big(3 + \frac{2k-1-r}{\gamma} \big) \Gamma\big(2k +  \frac{1+r-2k}{\gamma} \big) }{ \gamma (\gamma-1)^{\frac{1+r}{\gamma}} q^{3-\frac{1+r}{\gamma}} \Gamma\big(2k+3\big)  } \left( q (\gamma-1)^{\gamma-1}  \right)^{-\frac{2k}{\gamma}}.
\end{equation}
Plugging this into \eqref{E:AppendixIntegral3} yields
\begin{equation}\label{E:AppendixIntegral4}
\Theta_r(1,iq) = \frac{8\pi^2}{\gamma (\gamma-1)^{\frac{1+r}{\gamma}}} q^{\frac{1+r}{\gamma}-3} \sum_{k=0}^\infty \tfrac{\Gamma\left(3 + \frac{2k-1-r}{\gamma} \right) \Gamma\left(2k +  \frac{1+r-2k}{\gamma} \right) }{\Gamma\big(k+1\big)^2} \left( \frac{\gamma^2}{ 4(\gamma-1)^{2-\frac{2}{\gamma}} q^{\frac{2}{\gamma}}  } \right)^{k}.
\end{equation}

We study the convergence of this sum by considering (for $r \in \cj'_k$) the related series 
\begin{equation}\label{E:appendix-series-ak}
\sum_{k=0}^\infty a_k x^k, \qquad \textrm{where} \qquad a_k = \frac{\Gamma\big(3 + \frac{2k-1-r}{\gamma} \big) \Gamma\big(2k +  \frac{1+r-2k}{\gamma} \big) }{\Gamma\big(k+1\big)^2}.
\end{equation}
Sterling's formula says
\begin{align*}
\Gamma({k+1})^2 &\sim 2\pi e^{-2k} k^{2k+1}, \\
\Gamma\big(3 + \tfrac{2k-1-r}{\gamma} \big) &\sim  \sqrt{2\pi} \cdot e^{\frac{1+r-2k}{\gamma}-2} \left(2 + \tfrac{2k-1-r}{\gamma} \right)^{\frac{5}{2} + \frac{2k-1-r}{\gamma}}, \\
\Gamma\big(2k +  \tfrac{1+r-2k}{\gamma} \big) &\sim  \sqrt{2\pi}\cdot e^{1-2k-\frac{1+r-2k}{\gamma}} \left( 2k - 1+ \tfrac{1+r-2k}{\gamma}  \right)^{2k - \frac{1}{2}+ \frac{1+r-2k}{\gamma}}.
\end{align*}
Combining these shows that
\begin{align}
a_k &\sim \frac{1}{e k^{2k+1}} \cdot \left(2 + \tfrac{2k-1-r}{\gamma} \right)^{\frac{5}{2} + \frac{2k-1-r}{\gamma}}  \left( 2k - 1+ \tfrac{1+r-2k}{\gamma}  \right)^{2k - \frac{1}{2}+ \frac{1+r-2k}{\gamma}} \notag\\
&= \tfrac{k}{e} \cdot \left(\tfrac{2}{\gamma} + \tfrac{2\gamma-1-r}{\gamma k} \right)^{\frac{5}{2} + \frac{2k-1-r}{\gamma}}  \left( 2 - \tfrac{2}{\gamma} + \tfrac{1+r-\gamma}{\gamma k}  \right)^{2k - \frac{1}{2}+ \frac{1+r-2k}{\gamma}} \notag \\
&\sim \tfrac{k}{e} \cdot \left( \tfrac{2}{\gamma} \right)^{\frac{5}{2}-\frac{1+r}{\gamma}} \left(\tfrac{2}{\gamma} + \tfrac{2\gamma-1-r}{\gamma k} \right)^\frac{2k}{\gamma} \left(2-\tfrac{2}{\gamma} \right)^{\frac{1+r}{\gamma} -\frac{1}{2}}  \left( 2 - \tfrac{2}{\gamma} + \tfrac{1+r-\gamma}{\gamma k}  \right)^{2k - \frac{2k}{\gamma}} \notag \\
&= \tfrac{k}{e} \cdot \left( \tfrac{2}{\gamma} \right)^2 (\gamma-1)^{\frac{1+r}{\gamma} -\frac{1}{2}} \left( \tfrac{2}{\gamma} \right)^\frac{2k}{\gamma} \left(1 + \left(\tfrac{2\gamma-1-r}{\gamma} \right) \tfrac{\gamma}{2k}  \right)^{\frac{2k}{\gamma}} \notag \\
& \qquad \qquad \qquad \qquad \qquad \qquad \cdot \left(2 - \tfrac{2}{\gamma} \right)^{2k - \frac{2k}{\gamma}} \left( 1 + \left( \tfrac{1+r-\gamma}{\gamma}\right) \left( 2k - \tfrac{2k}{\gamma} \right)^{-1} \right)^{2k - \frac{2k}{\gamma}}.  \label{E:ak-calculation-1}
\end{align}
Since $\lim_{t \to \infty} \left(1+ \tfrac{c}{t}\right)^{t} = e^c$, we see from \eqref{E:ak-calculation-1} that
\begin{align}
a_k &\sim \tfrac{k}{e} \cdot \left( \tfrac{2}{\gamma} \right)^2 (\gamma-1)^{\frac{1+r}{\gamma} -\frac{1}{2}} \left( \tfrac{2}{\gamma} \right)^\frac{2k}{\gamma} e^{\frac{2\gamma-1-r}{\gamma}} \left(2 - \tfrac{2}{\gamma} \right)^{2k - \frac{2k}{\gamma}} e^{\frac{1+r-\gamma}{\gamma}} \notag\\
&= 4\gamma^{-2} (\gamma-1)^{\frac{1+r}{\gamma} -\frac{1}{2}} k \left( 4 \gamma^{-2} (\gamma-1)^{2-\frac{2}{\gamma}} \right)^k. \label{E:ak-asymptotic-formula}
\end{align}
It is easy to see that $\lim_{k \to \infty} \sqrt[k]{a_k} = 4 \gamma^{-2} (\gamma-1)^{2-\frac{2}{\gamma}}$, so the root test guarantees the convergence of $\sum_{k=0}^\infty a_k x^k$ on the interval
\begin{equation}\label{E:interval-of-convergence-ak-series}
0 \le x < \frac{\gamma^2}{4(\gamma-1)^{2-\frac{2}{\gamma}}}.
\end{equation}
Also from \eqref{E:ak-asymptotic-formula}, it is immediate that $\sum_{k=0}^\infty a_k x^k$ diverges at the right endpoint of \eqref{E:interval-of-convergence-ak-series}, since the sum $\sum_{k=0}^\infty k$ diverges.

These observations about $\sum a_k x^k$ yield immediate analogues for $\Theta_r(1,iq)$.  We now see that, for any $r \in \cj'_0$, formula \eqref{E:AppendixIntegral4} converges for $q>1$.  Indeed,
\begin{equation*}
\Theta_r(1,iq) = q^{\frac{1+r}{\gamma}-3} F_r(q),
\end{equation*}
where
\begin{equation}\label{E:def-of-appendix-function-Fr}
F_r(q) = \frac{8\pi^2}{\gamma (\gamma-1)^{\frac{1+r}{\gamma}}} \sum_{k=0}^\infty \frac{\Gamma\big(3 + \frac{2k-1-r}{\gamma} \big) \Gamma\big(2k +  \frac{1+r-2k}{\gamma} \big) }{\Gamma\big(k+1\big)^2} \left( \frac{\gamma^2}{4 (\gamma-1)^{2-\frac{2}{\gamma}} q^{\frac{2}{\gamma}}} \right)^{k}.
\end{equation}
It is now clear that $F_r$ is positive, strictly decreasing and real-analytic on $(1,\infty)$.  Furthermore, $\lim_{q\to1^+} F_r(q) = \infty$.  It is also clear that 
\begin{equation}\label{value-of-F_r-at-infty)}
\lim_{q\to\infty} F_r(q) =  \frac{8\pi^2 \Gamma\big(3-\frac{1+r}{\gamma}\big)\Gamma\big(\frac{1+r}{\gamma}\big)}{\gamma (\gamma-1)^{\frac{1+r}{\gamma}}}.
\end{equation}
This concludes the proof of Proposition \ref{P:Theta_r(1,iq)-calculation}.
\end{proof}

Equation \eqref{E:Theta_r_normalized_z1=/=0} in conjunction with Proposition \ref{P:Theta_r(1,iq)-calculation} gives a useful description of $\Theta_r(z)$ when $z_1 \neq 0$.  It still remains to consider the $z_1 = 0$ case, which by \eqref{E:Theta_r_z1=0} amounts to understanding $\Theta_r(0,i\epsilon)$.  Observe that 
\begin{align*}
\ell((0,i\epsilon),(\zeta_1,\zeta_2)) &=  \frac{4}{(\gamma|\zeta_1|^\gamma + \epsilon + is_\zeta)^2} =\frac{4}{((\gamma-1)\alpha_\zeta^\gamma + \epsilon + i s_\zeta)^2},
\end{align*}
and therefore
\begin{align*}
\Theta_r(0,i\epsilon)& = 16\int_{M_\gamma} \frac{1}{|(\gamma-1)\alpha_\zeta^\gamma + \epsilon + i s_\zeta|^4}\,\mu_r(\zeta) \\
&= 32\pi \int_{0}^\infty \alpha_\zeta^r \int_{-\infty}^\infty \frac{1}{|(\gamma-1)\alpha_\zeta^\gamma + \epsilon + i s_\zeta|^4} ds_{\zeta} d\alpha_\zeta.
\end{align*}
The inner  integral is calculated by a special case of \eqref{E:appendix-s-integral-formula}, while the outer integral is a special case of \eqref{E:appendix-alpha-integral-formula}.  We conclude that  
\begin{equation}\label{E:value-of-Theta_r(0,i_epsilon)}
\Theta_r(0,i\epsilon) =  \frac{8\pi^2 \Gamma\big(3-\frac{1+r}{\gamma}\big)\Gamma\big(\frac{1+r}{\gamma}\big)}{\gamma (\gamma-1)^{\frac{1+r}{\gamma}}}\, \epsilon^{\frac{1+r}{\gamma}-3}.
\end{equation}

These two results can be incorporated into a single statement.  Indeed, combining \eqref{E:Theta_r-written-in-alpha-epsilon-variables} with Proposition \ref{P:Theta_r(1,iq)-calculation} shows that if $|z_1| = \alpha \neq 0$, $r \in \cj'_0$,
\begin{align}
\Theta_r(z) = \Theta_r(\alpha, \theta, s, \epsilon) = \alpha^{r+1-3\gamma} \Theta_r(1, i q(\alpha,\epsilon)) &= \alpha^{r+1-3\gamma} q(\alpha,\epsilon)^{\frac{r+1}{\gamma}-3} F_r(q(\alpha,\epsilon)) \notag \\
&= (\alpha^\gamma+\epsilon)^{\frac{r+1}{\gamma}-3} F_r(q(\alpha,\epsilon)). \label{E:Theta_r-formula-alpha-epsilon}
\end{align}
Since $z \in M_\gamma^\epsilon$ for some $\epsilon > 0$, $q(0,\epsilon) = \infty$.  Equation \eqref{value-of-F_r-at-infty)} lets us define $F_r(\infty) := \lim_{q \to \infty} F_r(q)$.  Comparing \eqref{E:value-of-Theta_r(0,i_epsilon)} and \eqref{E:Theta_r-formula-alpha-epsilon} side by side shows that the latter now encompasses the former.

\begin{remark}
If $\gamma = 2$ and $r = 1$ (i.e. the Heisenberg group equipped with Euclidean measure on its parameter space $\R^3$), the expression for $\Theta_r(z)$ in \eqref{E:Theta_r-formula-alpha-epsilon} is independent of $\alpha$.  This is an instance of a phenomenon occurring on the $\cs_\beta$ hypersurfaces defined in \eqref{D:def-of-S_beta}.  Refer to the Appendix in \cite{BarEdh17} for more information. $\lozenge$
\end{remark}

\subsection{Proof of Theorem \ref{T:range-of-r-building-holomorphic-full-Leray}}

Let $\mathscr{B}$ denote the set of all subsets $K \subset \Omega_\gamma$ satisfying the following two properties:
\begin{enumerate}
\item $\alpha_K := \sup\{ \alpha: (\alpha,\theta,s,\epsilon) \in K \} < \infty$.
\item $\epsilon_K := \inf\{\epsilon: (\alpha,\theta,s,\epsilon) \in K \} > 0$.
\end{enumerate}
Given a fixed $K \in \sB$, observe that the infimum of $q(\alpha,\epsilon)$ when restricted to $K$ remains strictly greater than $1$. Indeed, if we define
\begin{equation}\label{E:def-q_K}
q_K := \inf\{q(\alpha,\epsilon) : (\alpha,\theta,s,\epsilon) \in K \},
\end{equation}
it is clear that $q_K \ge 1+ \tfrac{\epsilon_K}{\alpha_K^\gamma} > 1$.  $\sB$ contains, for instance, all compact subsets of $\Omega_\gamma$; it also contains unbounded subsets of $\Omega_\gamma$, as it places no restriction on $s$.

Fix some $r \in \cj_0 = (-1-\gamma, 2\gamma-1)$.  This determines a unique $r'$ value with $\frac{r+r'}{2} = \gamma-1$, and it is clear that $r' \in  \cj'_0 = (-1,3\gamma - 1)$.  Now observe that if $K \in \sB$, then $\sup_{z\in K} \Theta_{r'}(z)$ is bounded from above.  Indeed, for $z \in K$, \eqref{E:Theta_r-formula-alpha-epsilon} shows
\begin{equation}\label{E:upper-bound-on-Theta_{r'}}
 \Theta_{r'}(z) = (\alpha^\gamma+\epsilon)^{\frac{r'+1}{\gamma}-3} F_{r'}(q(\alpha,\epsilon)) \le \epsilon_K^{\frac{r'+1}{\gamma}-3} F_{r'}(q_K).
\end{equation}
Proposition \ref{P:Theta_r(1,iq)-calculation} guarantees that this is finite since $q_K > 1$.

If $f \in L^2(M_\gamma,\mu_r)$, $r\in \cj_0$, we will show that $\bm{L}f \in \co(\Omega_\gamma)$ by proving it is holomorphic on each $K\in \sB$.  First assume $f \in L^2(M_\gamma,\mu_r)$ is compactly supported.  Then $\bm{L}f$ is seen to be holomorphic by differentiating under the integral sign in \eqref{E:Leray-transform-appendix}.  For a general $f \in L^2(M_\gamma,\mu_r)$, choose a sequence of compactly supported $f_j$ tending to $f$ in $L^2(M_\gamma,\mu_r)$.
\begin{align}
\sup_{z \in K} \left|\bm{L}f_j(z) - \bm{L}f(z)\right| &= \frac{\gamma^2}{16 \pi^2 } \sup_{z \in K}  \left| \int_{M_\gamma} \ell(z,\zeta)(f_j(\zeta) - f(\zeta)) \sigma(\zeta) \right| \notag \\
&\le \frac{\gamma^2}{16 \pi^2 } \sup_{z \in K}  \left( \int_{M_\gamma} |\ell(z,\zeta)|^2 \mu_{r'}(\zeta) \right) \left( \int_{M_\gamma} |f_j(\zeta) - f(\zeta)|^2 \mu_r(\zeta) \right) \notag \\
&= C_{r',K} \norm{f_j - f}_{\mu_r}^2, \notag
\end{align}
where the constant only depends on $r'$ and $K$.  This shows that $\bm{L}f_j \to \bm{L}f$ uniformly on $K$, implying $\bm{L}f \in \co(K)$, being a uniform limit of holomorphic functions. Since $K \in \sB$ was arbitrary, we conclude that $\bm{L}f$ is holomorphic on all of $\Omega_\gamma$. \qed

\begin{remark}
Calculations seen throughout this appendix can be easily adapted to the sub-Leray operator $\bm{L}_k$.  Suppose that $r' \in \cj'_k = (- 2k(\gamma-1) -1  , 3\gamma - 1 +2k)$, the interval defined in \eqref{def-of-interval-J'_k}.  If $\frac{r+r'}{2} = \gamma-1$, it holds that $r \in (-\gamma-1-2k, 2\gamma -1 + 2k(\gamma-1)) := \cj_k$.
It can be verified that if $r \in \cj_k$ and $f \in L^2(M_\gamma,\mu_r)$, then $\bm{L}_k f \in \co(\Omega_\gamma)$.  Recall now that the interval of $r$ values for which $\bm{L}_k$ is bounded from $L^2(M_\gamma,\mu_r) \to L^2(M_\gamma,\mu_r)$ was given in \eqref{D:def-interval-Ik} by $\ci_k = (-2k-1 , (2k+2)(\gamma-1)+1)$.  It is immediate that $\ci_k \subset \cj_k$.  $\lozenge$
\end{remark}

\subsection{Two residue integral computations}\label{SS:appendix-integral-formulas}

\begin{proposition}\label{P:AppendixSIntegral}
For $C>0$ and $k$ a non-negative integer,
\begin{equation*}
\int_{-\infty}^\infty \frac{ds}{|C+is|^{2k+4}} = \frac{\pi}{4^{k+1}C^{2k+3}} \cdot \frac{\Gamma(2k+3)}{\Gamma(k+2)^2}.
\end{equation*}
\end{proposition}
\begin{proof}
Write
\begin{equation*}
\int_{-\infty}^\infty \frac{ds}{|C+is|^{2k+4}} = \int_{-\infty}^\infty \frac{ds}{(s^2 + C^2)^{k+2}}
\end{equation*}

Let $\mathcal{P}_R$ denote a counter-clockwise oriented, closed half-circle contour with radius $R$ lying in the upper-half plane with base lying on the interval $[-R,R]$.  If we set $g(z) = {(z^2+C^2)^{-k-2}}$, the residue theorem says (provided $R>C$),
\begin{equation}\label{E:AppendixResidueThm1}
2\pi i \,\textrm{Res}(g,iC) = \int_{\mathcal{P}_R} g(z)\,dz.
\end{equation}
We calculate
\begin{align*}
\textrm{Res}(g,iC) &= \frac{1}{\Gamma(k+2)} \frac{d^{k+1}}{dz^{k+1}}\left[ \frac{1}{(z+iC)^{k+2}}\right] \Bigg|_{z=iC} \\
&= \frac{1}{i 2^{2k+3}C^{2k+3}} \cdot \frac{\Gamma(2k+3)}{\Gamma(k+2)^2}.
\end{align*}

Now send $R\to\infty$ and note that the circular portion of the contour integral in \eqref{E:AppendixResidueThm1} tends to $0$.  This gives the result.
\end{proof}

\begin{proposition}\label{P:AppendixAlphaIntegral}
Let $q>0$, $\gamma>1$, $r \in \cj'_k = (- 2k(\gamma-1) -1  , 3\gamma - 1 +2k)$ and $k$ a non-negative integer.  Then
\begin{equation*}
\int_0^\infty \frac{\alpha^{2k(\gamma-1)+r}}{((\gamma-1)\alpha^\gamma + q)^{2k+3}} d\alpha =  \frac{\Gamma\left(3 + \frac{2k-1-r}{\gamma} \right) \Gamma\left(2k +  \frac{1+r-2k}{\gamma} \right) }{ \gamma (\gamma-1)^{\frac{1+r}{\gamma}} q^{3-\frac{1+r}{\gamma}} \Gamma\big(2k+3\big)  } \left( q (\gamma-1)^{\gamma-1}  \right)^{-\frac{2k}{\gamma}}.
\end{equation*}
\end{proposition}
\begin{proof}

Making the change of variable $x = \alpha^\gamma$, see that
\begin{align}\label{E:AppendixAlphaIntComp1}
\int_0^\infty  \frac{\alpha^{2k(\gamma-1)+r}}{((\gamma-1)\alpha^\gamma + q)^{2k+3}}\,d\alpha &= \frac{1}{\gamma(\gamma-1)^{2k+3}} \int_0^\infty \frac{x^{2k-1 + \frac{r+1-2k}{\gamma}}}{(x+E)^{2k+3}}\,dx,
\end{align}
with $E = \frac{q}{\gamma-1}$.  Now define the function
\begin{equation}
h(z) = \frac{z^{2k-1 + \frac{r+1-2k}{\gamma}}}{(z+E)^{2k+3}},
\end{equation}
which is multi-valued whenever the numerator exponent is a non-integer.  Restrict attention to a single-valued meromorphic branch of $h$ defined on $\C\backslash [0,\infty)$.  Note that $h$ has a pole of order $2k+3$ at $z = -E$.

For $0< \delta <R$, define the closed, positively oriented contour $\mathcal{P}_{\delta,R}$ by traversing the following paths in sequence:
\begin{itemize}
\item[$I$.] A line segment on the positive $x$-axis moving {\em right} from $\delta$ to $R$.
\item[$II$.] A circle of radius $R$ centered at the origin, starting on the positive $x$-axis and traveling {\em counterclockwise}.
\item[$III$.] A line segment on the positive $x$-axis moving {\em left} from $R$ to $\delta$.
\item[$IV$.] A circle of radius $\delta$ centered at the origin, starting on the positive $x$-axis and traveling {\em clockwise}.
\end{itemize}
As long as $\delta < E < R$, the residue theorem says
\begin{align}
2\pi i \, \textrm{Res}(h,-E) &= \int_{\mathcal{P}_{\delta,R}} h(z)\,dz \notag\\
&= \int_{I} h(z)\,dz + \int_{II} h(z)\,dz + \int_{III} h(z)\,dz + \int_{IV} h(z)\,dz. \label{E:AppendixAlphaIntComp2}
\end{align}
Keeping in mind the multi-valued nature of $h$, combine the $I$ and $III$ integrals:
\begin{align}
\int_{I} h(z)\,dz +  \int_{III} h(z)\,dz &= \int_\delta^R h(x)\,dx + \int^\delta_R h(x e^{2\pi i})\,dx \notag \\
&=  \int_\delta^R \frac{x^{2k-1 + \frac{r+1-2k}{\gamma}}}{(x+E)^{2k+3}} \,dx +  \int^\delta_R \frac{ x^{2k-1 + \frac{r+1-2k}{\gamma}} e^{2\pi i(2k-1 + \frac{r+1-2k}{\gamma})} }{(x e^{2\pi i}+E)^{2k+3}}\,dx \notag \\
&= \left(1 - e^{\frac{2\pi i (r+1-2k)}{\gamma}}\right) \int_\delta^R \frac{ x^{2k-1 + \frac{r+1-2k}{\gamma}} }{(x+E)^{2k+3}}\, dx. \label{E:AppendixAlphaInt_I+III}
\end{align}

Standard estimates show that since $r \in \cj'_k$,
\begin{equation}\label{E:AppendixAlphaIntLimit1}
\lim_{R \to \infty} \int_{II} h(z)\,dz = 0 = \lim_{\delta \to 0} \int_{IV} h(z)\,dz.
\end{equation}

Combining \eqref{E:AppendixAlphaIntLimit1} with \eqref{E:AppendixAlphaIntComp2} and \eqref{E:AppendixAlphaInt_I+III} shows
\begin{equation}\label{E:AppendixAlphaIntComp3}
2\pi i \, \textrm{Res}(h,-E) =  \left(1 - e^{\frac{2\pi i (r+1-2k)}{\gamma}}\right) \int_0^\infty \frac{ x^{2k-1 + \frac{r+1-2k}{\gamma}} }{(x+E)^{2k+3}}\, dx.
\end{equation}
Now calculate this residue.
\begin{align}
 \textrm{Res}(h,-E) &= \frac{1}{\Gamma(2k + 3)} \frac{d^{2k+2}}{dz^{2k+2}}\left( z^{2k-1 + \frac{r+1-2k}{\gamma}} \right) \Bigg|_{z=-E} \notag \\
 &= \frac{-\Gamma\big({2k+\frac{r+1-2k}{\gamma}\big)}}{\Gamma\big(2k + 3\big) \Gamma\big(\frac{r+1-2k}{\gamma} - 2 \big)} E^{\frac{r+1-2k}{\gamma}-3} e^{\frac{\pi i(r+1-2k)}{\gamma}}. \label{E:AppendizAlphaIntResidue}
\end{align}

Combining \eqref{E:AppendixAlphaIntComp3} and \eqref{E:AppendizAlphaIntResidue} shows that
\begin{align}
\int_0^\infty \frac{x^{2k(1-\frac{1}{\gamma})}}{(x+E)^{2k+3}}\, dx &= \frac{-2\pi i}{ 1 - e^{\frac{2\pi i (r+1-2k)}{\gamma}} } \cdot \frac{\Gamma\big({2k+\frac{r+1-2k}{\gamma}\big)}}{\Gamma\big(2k + 3\big) \Gamma\big(\frac{r+1-2k}{\gamma} - 2 \big)} E^{\frac{r+1-2k}{\gamma}-3} e^{\frac{\pi i(r+1-2k)}{\gamma}} \notag \\
&= \pi \csc\left(\pi\big(\tfrac{r+1-2k}{\gamma}\big)\right) \frac{\Gamma\big({2k+\frac{r+1-2k}{\gamma}\big)}}{\Gamma\big(2k + 3\big) \Gamma\big(\frac{r+1-2k}{\gamma} - 2 \big)} E^{\frac{r+1-2k}{\gamma}-3}  \notag \\
&=  \frac{\Gamma\big(\frac{r+1-2k}{\gamma}\big) \Gamma\big(1-\frac{r+1-2k}{\gamma}\big) \Gamma\big(2k+\frac{r+1-2k}{\gamma}\big)}{\Gamma\big(2k + 3\big) \Gamma\big(\frac{r+1-2k}{\gamma} - 2 \big)} E^{\frac{r+1-2k}{\gamma}-3} \notag \\
&= \big(\tfrac{r+1-2k}{\gamma} - 1\big) \big(\tfrac{r+1-2k}{\gamma}- 2 \big) \frac{\Gamma\big(1-\frac{r+1-2k}{\gamma}\big) \Gamma\big(2k+\frac{r+1-2k}{\gamma}\big)}{\Gamma\big(2k + 3\big)} E^{\frac{r+1-2k}{\gamma}-3} \notag \\
&=  \frac{\Gamma\big(3+\frac{2k-r-1}{\gamma}\big) \Gamma\big(2k+\frac{r+1-2k}{\gamma}\big)}{\Gamma\big(2k + 3\big)} E^{\frac{r+1-2k}{\gamma}-3}. \label{E:AppendixAlphaIntComp4}
\end{align}
Combining \eqref{E:AppendixAlphaIntComp1} and \eqref{E:AppendixAlphaIntComp4} gives the result.
\end{proof}


\bibliographystyle{acm}
\bibliography{BarEdh21}

\begin{thebibliography}{10}

\bibitem{ScandBook04}
{\sc Andersson, M., Passare, M., and Sigurdsson, R.}
\newblock {\em Complex convexity and analytic functionals}, vol.~225 of {\em
  Progress in Mathematics}.
\newblock Birkh\"auser Verlag, Basel, 2004.

\bibitem{AndAskRoyBOOK99}
{\sc Andrews, G.~E., Askey, R., and Roy, R.}
\newblock {\em Special functions}, vol.~71 of {\em Encyclopedia of Mathematics
  and its Applications}.
\newblock Cambridge University Press, Cambridge, 1999.

\bibitem{AxlBouRam}
{\sc Axler, S., Bourdon, P., and Ramey, W.}
\newblock {\em Harmonic function theory}, second~ed., vol.~137 of {\em Graduate
  Texts in Mathematics}.
\newblock Springer-Verlag, New York, 2001.

\bibitem{Bar16}
{\sc Barrett, D.~E.}
\newblock Holomorphic projection and duality for domains in complex projective
  space.
\newblock {\em Trans. Amer. Math. Soc. 368}, 2 (2016), 827--850.

\bibitem{BarBol07}
{\sc Barrett, D.~E., and Bolt, M.}
\newblock Cauchy integrals and {M}\"obius geometry of curves.
\newblock {\em Asian J. Math. 11}, 1 (2007), 47--53.

\bibitem{BarEdh17}
{\sc Barrett, D.~E., and Edholm, L.~D.}
\newblock The {L}eray transform: {F}actorization, dual {CR} structures, and
  model hypersurfaces in {CP}2.
\newblock {\em Adv. Math. 364\/} (2020), 107012.

\bibitem{BarGru18}
{\sc Barrett, D.~E., and Grundmeier, D.~E.}
\newblock Sums of {CR} functions from competing {CR} structures.
\newblock {\em Pacific J. Math. 293}, 2 (2018), 257--275.

\bibitem{BarGru20}
{\sc Barrett, D.~E., and Grundmeier, D.~E.}
\newblock Projective-umbilic points of circular real hypersurfaces in
  {$\Bbb{C}^2$}.
\newblock {\em Proc. Amer. Math. Soc. 148}, 12 (2020), 5241--5248.

\bibitem{BarLan09}
{\sc Barrett, D.~E., and Lanzani, L.}
\newblock The spectrum of the {L}eray transform for convex {R}einhardt domains
  in {$\mathbb{C}^2$}.
\newblock {\em J. Funct. Anal. 257}, 9 (2009), 2780--2819.

\bibitem{Bell16}
{\sc Bell, S.~R.}
\newblock {\em The {C}auchy transform, potential theory and conformal mapping},
  second~ed.
\newblock Chapman \& Hall/CRC, Boca Raton, FL, 2016.

\bibitem{Bol05}
{\sc Bolt, M.}
\newblock A geometric characterization: complex ellipsoids and the
  {B}ochner-{M}artinelli kernel.
\newblock {\em Illinois J. Math. 49}, 3 (2005), 811--826.

\bibitem{Bol07}
{\sc Bolt, M.}
\newblock A lower estimate for the norm of the {K}erzman-{S}tein operator.
\newblock {\em J. Integral Equations Appl. 19}, 4 (2007), 453--463.

\bibitem{Bol06}
{\sc Bolt, M.}
\newblock Spectrum of the {K}erzman-{S}tein operator for the ellipse.
\newblock {\em Integral Equations Operator Theory 57}, 2 (2007), 167--184.

\bibitem{Bol08}
{\sc Bolt, M.}
\newblock The {M}\"obius geometry of hypersurfaces.
\newblock {\em Michigan Math. J. 56}, 3 (2008), 603--622.

\bibitem{Bol10}
{\sc Bolt, M.}
\newblock The {M}\"obius geometry of hypersurfaces, {II}.
\newblock {\em Michigan Math. J. 59}, 3 (2010), 695--715.

\bibitem{BolRai15}
{\sc Bolt, M., and Raich, A.}
\newblock The {K}erzman-{S}tein operator for piecewise continuously
  differentiable regions.
\newblock {\em Complex Var. Elliptic Equ. 60}, 4 (2015), 478--492.

\bibitem{CowMacBook95}
{\sc Cowen, C.~C., and MacCluer, B.~D.}
\newblock {\em Composition operators on spaces of analytic functions}.
\newblock Studies in Advanced Mathematics. CRC Press, Boca Raton, FL, 1995.

\bibitem{Diaz87}
{\sc Diaz, K.~P.}
\newblock The {S}zego kernel as a singular integral kernel on a family of
  weakly pseudoconvex domains.
\newblock {\em Trans. Amer. Math. Soc. 304}, 1 (1987), 141--170.

\bibitem{EdhShe22}
{\sc Edholm, L.~D., and Shelah, Y.}
\newblock The {L}eray transform on rigid {H}artogs hypersurfaces.
\newblock {\em (In preparation)\/} (2022).

\bibitem{Fef79}
{\sc Fefferman, C.}
\newblock Parabolic invariant theory in complex analysis.
\newblock {\em Adv. in Math. 31}, 2 (1979), 131--262.

\bibitem{GKPBook}
{\sc Graham, R.~L., Knuth, D.~E., and Patashnik, O.}
\newblock {\em Concrete mathematics}, second~ed.
\newblock Addison-Wesley Publishing Company, Reading, MA, 1994.
\newblock A foundation for computer science.

\bibitem{GreSte78}
{\sc Greiner, P.~C., and Stein, E.~M.}
\newblock On the solvability of some differential operators of type {$
  \square_{b}$}.
\newblock In {\em Several complex variables ({C}ortona, 1976/1977)}. Scuola
  Norm. Sup. Pisa, Pisa, 1978, pp.~106--165.

\bibitem{Hirachi90}
{\sc Hirachi, K.}
\newblock Transformation law for the {S}zego projectors on {CR} manifolds.
\newblock {\em Osaka J. Math. 27}, 2 (1990), 301--308.

\bibitem{KerSte78a}
{\sc Kerzman, N., and Stein, E.~M.}
\newblock The {C}auchy kernel, the {S}zeg{\H o} kernel, and the {R}iemann
  mapping function.
\newblock {\em Math. Ann. 236}, 1 (1978), 85--93.

\bibitem{KerSte78b}
{\sc Kerzman, N., and Stein, E.~M.}
\newblock The {S}zeg{\H o} kernel in terms of {C}auchy-{F}antappi\`e kernels.
\newblock {\em Duke Math. J. 45}, 2 (1978), 197--224.

\bibitem{LanSte13}
{\sc Lanzani, L., and Stein, E.~M.}
\newblock Cauchy-type integrals in several complex variables.
\newblock {\em Bull. Math. Sci. 3}, 2 (2013), 241--285.

\bibitem{LanSte14}
{\sc Lanzani, L., and Stein, E.~M.}
\newblock The {C}auchy integral in {$\mathbb{C}^n$} for domains with minimal
  smoothness.
\newblock {\em Adv. Math. 264\/} (2014), 776--830.

\bibitem{LanSte17a}
{\sc Lanzani, L., and Stein, E.~M.}
\newblock The {C}auchy-{S}zeg{\H o} projection for domains in $\mathbb{C}^n$
  with minimal smoothness.
\newblock {\em Duke Math. J. 166}, 1 (2017), 125--176.

\bibitem{LanSte17c}
{\sc Lanzani, L., and Stein, E.~M.}
\newblock The role of an integration identity in the analysis of the
  {C}auchy-{L}eray transform.
\newblock {\em Sci. China Math. 60}, 11 (2017), 1923--1936.

\bibitem{LanSte17b}
{\sc Lanzani, L., and Stein, E.~M.}
\newblock The {C}auchy-{L}eray integral: counterexamples to the {$L^p$}-theory.
\newblock {\em Indiana Univ. Math. J. 68}, 5 (2019), 1609--1621.

\bibitem{Lefevre09}
{\sc Lef\`evre, P.}
\newblock Generalized essential norm of weighted composition operators on some
  uniform algebras of analytic functions.
\newblock {\em Integral Equations Operator Theory 63}, 4 (2009), 557--569.

\bibitem{Meyer82}
{\sc Meyer, Y.}
\newblock {\em Solution Des Conjectures De Calderon}.
\newblock Universidad Autonoma de Madrid, Madrid, Spain, 1982.

\bibitem{Range86}
{\sc Range, R.~M.}
\newblock {\em Holomorphic functions and integral representations in several
  complex variables}, vol.~108 of {\em Graduate Texts in Mathematics}.
\newblock Springer-Verlag, New York, 1986.

\bibitem{ShelahThesis}
{\sc Shelah, Y.}
\newblock A spectral exploration of the {Leray} transform in two different
  settings in {$\mathbb{C}^2$}.
\newblock {\em Ph.D. Thesis\/} (2021).

\bibitem{SimonBookOpTheory}
{\sc Simon, B.}
\newblock {\em Operator theory}.
\newblock A Comprehensive Course in Analysis, Part 4. American Mathematical
  Society, Providence, RI, 2015.

\end{thebibliography}
 
\end{document}